\documentclass[11pt,reqno]{amsart}

\usepackage{geometry}
\usepackage{amsmath,amsthm,amssymb,amsfonts}

\usepackage[utf8]{inputenc} 
\usepackage[T1]{fontenc}    
\usepackage[hidelinks]{hyperref} 
\usepackage[capitalise]{cleveref}

\usepackage{comment}
\usepackage{subcaption}
\usepackage{wrapfig}
\usepackage{bm}
\usepackage{url}            
\usepackage{booktabs}       
\usepackage{nicefrac}       
\usepackage{microtype}      

\usepackage{bbm}
\usepackage{cite}
\usepackage{psfrag}
\usepackage{cite}
\usepackage{color,soul}
\usepackage{breakcites}     
\usepackage{bigints}

\usepackage{enumitem}
\setlist{leftmargin=1.6em}

\newtheorem{definition}{Definition}
\newtheorem{example}{Example}

\newtheorem{remark}{Remark}
\newtheorem{theorem}{Theorem}
\newtheorem{proposition}{Proposition}
\newtheorem{lemma}{Lemma} 
\newtheorem{corollary}{Corollary}
\usepackage{graphicx}
\usepackage{xcolor}

\usepackage{algorithm}
\usepackage[noend]{algpseudocode}

\makeatletter
\def\subsubsection{\@startsection{subsubsection}{3}%
  \z@{.5\linespacing\@plus.7\linespacing}{-.5em}%
  {\normalfont\bfseries}}
\makeatother

\newcommand{\EOT}{\mathsf{S}_c^{\mspace{1mu}\epsilon}}

\newcommand{\Cov}[0]{\mathrm{Cov}}
\newcommand{\Var}[0]{\mathrm{Var}}

\newcommand{\sign}[0]{\mathrm{sign}}

\newcommand{\ind}[0]{\mathbbm{1}}

\newcommand{\calC}[0]{\mathcal{C}}
\newcommand{\calF}[0]{\mathcal{F}}

\newcommand{\calP}[0]{\mathcal{P}}
\newcommand{\calS}[0]{\mathcal{S}}
\newcommand{\calT}[0]{\mathcal{T}}
\newcommand{\calX}[0]{\mathcal{X}}

\newcommand{\calY}[0]{\mathcal{Y}}

\newcommand{\supp}{\mathrm{spt}}
\newcommand{\E}[0]{\mathbb{E}}

\newcommand{\R}[0]{\mathbb{R}}

\newcommand{\Prob}[0]{\mathbb{P}}

\newcommand{\Hmt}[1]{\dot H^{-1,2}\left( #1 \right)}
\newcommand{\Ht}[1]{\dot H^{1,2}\left( #1 \right)}
\newcommand{\dx}[1]{{d}#1}

\newcommand{\vertiii}[1]{{\left\vert\kern-0.25ex\left\vert\kern-0.25ex\left\vert #1 
    \right\vert\kern-0.25ex\right\vert\kern-0.25ex\right\vert}}

\newcommand{\sS}{\mathsf{S}}
\newcommand{\cA}{\mathcal{A}}
\newcommand{\cB}{\mathcal{B}}

\newcommand{\cF}{\mathcal{F}}

\newcommand{\cP}{\mathcal{P}}

\newcommand{\cX}{\mathcal{X}}
\newcommand{\cY}{\mathcal{Y}}

\newcommand{\EE}{\mathbb{E}}

\newcommand{\NN}{\mathbb{N}}

\newcommand{\RR}{\mathbb{R}}

\newcommand{\vasti}{\bBigg@{3.5 }}
\newcommand{\vast}{\bBigg@{4}}
\newcommand{\Vast}{\bBigg@{5}}
\newcommand{\Vastt}{\bBigg@{7}}

\makeatother

\newcommand{\be}{\begin{equation}}
\newcommand{\ee}{\end{equation}}
\newcommand{\ba}{\begin{align}}
\newcommand{\ea}{\end{align}}
\newcommand{\baa}{\begin{align*}}
\newcommand{\eaa}{\end{align*}}

\newcommand{\abs}[1]{\left|#1\right|}

\newcommand{\wass}{\mathsf{W}_1}
\newcommand{\Wp}{\mathsf{W}_p}
\newcommand{\GWp}{\mathsf{W}_p^{(\gamma_\sigma)}}
\newcommand{\GWo}{\mathsf{W}_1^{(\gamma_\sigma)}}

\newcommand{\Lip}{\mathsf{Lip}_{1,0}}
\newcommand{\ip}[1]{\left\langle#1\right\rangle}

\newcommand{\SWp}{\underline{\mathsf{W}}_p}
\newcommand{\SWone}{\underline{\mathsf{W}}_1}

\newcommand{\MSWp}{\overline{\mathsf{W}}_p}
\newcommand{\MSWone}{\overline{\mathsf{W}}_1}

\newcommand{\unitsph}{\mathbb{S}^{d-1}}

\newcommand{\ptheta}{\mathfrak{p}^\theta}

\newcommand{\dconv}{\stackrel{d}{\to}}

\newcommand{\kk}[1] {{\color{black!30!green}KK: [#1]}}

\DeclareMathOperator{\interior}{int}
\DeclareMathOperator{\diam}{diam}
\DeclareMathOperator{\inte}{int}

\DeclareMathOperator{\esssup}{esssup}

\begin{document}

\title[Statistical inference with regularized optimal transport]{Statistical inference with\\regularized optimal transport}

\thanks{
Z. Goldfeld is supported by the NSF CRII grant CCF-1947801, in part by the 2020 IBM Academic Award, and in part by the NSF CAREER Award CCF-2046018. 
K. Kato is partially supported by the  NSF grants DMS-1952306 and DMS-2014636.}

\date{First version: May 6, 2022. This version: \today}

\author[Z. Goldfeld]{Ziv Goldfeld}
\address[Z. Goldfeld]{
School of Electrical and Computer Engineering, Cornell University.
}
\email{goldfeld@cornell.edu}

\author[K. Kato]{Kengo Kato}
\address[K. Kato]{
Department of Statistics and Data Science, Cornell University.
}
\email{kk976@cornell.edu}

\author[G. Rioux]{Gabriel Rioux}

\address[G. Rioux]{
Center for Applied Mathematics, Cornell University.}
\email{ger84@cornell.edu}

\author[R. Sadhu]{Ritwik Sadhu}
\address[R. Sadhu]{
Department of Statistics and Data Science, Cornell University.
}
\email{rs2526@cornell.edu}

\begin{abstract}
Optimal transport (OT) is a versatile framework for comparing probability measures, with many applications to statistics, machine learning, and applied mathematics. However, OT distances suffer from computational and statistical scalability issues to high dimensions, which motivated the study of regularized OT methods like slicing, smoothing, and entropic penalty. This work establishes a unified framework for deriving limit distributions of empirical regularized OT distances, semiparametric efficiency of the plug-in empirical estimator, and bootstrap consistency. We apply the unified framework to provide a comprehensive statistical treatment of: (i) average- and max-sliced $p$-Wasserstein distances, for which several gaps in existing literature are closed; (ii) smooth distances with compactly supported kernels, the analysis of which is motivated by computational considerations; and (iii) entropic OT, for which our method generalizes existing limit distribution results and establishes, for the first time, efficiency and bootstrap consistency. While our focus is on these three regularized OT distances as applications, the flexibility of the proposed framework renders it applicable to broad classes of functionals beyond these examples.
\end{abstract}



\maketitle

\section{Introduction}
Optimal transport (OT) theory \cite{villani2008optimal,santambrogio2010} provides a versatile framework for comparing probability distributions. Originally introduced by Monge \cite{mongeOT1781} and later formulated by Kantorovich  \cite{kantorovich1942translocation}, the OT problem between two Borel probability measures $\mu,\nu$ on $\RR^d$ is defined by
\begin{equation}
    \mathsf{T}_c(\mu,\nu):=\inf_{\pi\in\Pi(\mu,\nu)}\int_{\RR^d \times \RR^d} c(x,y)d\pi(x,y),\label{EQ:Kantorovich_OT}
\end{equation}
where $\Pi(\mu,\nu)$ is the set of couplings between $\mu$ and $\nu$. The special case of the $p$-Wasserstein distance for $p \in [1,\infty)$ is given by $\mathsf{W}_p(\mu,\nu):=\left(\mathsf{T}_{\|\cdot\|^p}(\mu,\nu)\right)^{1/p}$. Thanks to an array of favorable properties, including the Wasserstein metric structure, a convenient duality theory, robustness to support mismatch, and the rich geometry induced on the space of probability measures, OT and the Wasserstein distance have seen a surge of applications in statistics, machine learning, and applied mathematics. These include generative modeling \cite{arjovsky_wgan_2017,gulrajani2017improved,tolstikhin2018wasserstein,bernton2019parameter,chen2021inferential}, robust/adversarial machine learning (ML) \cite{blanchet2019quantifying,wong2019wasserstein}, domain adaptation \cite{solomon2015convolutional,courty2016optimal}, image recognition \cite{rubner2000earth,sandler2011nonnegative,li2013novel}, vector quantile regression \cite{carlier2016vector,chernozhukov2017monge,ghosal2019multivariate,hallin2021quantile}, 
Bayesian estimation \cite{bernton2019approximate}, and causal inference \cite{torous2021optimal}. 
Unfortunately, OT distances are generally hard to compute and suffer from the curse of dimensionality in empirical estimation, whereby the number of samples needed for reliable estimation grows exponentially with dimension.

These deficits have motivated the introduction of regularized OT methods that aim to alleviate the said computational and statistical bottlenecks. Three prominent regularizations are: (1) slicing via lower-dimensional projections \cite{rabin2011wasserstein,bonnotte2013unidimensional,nadjahi2019asymptotic,bayraktar2021,nadjahi2020statistical}; (2) smoothing via convolution with a chosen kernel \cite{Goldfeld2020convergence,Goldfeld2020GOT,Goldfeld2020limit_fdiv,Goldfeld2020limit_wass,nietert21,sadhu21limit,zhang2021convergence,chen2021asymptotics,han2021nonparametric,goldfeld22limit,block2022smooth}; and (3) convexification via entropic penalty \cite{schrodinger1931uber,leonard2014survey,cuturi2013sinkhorn,altschuler2017near,genevay2019sample,mena2019statistical,delbarrio22EOT}. These techniques preserve many properties of classic OT but avoid the curse of dimensionality, which enables a scalable statistical theory. As reviewed below\footnote{We postpone the literature review on each regularization method to its respective section.}, much effort was devoted to exploring dimension-free empirical convergence rates and limit distributions, bootstrapping, and other statistical aspects of regularized OT, although several notable gaps in the literature remain. Furthermore, proof techniques for such results are typically on a case-by-case basis and do not follow a unified approach, despite evident similarities between~the three regularization methods as complexity reduction techniques of the classic OT framework. 

The present paper develops a unified framework for deriving limit distributions, semiparametric efficiency bounds, and bootstrap consistency for a broad class of functionals that, in particular, encompasses the empirical regularized OT distances mentioned above (\cref{sec: unified}). As example applications of the general framework, we explore a comprehensive treatment of the following problems:
\begin{itemize}[leftmargin=3.5mm]
    \item \textbf{Average- and max-sliced $\bm{\Wp}$ (\cref{sec: slicing}):} Our limit distribution theory closes existing gaps in the literature (e.g., a limit distribution result for sliced $\wass$ was assumed in  \cite{nadjahi2020statistical} but left unproven), with the efficiency and bootstrap consistency results providing additional constituents for valid statistical inference. 
    \item \textbf{Smooth $\bm{\Wp}$ with compactly supported kernels (\cref{sec: smooth Wp}):} Gaussian-smoothed OT was previously shown to preserve the classic Wasserstein structure while alleviating the curse of dimensionality. Motivated by computational considerations, herein we study smoothing with compactly supported kernels
    . We explore the metric, topological, and statistical aspects previously derived under Gaussian smoothing, and then discuss computation by lifting the algorithm from \cite{vacher2021dimension} to the smooth distance with a compactly supported kernel. 
    \item \textbf{Entropic OT (\cref{sec: EOT}):} A central limit theorem (CLT) for empirical entropic OT (EOT) was derived \cite{mena2019statistical,delbarrio22EOT} for independent data via a markedly different proof technique than proposed herein. Revisiting this problem using our general machinery, we rederive this CLT allowing for dependent data, and also obtain new results on semiparametric efficiency and bootstrap consistency.
\end{itemize}

The unified limit distribution framework, stated in Proposition \ref{prop: master proposition}, relies on the extended functional delta method for Hadamard directionally differentiable functionals \cite{shapiro1990,romisch2004}. To match the delta method with the regularized OT setup, we focus on a functional on a space of probability measures that is (a) locally Lipschitz with respect to (w.r.t.) the sup-norm for a Donsker function class and (b) G\^ateaux directionally differentiable at the population distribution. To apply this framework, we seek to: (i) set up the regularized distance as a locally Lipschitz functional $\delta$ w.r.t. $\|\cdot\|_{\infty,\cF} = \sup_{f \in \calF}| \cdot |$; (ii) show $\cF$ to be Donsker to obtain convergence of the empirical process in $\ell^\infty(\cF)$; (iii) characterize the G\^{a}teaux directional  derivative of $\delta$ at $\mu$. For each regularized distance (sliced, smooth, and entropic), we identify the appropriate function class $\cF$ and establish the desired Lipschitz continuity and differentiability, relying on OT duality theory. Regularization enforces the dual potentials to possess smoothness or low-dimensionality properties, which are leveraged to show that $\cF$ is Donsker. It is worth noting that our framework does not require independent and identically distributed (i.i.d.) data and can be applied for any estimate (not only the empirical distribution) of the population distribution, so long as the uniform limit theorem mentioned in (ii) holds true.

As the general framework stems from the extended functional delta method, the limiting variable of the (scaled and centered) empirical regularized distance is given by the directional derivative of $\delta$ at the population distribution. Linearity of the derivative implies that the limit variable is centered Gaussian. In this case, it is natural to ask whether the empirical distance attains the semiparametric efficiency lower bound (cf.  \cite[Chapter 25]{vanderVaart1998asymptotic}). Semiparametric efficiency bounds serve as analogs of Cram\'{e}r-Rao lower bounds in semiparametric estimation and account for the fundamental difficulty of estimating functionals of interest. We show that the asymptotic variance of the empirical distance indeed agrees with the semiparametric efficiency bound, relative to a certain tangent space. Still, even when the limiting variable is Gaussian, direct analytic estimation of the asymptotic variance may be nontrivial. 
To account for that, we explore bootstrap consistency for empirical regularized OT distances. 
Altogether, the limit distribution theory, semiparametric efficiency, and bootstrap consistency provide a comprehensive statistical account of the considered regularized OT distances.

A unifying approach of a similar flavor to ours, but for classic OT distances, was proposed in \cite{hundrieser2022unifying}. Focusing solely on the supremum functional, they used the extended functional delta method to derive limit distributions for classic $\Wp$, with $p\geq 2$, for compactly supported distributions under the alternative in dimensions $d\leq 3$. In comparison, our approach is more general and can treat any functional that adheres to the aforementioned local Lipschitz continuity and differentiability. This is crucial for analyzing regularized OT distances as some instances do not amount to a supremum functional. For instance, average-sliced Wasserstein distances correspond to mixed $L^1$-$L^\infty$ functionals, which are  not accounted for by the setup from \cite{hundrieser2022unifying}. The functional delta method was also used in \cite{sommerfeld2018,tameling2019} to derive limit distributions for OT between discrete population distributions by parametrizing them using simplex vectors. This result was extended to semi-discrete OT in \cite{del2022central} by exploiting the fact that complexity of the optimal potentials class is reduced when one of the measures is supported on a discrete set. Another recent application can be found in \cite{goldfeld22limit}, where this approach was leveraged for Gaussian-smoothed $\Wp$ by embedding the domain of the Wasserstein distance into a certain dual Sobolev space. 

The rest of the paper is organized as follows. \cref{sec: background} presents notation used throughout the paper and  necessary background on Wasserstein distances and the extended functional delta method. \cref{sec: unified} presents a unified framework for deriving limit distributions, bootstrap consistency, and semiparametric efficiency bounds for regularized OT distances. The general tools developed in \cref{sec: unified} will be applied to sliced Wasserstein distances in \cref{sec: slicing}, smooth Wasserstein distances with compactly supported kernels in \cref{sec: smooth Wp}, and EOT in  \cref{sec: EOT}. \cref{sec: summary} leaves some concluding remarks. Proofs for the results in Sections \ref{sec: background}--\ref{sec: EOT} can be found in Appendices \ref{sec: proof master proposition}--\ref{sec: EOT proof}.

\section{Background and Preliminaries}\label{sec: background}

This section collects notation used throughout the paper and sets up necessary background on Wasserstein distances and the extended functional delta method.

\subsection{Notation}\label{subsec: notation}
On a finite dimensional Euclidean space, $\| \cdot \|$ denotes the standard Euclidean norm. 
The unit sphere in $\RR^d$ is denoted by $\unitsph = \{ x \in \R^d : \|x\|=1\}$, while $B(x,r)$ denotes the open ball with center $x \in \R^d$ and radius $r>0$. For a subset $A$ of a topological space~$S$, let $\overline{A}^S$ denote the closure of $A$; if the space $S$ is clear from the context, then we simply write $\overline{A}$ for the closure. The space of Borel probability measures on $S$ is denoted by $\calP(S)$. 
When $S$ is a normed space with norm $\| \cdot \|$, we denote $\calP_p(S):= \{ \mu \in \calP(S) : \int \| x \|^p d\mu(x) < \infty \}$ for $1 \le p <\infty$. The (topological) support of $\mu \in\calP(S)$, denoted as $\supp(\mu)$, is defined by the set of $x \in S$ for which every open neighborhood of $x$ has positive $\mu$-measure.
For $\mu \in \calP(S)$ and a measurable map $f$ from $S$ into another measurable space, the pushforward of $\mu$ under $f$ is denoted as $f_{\sharp}\mu = \mu \circ f^{-1}$, i.e., if $X \sim \mu$ then $f(X) \sim f_{\sharp}\mu$. For any finite signed Borel measure $\gamma$ on $S$, we identify $\gamma$ with the linear functional $f \mapsto \gamma (f) = \int f d\gamma$. 
For $\mu \in \calP(S)$ and a $\mu$-integrable function $h$ on $S$, $h\mu$ denotes the signed measure $h d\mu$.

For given $\mu,\nu \in \calP(\R^d)$, we use $(X_1,Y_1),(X_2,Y_2),\dots$ to designate the coordinate projections of the product probability space $\prod_{i=1}^{\infty} \big(\R^{2d}, \cB(\R^{2d}),\mu \otimes \nu\big)$. To generate auxiliary random variables, we extend the probability space as $(\Omega, \cA, \Prob) = \left [ \prod_{i=1}^{\infty} \big(\R^{2d}, \cB(\R^{2d}),\mu \otimes \nu\big)\right] \times \big([0,1],\cB([0,1]), \mathrm{Leb}\big)$, where $\mathrm{Leb}$ is the Lebesgue measure on $[0,1]$. 
Let $\stackrel{w}{\to}, \stackrel{d}{\to}$, and $\stackrel{\Prob}{\to}$ denote weak convergence of probability measures, convergence in distribution of random variables, and convergence in probability, respectively. When necessary, convergence in distribution is understood in the sense of Hoffmann-J{\o}rgensen (cf. Chapter 1 in \cite{vanderVaart1996}). 
Throughout the paper, we frequently consider the empirical and bootstrap measures, which are defined as follows. Given a probability measure $\mu\in\cP(S)$, we denote the empirical measure of i.i.d. observations $X_1,\ldots,X_n$ from $\mu$ by $\hat{\mu}_n:=n^{-1}\sum_{i=1}^n\delta_{X_i}$. Given such data $X_1,\dots,X_n$, let $X_1^B,\dots,X_n^B$ be an independent sample from $\hat{\mu}_n$, and set $\hat{\mu}_n^B:= n^{-1}\sum_{i=1}^n \delta_{X_i^B}$ as the bootstrap empirical distribution. We use $\Prob^B$ to denote the conditional probability given the data.

We write $N(\epsilon, \cF, d)$ for the $\epsilon$-covering number of a function class $\cF$ w.r.t. a pseudometric $d$, and $N_{[\,]}(\epsilon, \cF, d)$ for the corresponding bracketing number. 
For two functions $f$ and $g$ on $\calX$ and $\calY$, respectively, let $f \oplus g$ be a function on $\calX \times \calY$ defined by $(f\oplus g) (x,y) = f(x) + g(y)$ for $(x,y) \in \calX \times \calY$. 
For any nonempty set $S$, let $\ell^\infty (S)$ the space of bounded real functions on $S$ equipped with the sup-norm $\| \cdot \|_{\infty,S} = \sup_{x \in S}| \cdot |$. The space $(\ell^\infty (S),\| \cdot \|_{\infty,S})$ is a Banach space. For any measure space $(S,\calS,\mu)$ and $1 \le p < \infty$, let $L^p(\mu) = L^p (S,\calS,\mu)$ denote the space of measurable functions $f: S \to \R$ with $\| f \|_{L^p(\mu)} = (\int |f|^p d\mu)^{1/p} < \infty$. The space $(L^p(\mu),\| \cdot \|_{L^p(\mu)})$ is a Banach space, and if $\mu$ is $\sigma$-finite and $\calS$ is countably generated, then the space is separable.  For $\beta \in (0,2]$, let $\psi_{\beta}(t) = e^{t^\beta}-1$ for $t \ge 0$, and recall that the corresponding Orlicz (quasi-)norm of a real-valued random variable $\xi$ is defined as $\| \xi \|_{\psi_{\beta}}:= \inf \{ C>0 : \E[\psi_{\beta}(|\xi|/C)] \le 1 \}$. We call $\mu\in\cP(\R^d)$ \textit{$\beta$-sub-Weibull} if $\| \| X \| \|_{\psi_{\beta}} < \infty$ for $X \sim \mu$. 
We say that $\mu$ is \textit{sub-Gaussian} if it is $2$-sub-Weibull. 
For two numbers $a$ and $b$, we use the notation $a \wedge b = \min \{ a,b \}$ and $a \vee b = \max \{a,b \}$.

\subsection{Wasserstein distances}\label{subsec: Wp}

The Wasserstein distance is a specific instance of the OT problem from \eqref{EQ:Kantorovich_OT}, defined as follows.
\begin{definition}[Wasserstein distance]
Let $1 \le p < \infty$. The $p$-th Wasserstein distance between $\mu,\nu \in \calP_p(\R^d)$ is defined as 
\begin{equation}\label{eq:OT-problem}
\Wp(\mu,\nu):= \inf_{\pi \in \Pi(\mu,\nu)} \left [ \int_{\R^d \times \R^d} \|x-y\|^p \, d \pi(x,y) \right]^{1/p},
\end{equation}
where  $\Pi(\mu,\nu)$ is the set of couplings of $\mu$ and $\nu$.
\end{definition}

The $p$-Wasserstein distance is a metric on $\calP_p(\R^d)$ and metrizes weak convergence plus convergence of $p$th moments, i.e., $\Wp(\mu_n,\mu) \to 0$ if and only if $\mu_n \stackrel{w}{\to} \mu$ and $\int \| x \|^p d\mu_n(x) \to \int \|x\|^p d\mu(x)$. 

Wasserstein distances adhere to the following dual form (cf. \cite[Theorem 5.9]{villani2008optimal} or \cite[Theorem 6.1.5]{ambrosio2005}):
\begin{equation}
\label{eq: duality}
\Wp^p(\mu,\nu) = \sup_{\varphi \in L^1(\mu)} \left [ \int_{\R^d} \varphi d\mu + \int_{\R^d} \varphi^c d\nu \right ],
\end{equation}
where $\varphi^c(y) = \inf_{x \in \R^d}\big [ \|x-y\|^p -\varphi(x) \big]$ is the $c$-transform of $\varphi$ (for the cost $c(x,y)=\|x-y\|^p$). A function $f:\R^d \to [-\infty,\infty)$ is called \textit{c-concave} if $f=g^c$ for some function $g: \R^d \to [-\infty,\infty)$. There is at least one $c$-concave  $\varphi \in L^1(\mu)$ that attains the supremum in~\eqref{eq: duality}, and we call this $\varphi$ an \textit{OT potential} from $\mu$ to $\nu$ for $\Wp$. Further, when $1 < p < \infty$ and $\mu$ is supported on a connected set with negligible boundary and has a (Lebesgue) density, then the OT potential from $\mu$ to $\nu$ is unique on $\inte(\supp(\mu))$ up to additive constants \cite[Corollary 2.7]{delbarrio2021}. Various smoothness properties of the potentials can be established under appropriate regularity conditions on the cost and $\mu,\nu$---a fact that we shall leverage in our derivations. 

\begin{remark}[Literature review on $\Wp$ limit distribution theory]
Distributional limits of $\sqrt{n}\big(\Wp^p(\hat{\mu}_n,\nu)-\Wp^p(\mu,\nu)\big)$ and its two-sample analogue for discrete $\mu,\nu$ under both the null $\mu=\nu$ and the alternative $\mu \ne \nu$ were derived in \cite{sommerfeld2018,tameling2019}. Similar results for~general distributions are known only in the one-dimensional case. Specifically,  for $p=1,2$, \cite{del1999central,delbarrio2005} leverage the representations of $\Wp$ in $d=1$ as the $L^p$ norm between distribution functions ($p=1$) and quantile functions ($p=2$) to derive distributional limits under the null. Limit distributions in $d=1$ for $p \ge 2$ under the alternative ($\mu\neq\nu$) were derived in \cite{delbarrio2019}. 
In arbitrary dimension, \cite{delbarrio2019central} establish~asymptotic normality of $\sqrt{n}\big(\mathsf{W}_2^2(\hat{\mu}_n,\nu)-\EE\big[\mathsf{W}_2^2(\hat{\mu}_n,\nu)\big]\big)$ under the alternative $\mu\ne\nu$. This~was extended to general transportation costs satisfying certain regularity conditions in \cite{delbarrio2021}. The main limitation of these results is the centering around the expected empirical distance (and not the population one), which does not enable performing inference for $\Wp$. This gap was addressed in \cite{manole2021plugin}, where a CLT for $\sqrt{n}\big(\mathsf{W}_2^2(\tilde{\mu}_n,\nu)-\mathsf{W}_2^2(\mu,\nu)\big)$ was established, but for a wavelet-based estimator $\tilde{\mu}_n$ of $\mu$ (as opposed to the empirical distribution), while assuming several technical conditions on the Lebesgue densities of $\mu,\nu$. As mentioned in the introduction, \cite{hundrieser2022unifying} leverage the extended functional delta method for the supremum functional to obtain limit distributions for $\Wp$, with $p\geq 2$, for compactly supported distributions under the alternative in dimensions $d\leq 3$. 
\end{remark}

\subsection{Extended functional delta method}\label{subsec: functional delta method}

Our unified framework for deriving limit distributions of empirical regularized OT distances relies on the extended functional delta method, which we set up next. Let $\mathfrak{D},\mathfrak{E}$ be normed spaces and $\phi: \Theta \subset \mathfrak{D} \to \mathfrak{E}$ be a map. Following \cite{shapiro1990,romisch2004}, we say that $\phi$ is 
\textit{Hadamard directionally differentiable}  at $\theta \in \Theta$ if there exists a map $\phi'_{\theta}: \calT_{\Theta}(\theta) \to \mathfrak{E}$ such that 
\begin{equation}
\lim_{n \to \infty} \frac{\phi(\theta + t_n h_n) - \phi(\theta)}{t_n} = \phi'_{\theta}(h)
\label{eq: Hadamard}
\end{equation}
for any $h \in \calT_{\Theta}(\theta)$, $t_n \downarrow 0$, and $h_n \to h$ in $\mathfrak{D}$ such that $\theta + t_n h_n \in \Theta$. Here $\calT_{\Theta}(\theta)$ is the \textit{tangent cone} to $\Theta$ at $\theta$ defined as 
\[
\calT_{\Theta}(\theta):= \left \{ h \in \mathfrak{D}:\, h = \lim_{n \to \infty} \frac{\theta_n - \theta}{t_n}  \ \text{for some $\theta_n \to \theta$ in $\Theta$ and $t_n \downarrow 0$} \right \}. 
\]
The tangent cone $\calT_{\Theta}(\theta)$ is closed, and if $\Theta$ is convex, then $\calT_{\Theta}(\theta)$ coincides with the closure in $\mathfrak{D}$ of $\{ (\vartheta - \theta)/t:\, \vartheta \in \Theta, t > 0 \}$. The derivative $\phi_{\theta}'$ is positively homogeneous and  continuous but need not be linear.

\begin{lemma}[Extended functional delta method \cite{shapiro1991,dumbgen1993,romisch2004,fang2019}]
\label{lem: functional delta method}
Let $\mathfrak{D},\mathfrak{E}$ be normed spaces and $\phi: \Theta \subset \mathfrak{D} \to \mathfrak{E}$ be a map that is Hadamard directionally differentiable at $\theta \in \Theta$ with derivative $\phi_{\theta}': \calT_{\Theta}(\theta) \to \mathfrak{E}$. Let $T_n: \Omega \to \Theta$ be maps such that $r_n (T_n - \theta) \stackrel{d}{\to} T$ for some $r_n \to \infty$ and Borel measurable map $T: \Omega \to \mathfrak{D}$ with values in $\calT_{\Theta}(\theta)$. Then, $r_n \big(\phi(T_n) - \phi(\theta)\big) \stackrel{d}{\to} \phi_{\theta}'(T)$. Further, if $\Theta$ is convex, then we have $r_n \big(\phi(T_n) - \phi(\theta)\big) - \phi_{\theta}'\big(r_n(T_n-\theta)\big) \to 0$ in outer probability. 
\end{lemma}

Lemma \ref{lem: functional delta method} is at the core of our framework for deriving limit distributions. It is termed the ``extended'' functional delta method as it extends the (classical) functional delta method for Hadamard differentiable maps to directionally differentiable ones.  

\medskip
While  Hadamard directional differentiability is sufficient to derive limit distributions, bootstrap consistency often requires (full) Hadamard differentiability. Recall that the map $\phi$ is \textit{Hadamard differentiable} at $\theta$ \textit{tangentially to} a vector subspace $\mathfrak{D}_0 \subset \mathfrak{D}$ if there exists a continuous \textit{linear} map $\phi_\theta': \mathfrak{D}_0 \to \mathfrak{E}$ satisfying \eqref{eq: Hadamard} for any $h \in \mathfrak{D}_0$, $t_n \to 0 \ (t_n \ne 0)$, and  $h_n \to h$ in $\mathfrak{D}$ such that $\theta + t_n h_n \in \Theta$. 
The differences from Hadamard directional differentiability is that the derivative $\phi_\theta'$ must be linear and thus the domain must be a vector subspace of $\mathfrak{D}$, and the sequence $t_n \to 0$ must be a generic (nonzero) sequence converging to zero. The next lemma is useful for verifying Hadamard differentiability from the directional one. 

\begin{lemma}
\label{lem: H differentiability}
Let $\phi: \Theta \subset \mathfrak{D} \to \mathfrak{E}$ be Hadamard directionally differentiable at $\theta \in \Theta$ with derivative $\phi_{\theta}': \calT_{\Theta}(\theta) \to \mathfrak{E}$. If $\calT_{\Theta}(\theta)$ contains a subspace $\mathfrak{D}_0$ on which $\phi_\theta'$ is linear, then $\phi$ is Hadamard differentiable at $\theta$ tangentially to $\mathfrak{D}_0$.
\end{lemma}

\section{Unified Framework for Statistical Inference}\label{sec: unified}

This section develops a general framework for deriving limit distributions, bootstrap consistency, and semiparametric efficiency bounds for regularized OT distances. We first treat the former two aspects together, and then move on to discuss efficiency.  Throughout this section, $\mu_n$ designates an arbitrary random probability measure and not necessarily the empirical measure (unless explicitly stated otherwise). 
\subsection{Limit distributions and bootstrap consistency}

The following result is an adaptation of the extended functional delta method from \cref{lem: functional delta method} to the space of probability measures, which enables directly applying it to empirical regularized OT. 

\begin{proposition}[Limit distributions]
\label{prop: master proposition}
Consider the setting:
\begin{quote}
\textbf{(Setting $\bm{\circledast}$)} \
Let $\calF$ be a class of Borel measurable functions on a topological space $S$ with a finite envelope $F$. For a given $\mu \in \calP(S)$, let $\delta$ be a map from $\calP_0 \subset \calP(S)$ into a Banach space $(\mathfrak{E},\| \cdot \|_{\mathfrak{E}})$, where $\calP_0$ is a convex subset such that  $\mu \in \calP_0$ and  $\int F d\nu < \infty$ for all $\nu \in \calP_0$.
\end{quote}
Further suppose that
\begin{enumerate}[label=(\alph*)]
    \item $\mu_n: \Omega \to \calP_0$ are random probability measures with values in $\calP_0$ for all $n \in \NN$, such that there exists a tight random variable $G_\mu$ in $\ell^\infty (\calF)$ with $\sqrt{n}(\mu_n -\mu) \stackrel{d}{\to} G_{\mu}$ in $\ell^\infty (\calF)$; 
    \item $\delta$ is locally Lipschitz continuous at $\mu$ with respect to $\| \cdot \|_{\infty,\calF}$, in the sense that there exist constants $\epsilon > 0$ and $C < \infty$ such that 
\[\qquad\quad \| \nu -\mu \|_{\infty,\calF} \vee \| \nu' - \mu \|_{\infty,\calF}< \epsilon\ \ \implies \ \ \| \delta (\nu) - \delta(\nu') \|_{\mathfrak{E}} \le C\| \nu - \nu' \|_{\infty,\calF};
\]
\item For every $\nu \in \calP_0$, the mapping $t \mapsto \delta\big(\mu + t (\nu-\mu)\big)$ is right differentiable at $t=0$, and denote its right derivative by 
\begin{equation}
\delta_{\mu}'(\nu-\mu)=\lim_{t \downarrow 0} \frac{\delta\big(\mu+t(\nu-\mu)\big)-\delta(\mu)}{t}. 
\label{eq: right derivative}
\end{equation}
\end{enumerate}
Then (i) $\delta_{\mu}'$ uniquely extends to a continuous, positively homogeneous map on the tangent cone of $\calP_0$ at $\mu$:
\[
\calT_{\calP_0}(\mu):=\overline{\big\{ t(\nu-\mu):\,\nu \in \calP_0, t > 0 \big\}}^{\ell^\infty (\calF)};
\] 
(ii) $G_{\mu} \in T_{\calP_0}(\mu)$ almost surely (a.s.); and
(iii) $\sqrt{n}\big(\delta (\mu_n) - \delta(\mu)\big) -\delta_{\mu}'\big(\sqrt{n}(\mu_n-\mu)\big) \to 0$ holds
in outer probability. Consequently, we have the following convergence in distribution $\sqrt{n}\big(\delta (\mu_n) - \delta(\mu)\big)  \stackrel{d}{\to} \delta_{\mu}'(G_\mu)$. 
\end{proposition}

The proof first identifies $\delta$ as a map defined on a subset of $\ell^\infty (\calF)$. Formally, let $\tau: \calP_0 \ni \nu \mapsto (f \mapsto \nu (f)) \in \ell^\infty (\calF)$, and we identify $\delta$ with $\bar{\delta}: \tau \calP_{0,\epsilon} \to \mathfrak{E}$~defined by $\bar{\delta}(\tau \nu) =\delta(\nu)$, where $\calP_{0,\epsilon} = \{ \nu \in \calP_0 : \| \nu - \mu \|_{\infty,\calF} < \epsilon \}$.  The local Lipschitz condition (b) guarantees that the map $\bar{\delta}$ is well-defined (indeed, without the local Lipschitz condition, $\bar\delta$ may not be well-defined as $\tau$ may fail to be one-to-one). With this identification, we apply the extended functional delta method, Lemma \ref{lem: functional delta method}, by establishing  Hadamard directional differentiability of $\delta$ at $\mu$. The latter essentially follows by local Lipschitz continuity (condition (b)) and  G\^{a}teaux directional differentiability (condition (c)). Since the derivative $\delta_\mu'$ is a priori defined only on $\calP_{0,\epsilon} - \mu$, we need to extend the derivative to the tangent cone $\calT_{\calP_0}(\mu)$, for which we need completeness of the space $\mathfrak{E}$; see the proof in \cref{sec: proof master proposition} for details.

\medskip
For i.i.d. data $X_1,\dots,X_n \sim \mu$ and $\mu_n=\hat{\mu}_n$ as the empirical measure, to apply \cref{prop: master proposition} we will: (i) find a $\mu$-Donsker function class $\calF$ such that the functional $\delta$ is locally Lipschitz w.r.t. $\| \cdot \|_{\infty,\calF}$ at $\mu$; and (ii) find the G\^{a}teaux  directional derivative~\eqref{eq: right derivative}. In our applications to regularized OT, such a function class $\calF$ will be chosen to contain dual potentials corresponding to a proper class of distributions. Regularization enforces dual potentials to possess certain smoothness or low-dimensionality properties, which will guarantee that $\calF$ is indeed $\mu$-Donsker. The dual OT formulation also plays a crucial role in finding the  G\^{a}teaux directional derivative \eqref{eq: right derivative}.

\begin{remark}[Relaxed condition]
When $\delta (\mu_n)$ is well-defined, the condition that $\mu_n$ takes values in $\calP_0$ can be relaxed to $\mu_n \in \calP_0$ with inner probability approaching one.
\end{remark}

\begin{remark}[Data generating process]
\cref{prop: master proposition} does not impose any dependence conditions on the data. In particular, \cref{prop: master proposition} can be applied to dependent data as long as one can verify the uniform limit theorem in Condition (a). See, e.g., \cite{levental1988uniform,andrews1994introduction,arcones1994central,doukhan1995invariance,bae1995uniform,nishiyama2000weak,davezies2021empirical} on uniform CLTs for dependent data. 
\end{remark}

\begin{remark}[Convexity of $\calP_0$]
\label{rem: convexity}
The assumption that $\calP_0$ is convex can be replaced with the condition that $\calP_0$ is convex as a subset of $\ell^\infty (\calF)$. Namely, using the mapping $\tau: \calP_0 \ni \nu \mapsto (f \mapsto \nu (f)) \in \ell^\infty (\calF)$, we only need that $\tau \calP_0 = \{ \tau \nu : \nu \in \calP_0 \} \subset \ell^\infty (\calF)$ is convex. Condition (c) then should read that $t \mapsto \bar{\delta}((1-t)\tau \mu + t\tau\nu)$ is differentiable from the right at $t =0$ with derivative $\delta_\mu'(\mu-\nu)  = \lim_{t \downarrow 0} t^{-1} \big\{\bar{\delta}((1-t)\tau \mu + t\tau\nu) - \bar{\delta}(\tau\mu) \big\}$, where $\bar{\delta}(\tau \nu) =\delta(\nu)$. This modification is needed to cover the two-sample setting; see, e.g., the proof of \cref{thm: limit distribution SWp} Part (ii). 
\end{remark}

\subsubsection{Bootstrap consistency} In applications of \cref{prop: master proposition}, the obtained limit distribution is often non-pivotal in the sense that it depends on the population distribution $\mu$, which is unknown in practice. To circumvent the difficulty of estimating the distribution of $\delta_\mu'(G_\mu)$ directly, one may apply the bootstrap. When $\calF$ is $\mu$-Donsker and $\mu_n=\hat{\mu}_n$ is the empirical distribution of i.i.d. data from $\mu$, then the bootstrap (applied to the functional $\delta$) is consistent for estimating the distribution of $\delta_\mu'(G_\mu)$ provided that the map $\nu \mapsto \delta (\nu)$ is Hadamard differentiable w.r.t. $\| \cdot \|_{\infty,\calF}$ at $\nu=\mu$ tangentially to a subspace of $\ell^\infty (\calF)$ that contains the support of $G_\mu$; cf. Theorem 23.9 in \cite{vanderVaart1998asymptotic} or Theorem 3.9.11 in \cite{van1996weak}. The following corollary is useful for invoking such theorems under the setting of Proposition \ref{prop: master proposition}.

\begin{corollary}[Bootstrap consistency via Hadamard differentiability]
\label{cor: master}
Consider the setting of Proposition \ref{prop: master proposition}. If, in addition,  $G_\mu$ is a mean-zero Gaussian variable in $\ell^\infty (\calF)$, then  $\supp(G_\mu)$ is a vector subspace of $\ell^\infty (\calF)$. If further $\delta_\mu'$ is linear on $\supp(G_\mu)$, then $\nu \mapsto \delta (\nu)$ is Hadamard differentiable w.r.t. $\| \cdot \|_{\infty,\calF}$ at $\nu=\mu$ tangentially to~$\supp(G_\mu)$.
\end{corollary}

In general, when the functional is  Hadamard directionally differentiable with a nonlinear derivative, the bootstrap fails to be consistent; cf. \cite{dumbgen1993,fang2019}. An alternative way to estimate the limit distribution in such cases is to use subsampling \cite{politis1994large} or a rescaled bootstrap \cite{dumbgen1993}.

\subsection{Semiparametric efficiency}

In Proposition \ref{prop: master proposition}, if $\delta_\mu'$ is linear and $G_\mu$ is mean-zero Gaussian, then the limit distribution $\delta_\mu'(G_\mu)$ is mean-zero Gaussian as well. In such cases, it is natural to ask if the  plug-in estimator $\delta(\mu_n)$ is asymptotically efficient in the sense of \cite[p.~367]{vanderVaart1998asymptotic}, relative to a certain tangent space. Informally, the semiparametric efficiency bound at $\mu$ is computed as the largest Cram\'{e}r-Rao lower bound among one-dimensional submodels passing through $\mu$. 

Formally, consider estimating a functional $\kappa: \calP \subset \calP(S) \to \R$ at $\mu \in \calP$ from i.i.d. data $X_1,\dots,X_n \sim \mu$. We consider submodels $\{ \mu_t : 0  \le  t < \epsilon' \}$ with $\mu_0 = \mu$ such that, for some measurable score function $h: S \to \R$, we have
\[
\int \left [ \frac{d\mu_t^{1/2} - d\mu^{1/2}}{t} - \frac{1}{2} hd\mu^{1/2} \right]^2 \to 0,
\]
where $d\mu_t$ and $d\mu$ are Radon-Nikodym densities w.r.t. a common dominating measure and the integration is taken w.r.t. the dominating measure. 
Score functions are square integrable w.r.t. $\mu$ and $\mu$-mean zero. A \textit{tangent set} $\dot \calP_\mu \subset L^2 (\mu)$ of the model $\calP$ at $\mu$ is the set of score functions corresponding to a collection of such submodels. If $\dot \calP_\mu$ is a vector subspace of $L^2(\mu)$, then it is called a \textit{tangent space}. Relative to a given tangent set $\dot \calP_\mu$, the functional $\kappa: \calP \to \R$ is called \textit{differentiable} at $\mu$ if there exists a continuous linear functional $\dot \kappa_\mu: L^2(\mu) \to \R$ such that, for every $h \in \dot \calP_\mu$ and a submodel $t \mapsto \mu_t$ with score function $h$,
\[
\frac{\kappa (\mu_t) - \kappa(\mu)}{t} \to \dot{\kappa}_\mu h, \quad t \downarrow 0.
\]

The semiparametric efficiency bound for estimating $\kappa$ at $\mu$, relative to $\dot \calP_\mu$, is defined as 
\[
\sigma_{\kappa,\mu}^2 = \sup_{h \in \mathrm{lin}(\dot \calP_\mu)} \frac{(\dot \kappa_\mu h)^2}{\| h \|_{L^2(\mu)}^2}, 
\]
where $\mathrm{lin}(\dot \calP_\mu)$ is the linear span of $\dot \calP_\mu$. In particular, the $N(0,\sigma_{\kappa,\mu}^2)$ distribution serves as the ``optimal'' limit distribution for estimating $\kappa$ at $\mu$ in the sense of the H\'ajek-Le Cam convolution theorem and also in the local asymptotic minimax sense; see Chapter 25 in \cite{vanderVaart1998asymptotic} for details. 

\medskip
The next proposition concerns the computation of the semiparametric efficiency bound in the context of \cref{prop: master proposition}.

\begin{proposition}[Semiparametric efficiency]
\label{prop: master proposition 2}
For Setting~$\circledast$ from \cref{prop: master proposition} with $\mathfrak{E}=\R$, consider estimating $\delta: \calP_0 \to \R$ at $\mu$ from i.i.d. data $X_1,\dots,X_n \sim \mu$. Set 
\[
\dot \calP_{0,\mu} =\{ h :\, \text{$h: S \to \R$ is bounded and measurable with  $\mu$-mean zero} \}.
\]
Suppose that (a) the function class $\calF$ is $\mu$-pre-Gaussian, i.e., there exists a tight mean-zero Gaussian process $G_\mu=\big(G_\mu(f)\big)_{f \in \calF}$ in $\ell^\infty (\calF)$ with covariance function $\Cov\big(G_\mu(f),G_\mu(g)\big)= \Cov_\mu (f,g)$; (b) for every $h \in \dot \calP_{0,\mu}$, $(1+th) \mu \in \calP_0$ for sufficiently small $t > 0$; and (c) there exists a continuous linear functional $\delta_{\mu}': \ell^\infty(\calF) \to \R$ such that \eqref{eq: right derivative} holds for every $\nu \in \calP_0$ of the form $\nu = (1+h)\mu$ for some $h \in \dot \calP_{0,\mu}$.
Then, the semiparametric efficiency bound for estimating $\delta$ at $\mu$ relative to the tangent space $\dot \calP_{0,\mu}$ agrees with $\Var\big(\delta_\mu'(G_\mu)\big)$. 
\end{proposition}

Proposition \ref{prop: master proposition 2} can be thought of as a variant of Theorem 3.1 in \cite{van1991efficiency}, which asserts that a Hadamard differentiable functional (tangentially to a sufficiently large subspace) of an asymptotically efficient estimator is again asymptotically efficient; see Remark \ref{rem: van der vaart} for more details. In  \cref{sec: proof master proposition}, we provide a direct and self-contained proof of  \cref{prop: master proposition 2}. We note that \cref{prop: master proposition 2} covers a slightly more general situation than \cite[Theorem 3.1]{van1991efficiency} since it only requires G\^{a}teaux differentiability of the map $\delta$, and choosing a pre-Gaussian function class $\calF$ in such a way that the derivative $\delta_\mu'$ extends to a continuous linear functional on $\ell^\infty (\calF)$. In particular, the efficiency bound computation in  \cref{prop: master proposition 2} is applicable even when \cref{prop: master proposition} is difficult to apply. For instance, when the G\^ateaux derivative $\delta_\mu'$ in (\ref{eq: right derivative}) is a point evaluation, $\delta_{\mu}'(\nu -\mu) = (\nu-\mu)(f^\star)$ for some function $f^\star \in L^2(\mu)$, we can choose $\calF = \{ f^\star \}$ (singleton) and apply \cref{prop: master proposition 2} to conclude that $\Var_\mu (f^\star)$ agrees with the semiparametric efficiency bound, relative to $\dot \calP_{0,\mu}$ (note that the function class $\calF$ in \cref{prop: master proposition 2} need not be the same as the one in \cref{prop: master proposition}).

The following corollary covers the two-sample case. Define $\dot \calP_{0,\nu}$  analogously to $\dot \calP_{0,\mu}$ and set $\dot \calP_{0,\mu} \oplus \dot \calP_{0,\nu} = \{ h_1 \oplus h_2 : h_1 \in \dot \calP_{0,\mu}, h_2 \in \dot \calP_{0,\nu} \}$. 
\begin{corollary}[Semiparametric efficiency in two-sample setting]
\label{cor: efficiency two sample}
Let $\calF$ be a class of Borel measurable functions on a topological space $S$ with finite envelope $F$, and for given $\mu,\nu \in \calP(S)$, let $\calP_{0,\mu},\calP_{0,\nu}$ be subsets of $\calP(S)$ containing $\mu,\nu$, respectively, such that $\int F d\rho < \infty$ for all $\rho \in \calP_{0,\mu} \cup \calP_{0,\nu}$. Let $\calP_{0,\mu} \otimes \calP_{0,\nu} = \{ \rho_1 \otimes \rho_2 : \rho_1 \in \calP_{0,\mu}, \rho_2 \in \calP_{0,\nu} \}$. Consider estimating $\delta: \calP_{0,\mu} \otimes \calP_{0,\nu}  \to \R$ at $\mu \otimes \nu$ from i.i.d. data $(X_1,Y_1),\dots,(X_n,Y_n) \sim \mu \otimes \nu$. Suppose that (a) the function class $\calF$ is pre-Gaussian w.r.t. $\mu$ and $\nu$; (b) for every $h_1 \otimes h_2 \in \dot \calP_{0,\mu} \oplus \dot \calP_{0,\nu}$, $\big((1+th_1)\mu\big) \otimes \big((1+th_2)\nu \big) \in \calP_{0,\mu} \otimes \calP_{0,\nu}$ for sufficiently small $t > 0$; (c) 
there exist continuous linear functionals $\delta_{\mu}': \ell^\infty(\calF) \to \R$ and $\delta_{\nu}':\ell^\infty(\calF) \to \R$ such that $t^{-1} \big \{\delta\big(\big((1+th_1)\mu\big) \otimes \big((1+th_2)\nu \big)\big) -  \delta(\mu \otimes \nu) \big\} \to \delta_\mu'(h_1\mu) + \delta_{\nu}'(h_2\nu)$ as $t \downarrow 0$ for every $h_1 \otimes h_2 \in \dot \calP_{0,\mu} \oplus \dot \calP_{0,\nu}$. Then, the semiparametric efficiency bound for estimating $\delta$ at $\mu \otimes \nu$ relative to the tangent space $\dot \calP_{0,\mu} \oplus \dot \calP_{0,\nu}$ agrees with $\Var\big(\delta_\mu'(G_\mu)\big) + \Var\big(\delta_\mu'(G_\nu)\big)$, where $G_\mu$ and $G_\nu$ are tight $\mu$- and $\nu$-Brownian bridges in $\ell^\infty (\calF)$, respectively. 
\end{corollary}

\begin{remark}[Efficiency of wavelet-based estimator of $\mathsf{W}_2$]
Theorem 18 in \cite{manole2021plugin} shows a CLT for a wavelet-based estimator $\mathsf{W}_2(\tilde{\mu}_n,\nu)$ for $\mathsf{W}_2(\mu,\nu)$ in the one-sample case under high-level assumptions that include global regularity of the OT potential $\varphi$: $\sqrt{n}\big(\mathsf{W}_2^2(\tilde{\mu}_n,\nu) - \mathsf{W}_2^2(\mu,\nu)\big) \stackrel{d}{\to} N\big(0,\Var_\mu(\varphi)\big)$. Their proof consists of first establishing a CLT for the expectation centering (similarly to \cite{delbarrio2019central}) and then showing that the bias is negligible. On the other hand, Proposition \ref{prop: master proposition} would be difficult to apply to obtain the same result, unless one assumes a \textit{uniform} bound on the H\"{o}lder norm of OT potentials corresponding to a local neighborhood of $\mu$. However, such uniform bounds on global regularity of OT potentials are currently unavailable, except for very limited cases (cf. the discussion after Theorem 3 in \cite{manole2021plugin}). Still, under the assumption of Theorem 18 in \cite{manole2021plugin}, it is readily verified that $\rho \mapsto \mathsf{W}_2(\rho,\nu)$ is G\^ateaux differentiable at $\mu$ with derivative $(\rho - \mu)(\varphi)$ (cf. the proof of Lemma \ref{lem: directional derivative SWp}), so $\Var_\mu(\varphi)$ indeed coincides with the semiparametric efficiency bound. 
\end{remark}

\section{Sliced Wasserstein distances}
\label{sec: slicing}

This section studies sliced Wasserstein distances. We start with some background and a brief literature review, and then move on to a statistical analysis of limit distributions, bootstrapping, and efficiency. 

\subsection{Background}

Average- and max-sliced Wasserstein distances are defined next.  

\begin{definition}[Sliced Wasserstein distances]
Let $1 \le p < \infty$. The average-sliced and max-sliced $p$-Wasserstein distances between $\mu,\nu \in \calP_p(\R^d)$ are defined, respectively,~as
\[
\SWp(\mu,\nu):= \left[\int_{\unitsph} \Wp^p(\ptheta_\sharp \mu, \ptheta_\sharp \nu) d\sigma(\theta)\right]^{1/p} \ \text{and} \ \  \MSWp(\mu,\nu):= \max_{\theta\in\unitsph} \Wp(\ptheta_\sharp \mu, \ptheta_\sharp \nu),
\]
where $\ptheta: \R^d \to \R$ is the projection map $x \mapsto \theta^{\intercal} x$ and $\sigma$ is the uniform distribution on the unit sphere $\unitsph$. 
\end{definition}

The sliced distances $\SWp$ and $\MSWp$ are metrics on $\calP_p(\R^d)$ and, in fact, induce the same topology as $\Wp$ \cite{bayraktar2021}.
\begin{lemma}[Proposition 2.1 and Theorem 2.1 of \cite{bayraktar2021}]
Both $\SWp$ and $\MSWp$ are metrics on $\calP_p(\R^d)$ that generate the same topology as $\Wp$, i.e., for any $\mu_n,\mu \in \calP_p(\R^d)$, 
\[
\SWp(\mu_n,\mu) \to 0  \iff \MSWp(\mu_n,\mu) \to 0  \iff \Wp(\mu_n,\mu) \to 0. 
\]
\end{lemma}

Sliced Wasserstein distances are efficiently~computable using the closed-form expression for $\Wp$ between distributions on $\R$ using quantile functions. For $\mu \in \calP(\R^d)$ and $\theta \in \unitsph$, denote by $F_{\mu}(\cdot\,;\theta)$ and $F_{\mu}^{-1}(\cdot\,;\theta)$ the distribution and quantile functions of $\ptheta_\sharp \mu$, respectively, i.e., 
\[
F_{\mu}(t;\theta)= \mu \big ( \{ x \in \R^d : \theta^{\intercal} x \le t \} \big ) \quad \text{and} \quad F_{\mu}^{-1}(\tau;\theta) = \inf \{ t\in\RR : F_{\mu}(t;\theta) \ge \tau \}. 
\]
Then, $\Wp(\ptheta_\sharp \mu, \ptheta_\sharp \nu)$ equals the $L^p$-norm between the corresponding quantile functions,
\[
\Wp^p(\ptheta_\sharp \mu, \ptheta_\sharp \nu) = \int_{0}^1 \big|F_{\mu}^{-1}(\tau;\theta) - F_{\nu}^{-1}(\tau;\theta)\big|^p d\tau.
\]
For $p=1$, the above further simplifies to the $L^1$ distance between the corresponding distribution functions, namely
\begin{equation}
\label{eq: W1 cdf}
\mathsf{W}_1 (\ptheta_\sharp \mu, \ptheta_\sharp \nu) = \int_{\R} \big|F_{\mu}(t;\theta) - F_{\nu}(t;\theta)\big| \, dt .
\end{equation}
Also, sliced Wasserstein distances between projected empirical distributions can be readily computed using order statistics. Let $\hat{\mu}_n:=n^{-1}\sum_{i=1}\delta_{X_i}$ and $\hat{\nu}_n:=n^{-1}\sum_{i=1}\delta_{Y_i}$ be the empirical distributions of samples $X_1,\ldots,X_n$ and $Y_1,\ldots,Y_n$. 
For each $\theta \in \unitsph$, let $X_i (\theta) = \theta^{\intercal}X_i$, and let $X_{(1)}(\theta) \le \dots \le X_{(n)}(\theta)$ be the order statistics. Define $Y_{(1)}(\theta) \le \cdots \le Y_{(n)}(\theta)$ analogously. Then, 
by Lemma 4.2 in \cite{bobkov2019one},
\begin{equation}
\Wp^p(\ptheta_\sharp \hat{\mu}_n, \ptheta_\sharp \hat{\nu}_n) = \frac{1}{n}\sum_{i=1}^n \big|X_{(i)}(\theta) - Y_{(i)}(\theta)\big|^p.\label{EQ:order_stat}
\end{equation}
The sliced distances $\SWp$ and $\MSWp$ can be computed by integrating or maximizing the above over $\theta\in\unitsph$.

\subsubsection{Literature review}

Sliced Wasserstein distances have been applied to various statistical inference and machine learning tasks, including barycenter computation \cite{rabin2011wasserstein}, generative modeling \cite{deshpande2018generative,deshpande2019max,nadjahi2019asymptotic,nadjahi2020statistical}, and autoencoders \cite{kolouri2018sliced}. The statistical literature on sliced distances mostly focused on expected value analysis. Specifically, \cite{niles2019estimation} show that if $\mu$ satisfies a $T_{q}(\sigma^2)$ inequality with $q \in [1,2]$, then
$\EE\big[\mspace{1mu}\MSWp(\hat{\mu}_n,\mu)\big] \lesssim \sigma \big( n^{-1/(2p)} + n^{(1/q-1/p)_+}\sqrt{(d\log n)/n}\, \big)$ up to a constant that depends only on $p$.
Further results on empirical convergence rates can be found in \cite{lin2021projection}, where both $\SWp$ and $\MSWp$ were treated, while replacing the transport inequality assumption of \cite{niles2019estimation} with exponential moment bounds (via Bernstein's tail conditions) or Poincar\'e type inequalities. A limit distribution result for one-sample sliced $\wass$ was mentioned in \cite{nadjahi2020statistical} but was left as an unproven assumption. Extensions to sliced $\Wp$ and two-sample results, all of which are crucial for principled statistical inference, are currently~open. Consistency of the bootstrap and efficiency bounds are also unaccounted for by the existing literature.\footnote{After the first version of the present paper was posted on the arXiv, we became aware that the latest update of \cite{manole2022minimax} (arXiv update: April 4, 2022) contains limit distribution and bootstrap results for $\SWp$ with $p > 1$ under the alternative, but under somewhat 
restrictive assumptions and using a different proof technique; see Remark \ref{rem: comparison manole} for more details. Our work is independent of \cite{manole2022minimax}.}

\subsection{Statistical analysis}

We move on to the statistical aspects of sliced $\Wp$, closing the aforementioned gaps. The $p>1$ case is treated under the general framework of \cref{sec: unified} for compactly supported distributions. For $p=1$, we present a separate derivation that makes use of the simplified form in \eqref{eq: W1 cdf} and the Kantorovich-Rubinstein duality to obtain the results under mild moment assumptions.

\subsubsection{Order $\bm{p>1}$} The next theorem characterizes limit distributions for average-sliced $p$-Wasserstein distances under both the one- and two-sample settings. It also states asymptotic efficiency of the empirical plug-in estimator, and consistency of the bootstrap. The latter facilitates statistical inference by providing a tractable estimate of the limiting distribution, and is set up as follows. Given the data $X_1,\dots,X_n$, let $X_1^B,\dots,X_n^B$ be an independent sample from $\hat{\mu}_n$, and set $\hat{\mu}_n^B:= n^{-1}\sum_{i=1}^n \delta_{X_i^B}$ as the bootstrap empirical distribution. Define $\hat{\nu}_n^B$ analogously and let $\Prob^B$ denote the conditional probability given the data.

\begin{theorem}[Limit distribution, efficiency, and bootstrap consistency for $\SWp^p$]
\label{thm: limit distribution SWp}
Let $1 < p < \infty$, and suppose that $\mu,\nu$ are compactly supported, such that $\mu$ is absolutely continuous and $\supp(\mu)$ is convex. For every $\theta \in \unitsph$, let $\varphi^{\theta}$ be an OT potential from $\ptheta_\sharp \mu$ to $\ptheta_\sharp \nu$ for $\mathsf{W}_p$, which is unique up to additive constants on $\inte (\supp(\ptheta_\sharp \mu))$. Also, set $\psi^\theta = [ \varphi^{\theta} ]^c$ as the $c$-transform of $\varphi^\theta$ for $c(s,t) = |s-t|^p$. 
The following hold. 

\begin{enumerate}
\setlength\itemsep{0.5em}
\item[(i)] We have
\[
\sqrt{n}\big(\SWp^p(\hat{\mu}_n,\nu) -\SWp^p(\mu,\nu)\big)  \stackrel{d}{\to} N \left (0, v_{p}^2 \right ),
\]
where $v_{p}^2 = \iint \Cov_{\mu}\big ( \varphi^{\theta} \circ \ptheta, \varphi^{\vartheta}\circ \mathfrak{p}^{\vartheta} \big) d\sigma(\theta) d\sigma(\vartheta)$, which is well-defined under the current assumption. 
The asymptotic variance $v_{p}^2$ coincides with the semiparametric efficiency bound for  estimating $\SWp^p(\cdot,\nu)$ at $\mu$. Also, provided that $v_{p}^2 > 0$, we have
\[
\sup_{t \in \R} \left | \Prob^B \Big ( \sqrt{n}\big(\SWp^p(\hat{\mu}_n^B,\nu) -\SWp^p(\hat{\mu}_n,\nu)\big) \le t \Big ) - \Prob\big(N(0,v_{p}^2) \le t\big) \right| \stackrel{\Prob}{\to} 0. 
\]
\item[(ii)] If in addition $\nu$ is absolutely continuous with convex support, then 
\[
\sqrt{n}\big(\SWp^p(\hat{\mu}_n,\hat{\nu}_n) -\SWp^p(\mu,\nu)\big)  \stackrel{d}{\to} N \left (0, v_{p}^2+w_{p}^2 \right ),
\]
where $v_{p}^2$ is given in (i) and $w_{p}^2 = \iint \Cov_{\nu}\big ( \psi^{\theta} \circ \ptheta, \psi^{\vartheta}\circ \mathfrak{p}^{\vartheta} \big) d\sigma(\theta) d\sigma(\vartheta)$.
The asymptotic variance $v_{p}^2+w_{p}^2$ coincides with the semiparametric efficiency bound for  estimating $\SWp^p(\cdot,\cdot)$ at $(\mu,\nu)$. Also, 
provided that $v_{p}^2+w_{p}^2 > 0$, we have
\[
\sup_{t \in \R} \left | \Prob^B \Big ( \sqrt{n}\big(\SWp^p(\hat{\mu}_n^B,\hat{\nu}_n^B) -\SWp^p(\hat{\mu}_n,\hat{\nu}_n)\big) \le t \Big ) - \Prob \big(N(0,v_{p}^2+w_{p}^2) \le t\big) \right| \stackrel{\Prob}{\to} 0.
\]
\end{enumerate} 
\end{theorem}

The derivation of the limit distributions in \cref{thm: limit distribution SWp} follows from \cref{prop: master proposition}. We outline the main idea for the one-sample case. The functional of interest is set as the $p$th power of the average-sliced $p$-Wasserstein distance. Leveraging compactness of supports, we then show that $\SWp^p$ is Lipschitz w.r.t. $\MSWone$ (cf. Lemma \ref{lem: Lipschitz}). From the Kantorovich-Rubinstein duality, $\MSWone$ can be expressed as $\MSWone (\mu,\nu) = \| \mu - \nu \|_{\infty,\calF}$ with $\cF = \{\varphi \circ \ptheta : \theta \in \unitsph,  \varphi \in \mathsf{Lip}_{1,0}(\R)\}$, which is shown to be $\mu$-Donsker ($\mathsf{Lip}_{1,0}(\R)$ denotes the class of $1$-Lipschitz functions $\varphi$ on $\R$ with $\varphi(0)=0$). Evaluating the G\^{a}teaux directional derivative of the sliced distance, we have all the conditions needed to invoke Proposition \ref{prop: master proposition}, which in turn yields the distributional limits. For $\SWp$, the corresponding derivative turns out to be linear (in a suitable sense), so that asymptotic efficiency of the plug-in estimator and the bootstrap consistency follow from Proposition \ref{prop: master proposition 2} and Corollary \ref{cor: master} combined with Theorem 23.9 in \cite{vanderVaart1998asymptotic}. 

For the two-sample case, we think of $\SWp(\mu,\nu)$ as a functional of the product measure $\mu \otimes \nu$, as the correspondence between $(\mu,\nu)$ and $\mu \otimes \nu$ is one-to-one. With this identification, the rest of the argument is analogous to the one-sample case. In addition, in the two-sample case, the semiparametric efficiency bound is defined relative to the tangent space 
\[
\big\{ h_1 \oplus h_2 : \text{$h_1$ and $h_2$ are bounded measurable functions with $\mu (h_1) = \nu(h_2)=0$} \big \}.
\]
We will keep this convention when discussing  semiparamtric efficiency bounds in the two-sample case.

\begin{remark}[Removing $p$th power]
While \cref{thm: limit distribution SWp} states limit distributions for the $p$th power of $\SWp$, we can readily obtain corresponding results for the average-sliced $p$-Wasserstein distance itself by invoking the delta method for the map $s \mapsto s^{1/p}$.
\end{remark}

\begin{remark}[Comparison with \cite{manole2022minimax}]
\label{rem: comparison manole}
Theorem 4 in \cite{manole2022minimax} derives limit distributions and bootstrap consistency for empirical $\SWp$ under the alternative, subject to the assumption that the projected densities $\{ f_\mu(\cdot; \theta) \}_{\theta \in \unitsph}$ and $\{ f_\nu(\cdot; \theta) \}_{\theta \in \unitsph}$ are uniformly integrable, with 
\[
\sup_{\theta \in \unitsph} \esssup_{0 < u <1} \frac{1}{f_{\mu}(F_\nu^{-1}(t; \theta); \theta)} \vee \frac{1}{f_{\nu}(F_\nu^{-1}(t; \theta); \theta)} < \infty.
\]
Here $f_\mu(\cdot;\theta)$ and $F^{-1}_\mu(\cdot;\theta)$ are the (Lebesgue) density and quantile functions of $\ptheta_{\sharp}\mu$, and their composition is the so-called $I$-function; cf. \cite[Equation (5.2)]{bobkov2019one}. Verification of this condition for  given distributions on $\R^d$ seems nontrivial. The proof of \cite[Theorem 4]{manole2022minimax} exploits the quantile function representation of $\Wp$ in $d=1$ along with a linearization step (of quantile functions). Our limit theorem, on the other hand, assumes that $\mu$ has a density with compact and convex support, and employs a markedly different proof via the general framework of Proposition \ref{prop: master proposition}.
\end{remark}

We next consider the max-sliced Wasserstein distance, and provide one- and two-sample limit distributions. It turns out that the corresponding  Hadamard directional derivative is nonlinear and therefore the limit is non-Gaussian and the nonparametric bootstrap is inconsistent (cf. \cite{dumbgen1993,fang2019}).

\begin{theorem}[Limit distribution for $\MSWp$]
\label{thm: limit distribution MSWp}
Consider the assumption of Theorem~\ref{thm: limit distribution SWp}.
\begin{enumerate}
\setlength\itemsep{0.5em}
\item[(i)] Setting $\mathfrak{S}_{\mu,\nu}:= \{ \theta \in \unitsph : \mathsf{W}_p(\ptheta_\sharp \mu, \ptheta_\sharp \nu) = \MSWp(\mu,\nu) \}$, we have
 \[
\sqrt{n}\big(\MSWp^p(\hat{\mu}_n,\nu) -\MSWp^p(\mu,\nu)\big)  \stackrel{d}{\to} \sup_{\theta \in \mathfrak{S}_{\mu,\nu}} \mathbb{G}_\mu (\theta),
\]
where $\big(\mathbb{G}_\mu(\theta)\big)_{\theta \in \unitsph}$ is a centered Gaussian process with continuous paths and covariance function $\Cov\big(\mathbb{G}_\mu(\theta), \mathbb{G}_\mu(\vartheta)\big) = \Cov_{\mu}\big ( \varphi^{\theta} \circ \ptheta, \varphi^{\vartheta}\circ \mathfrak{p}^{\vartheta} \big)$, which is well-defined under the current assumption. 
\item[(ii)] If in addition $\nu$ is also absolutely continuous with convex support, then 
\[
\sqrt{n}\big(\MSWp^p(\hat{\mu}_n,\hat{\nu}_n) -\MSWp^p(\mu,\nu)\big)  \stackrel{d}{\to} \sup_{\theta \in \mathfrak{S}_{\mu,\nu}} \big[\mathbb{G}_\mu (\theta) + \mathbb{G}_{\nu}'(\theta)\big],
\]
where $\big(\mathbb{G}_\mu(\theta)\big)_{\theta \in \unitsph}$ is given in (i), $\big(\mathbb{G}_\nu'(\theta)\big)_{\theta \in \unitsph}$ is another centered Gaussian process with continuous paths and covariance function $\Cov\big(\mathbb{G}_\nu'(\theta), \mathbb{G}_\nu'(\vartheta)\big) = \Cov_{\nu}\big(\psi^{\theta} \circ \ptheta, \psi^{\vartheta}\circ \mathfrak{p}^{\vartheta}\big)$ independent of $\mathbb{G}_\mu$.
\end{enumerate}
\end{theorem}

Observe that,  for $\mu,\nu \in \calP_p (\R^d)$, the map $\theta \mapsto \mathsf{W}_p(\ptheta_\sharp \mu, \ptheta_\sharp \nu)$ is continuous, so the set $\mathfrak{S}_{\mu,\nu}$ is nonempty. 
The proof of \cref{thm: limit distribution MSWp} also relies on the general framework of Proposition \ref{prop: master proposition}. As in the average case, $\MSWp$ is Lipschitz w.r.t. $\MSWone$, and the rest is to characterize the G\^{a}teaux directional derivative, which requires extra work.

\begin{remark}[Bias of plug-in estimator for $\MSWp$ and correction]
In general, the limit distributions for the max-sliced distance in Theorem \ref{thm: limit distribution MSWp} have positive means, which implies that empirical $\MSWp$ tends to be upward biased at the order of $n^{-1/2}$. Such a upward bias commonly appears in a plug-in estimation of the maximum of a nonparametric function (cf. \cite{chernozhukov2013intersection}). One may correct this bias using a precision correction similar to \cite{chernozhukov2013intersection}. Namely, define
\[
\widehat{\mathbb{W}}_p (\alpha):= \MSWp^p(\hat{\mu}_n,\hat{\nu}_n) - k_{\alpha}/\sqrt{n}
\]
where $k_{\alpha}$ is the $\alpha$-quantile of $\sup_{\theta \in \mathfrak{S}_{\mu,\nu}} \big[\mathbb{G}_\mu (\theta) + \mathbb{G}_{\nu}'(\theta)\big]$, which can be estimated via subsampling. Provided that $k_{\alpha}$ is a continuity point of the distribution function of the limit variable, this estimator is upward $\alpha$-quantile unbiased, and in particular, median unbiased when $\alpha=1/2$, meaning that $\Prob\big(\widehat{\mathbb{W}}_p (\alpha) \le \MSWp^p(\mu,\nu)\big) =\alpha  + o(1)$.
\end{remark}

\begin{remark}[Null case]
The limit distributions in Theorems \ref{thm: limit distribution SWp} and \ref{thm: limit distribution MSWp} degenerate to zero under the null, i.e., $\mu = \nu$. This is because $\varphi^\theta \circ \ptheta$ and $\psi^\theta \circ \ptheta$ are constant $\mu$- and $\nu$-a.e., respectively, for each $\theta \in \unitsph$. In $d=1$, \cite{delbarrio2005} derived a limit distribution for $\mathsf{W}_2(\hat{\mu}_n,\mu)$ using the quantile function representation of $\mathsf{W}_2$, which requires several technical conditions concerning the tail of the Lebesgue density of $\mu$. Their argument hinges on the approximation of the (general) sample quantile process by the uniform quantile process, and transfer the results on the latter to the former. This argument does not extend to the sliced distances, at least directly, as the quantile process here is indexed by the additional projection parameter $\theta \in \unitsph$. 
Also, it seems highly nontrivial to find simple conditions on $\mu$ itself under which the projected distributions $\ptheta_{\#}\mu$ satisfy the conditions in \cite{delbarrio2005} for all (or uniformly over) 
$\theta \in \unitsph$. We leave null limit distributions for $\SWp$ and $\MSWp$ as a future research~topic. 
\end{remark}

\subsubsection{Order $\bm{p=1}$} The analysis for $\SWone$ relies on the explicit expression of $\mathsf{W}_1$ between distributions on $\RR$ as the $L^1$ distance between distribution functions (cf.~\eqref{eq: W1 cdf}), which implies that
\begin{equation}
\label{eq: SWone expression}
\SWone(\mu,\nu) =  \int_{\unitsph} \int_{\R} |F_{\mu}(t;\theta)-F_{\nu}(t;\theta)| \, dt \,  d\sigma(\theta).
\end{equation}
For $\MSWone$, the Kantorovich-Rubinstein duality implies that 
\begin{equation}
\begin{split}
\MSWone (\mu,\nu) &= \sup_{\theta \in \unitsph} \sup_{\varphi \in \Lip (\R)} \left [ \int \varphi d(\ptheta_{\sharp}\mu-\ptheta_{\sharp}\nu) \right ] \\
&=\sup_{\theta \in \unitsph} \sup_{\varphi \in \Lip (\R)} \left [ \int \varphi(\theta^{\intercal}x) d(\mu-\nu)(x) \right ].
\end{split}
\label{eq: MSWone expression}
\end{equation}
Here $\Lip (\R)$ denotes the class of $1$-Lipschitz functions $\varphi$ on $\R$ with $\varphi(0)=0$. 
These explicit expressions enable us to derive a limit distribution theory for $\SWone$ and $\MSWone$ under mild moment conditions, as presented next.

To state the result, set $\lambda$ as the Lebesgue measure on $\R$ and denote
\[
\sign (t) = 
\begin{cases}
1  & t > 0 \\
0 & t=0 \\
-1 & t<0
\end{cases}
.
\]
Recall that a stochastic process $\big(Y(t)\big)_{t \in T}$ indexed by a measurable space $T$ is called  measurable if $(t,\omega) \mapsto Y(t,\omega)$ is jointly measurable. 

\begin{theorem}[Limit distribution for $\SWone$ and $\MSWone$]
\label{thm: sliced W1 limit}
Let $\epsilon>0$ be arbitrary. 
\begin{enumerate}
\setlength\itemsep{0.5em}
\item[(i)] Assume  $\mu \in \calP_{2+\epsilon}(\R^d)$ and $\nu \in \calP_1(\R^d)$. Then, there exists a measurable, centered Gaussian process $\mathsf{G}_{\mu} = \big(\mathsf{G}_{\mu} (t,\theta) \big)_{(t,\theta) \in \R \times \unitsph}$ with paths in $L^1(\lambda\otimes\sigma)$ and  covariance function  
\begin{equation}
\begin{split}
\Cov\big(\mathsf{G}_{\mu}(s,\theta)&,\mathsf{G}_{\mu}(t,\vartheta)\big)\\
&= \mu \big(\{ x \in \R^d : \theta^\intercal x \le s, \vartheta^\intercal y \le t \}\big)- F_{\mu}(s;\theta)F_{\mu}(t;\vartheta),\\ 
\label{eq: cov function}
\end{split}
\end{equation}
such that
\begin{equation}
\begin{split}
\sqrt{n} \big( \SWone (\hat{\mu}_n,\nu) &- \SWone (\mu,\nu) \big) \\
&\dconv  \iint \big[\mathrm{sign} (F_{\mu} - F_{\nu})\big] \mathsf{G}_{\mu} d\lambda d\sigma  + \iint_{F_{\mu}=F_{\nu}} |\mathsf{G}_{\mu}| d\lambda d\sigma.
\end{split}
\label{eq: limit distribution SW1}
\end{equation}
\item[(ii)] If $\mu, \nu \in \calP_{2+\epsilon}(\R^d)$, then
\[
\begin{split}
\sqrt{n} \big( \SWone (\hat{\mu}_n,&\hat{\nu}_n) - \SWone (\mu,\nu) \big) \\
&\dconv  \iint \big[\mathrm{sign} (F_{\mu} - F_{\nu})\big](\mathsf{G}_{\mu}-\mathsf{G}_{\nu}') d\lambda d\sigma  + \iint_{F_{\mu}=F_{\nu}} |\mathsf{G}_{\mu}-\mathsf{G}_{\nu}'| d\lambda d\sigma,
\end{split}
\]
where $\mathsf{G}_{\mu}$ is given in (i) and $\mathsf{G}_{\nu}'= \big(\mathsf{G}_{\nu}' (t,\theta) \big)_{(t,\theta) \in \R \times \unitsph}$ is another measurable, centered Gaussian process with paths in $L^1(\lambda \otimes \sigma)$ and covariance function given by \eqref{eq: cov function} with $\mu$ replaced by $\nu$, and $\mathsf{G}_{\mu}$ and $\mathsf{G}_{\nu}'$ are independent. 
\medskip
\item[(iii)] Assume $\mu \in \calP_{4+\epsilon}(\R^d)$ and $\nu \in \calP_1(\R^d)$. Consider the function class
\[
\cF = \left\{ \varphi \circ \ptheta : \theta \in \unitsph, \varphi \in \mathsf{Lip}_{1,0}(\R)\right \}.
\]
Then, there exists a  tight $\mu$-Brownian bridge process $G_\mu$ in $\ell^\infty(\cF)$
such that
\[
\sqrt{n}\big( \MSWone(\hat{\mu}_n,\nu) - \MSWone(\mu,\nu) \big) \dconv 
\sup_{f \in M_{\mu,\nu}} G_{\mu}(f)
,
\]
where $M_{\mu,\nu} =\big\{ f \in \overline{\cF}^{\mu} : \mu (f-\nu(f)) = \MSWone(\mu,\nu)\big\}$ and $\overline{\cF}^{\mu}$ is the completion of $\calF$ for the standard deviation pseudometric $(f,g) \mapsto \sqrt{\Var_{\mu}(f-g)}$. 
\item[(iv)] If $\mu,\nu \in \calP_{4+\epsilon}(\R^d)$, then there exists a  tight $\nu$-Brownian bridge process $G_\nu'$ in $\ell^\infty(\cF)$ independent of $G_\mu$ above such that
\[
\sqrt{n}\big( \MSWone(\hat{\mu}_n,\hat{\nu}_n) - \MSWone(\mu,\nu) \big) \dconv 
\sup_{f \in M'_{\mu,\nu}} \big[G_{\mu}(f)-G_\nu'(f)\big],
\]
where $M_{\mu,\nu}' =\{ f \in \overline{\cF}^{\mu,\nu} : (\mu-\nu)(f) = \MSWone(\mu,\nu) \}$ and $\overline{\cF}^{\mu,\nu}$ is the completion of $\calF$ for the pseudometric $(f,g) \mapsto \sqrt{\Var_{\mu}(f-g)}+\sqrt{\Var_{\nu}(f-g)}$. 
\end{enumerate}
\end{theorem}

Notably, while \cref{thm: limit distribution SWp} requires distributions to have compact and convex support, the limit distributions in \cref{thm: sliced W1 limit} hold under mild moment assumptions. The derivation for $\SWone$ relies on the CLT in $L^1$ to deduce convergence of empirical projected distribution functions in $L^1(\lambda \otimes \sigma)$. The limit distribution is then obtained via the functional delta method by casting $\SWone$ as the $L^1(\lambda \otimes \sigma)$ norm between distribution functions and characterizing the corresponding Hadamard directional derivative. For $\MSWone$, we use the Kantorovich-Rubinstein duality in conjunction with the fact that the class of projection 1-Lipschitz functions is Donsker under the said moment condition. We note that, in Part (iii), if $\mu = \nu$, then $M_{\mu,\nu} = \overline{\calF}^\mu$, and since $G_\mu$ has uniformly continuous paths w.r.t. the standard deviation pseudometric, the limit variable becomes $\sup_{f \in \calF}G_\mu(f)$ when $\mu=\nu$. Likewise, in Part (iv), the limit variable becomes $\sup_{f \in \calF}[G_\mu(f) - G_\nu'(f)]$ when $\mu=\nu$.

\medskip
For $\SWone$, if the second term on the right-hand side of (\ref{eq: limit distribution SW1}) is zero, then the asymptotic normality holds. We state this result including its two-sample analogue next. 

\begin{corollary}[Asymptotic normality for $\SWone$]\label{COR:asymp_normal_Wone}
\label{cor: SWone}
For $\mu\in\cP(\RR^d)$ and $\theta\in\unitsph$, define $\overline{l}_\mu^\theta = \sup \supp(\mathfrak{p}^{\theta}_{\sharp}\mu)$ and $\underline{l}^\theta_\mu = \inf \supp(\mathfrak{p}^{\theta}_{\sharp}\mu)$. The following hold.
\begin{enumerate}
    \item[(i)] Under the assumption of Theorem~\ref{thm: sliced W1 limit} Part (i), if in addition  $F_{\mu} (t;\theta) \ne F_{\nu} (t;\theta)$ for $(\lambda \otimes \sigma)$-almost all $(t,\theta) \in  \big \{ (s,\vartheta) :  s \in [\underline{l}_\mu^\vartheta,\overline{l}_\mu^\vartheta], \vartheta \in \unitsph \big \}$, then
    \[
\sqrt{n}\big( \SWone (\hat{\mu}_n,\nu) - \SWone (\mu,\nu)\big) \dconv N(0,v_{1}^2),
\]
where $v_{1}^2$ is the variance of $\iint \big[\mathrm{sign} (F_{\mu} - F_{\nu})\big] \mathsf{G}_{\mu} d\lambda d\sigma$. The asymptotic variance $v_{1}^2$ agrees with the semiparametric efficiency bound for estimating $\SWone (\cdot,\nu)$ at $\mu$. 
\item[(ii)] Under the assumption of Theorem \ref{thm: sliced W1 limit} Part (ii), if in addition  $F_{\mu} (t;\theta) \ne F_{\nu} (t;\theta)$ for $(\lambda \otimes \sigma)$-almost all $(t,\theta) \in  \big \{ (s,\vartheta) : s \in [\underline{l}_\mu^\vartheta \wedge \underline{l}_\nu^\vartheta,\overline{l}_\mu^\vartheta \vee \overline{l}_\nu^\vartheta], \vartheta \in \unitsph \big \}$, then
\[
\sqrt{n}\big( \SWone (\hat{\mu}_n,\hat{\nu}_n) - \SWone (\mu,\nu) \big) \dconv N(0,v_{1}^2+w_{1}^2),
\]
where $v_{1}^2$ is as above and $w_{1}^2$ is the variance of $\iint \big[\mathrm{sign} (F_{\mu} - F_{\nu})\big] \mathsf{G}_{\nu}' d\lambda d\sigma$. The asymptotic variance $v_{1}^2+w_{1}^2$ agrees with the semiparametric efficiency bound for estimating $\SWone (\cdot,\cdot)$ at $(\mu,\nu)$. 
\end{enumerate}
Finally, bootstrap consistency (as in Theorem \ref{thm: limit distribution SWp}) holds for both cases (i) and (ii). 
\end{corollary}

As an example, the above asymptotic normality holds when the population distributions are both Gaussian.

\begin{example}
Consider $\mu = N(\xi_1,\Sigma_1)$ and $\nu = N(\xi_2,\Sigma_2)$. Then $\ptheta_{\sharp}\mu = N(\theta^{\intercal}\xi_1,\theta^{\intercal}\Sigma_1\theta)$ and $\ptheta_{\sharp}\mu = N(\theta^{\intercal}\xi_2,\theta^{\intercal}\Sigma_2\theta)$. In this case, as long as $\xi_1 \ne \xi_2$ or $\Sigma_1 \ne \Sigma_2$, we have $\sigma (\{ \theta : \theta^{\intercal}\xi_1 = \theta^{\intercal}\xi_2 \ \text{and} \  \theta^{\intercal}\Sigma_1\theta = \theta^{\intercal}\Sigma_2\theta \}) = 0$, which follows from the fact that $\frac{Z}{\|Z\|} \sim \sigma$ for $Z \sim N(0,I_d)$ and Lemma 1 in \cite{okamoto1973}. Thus, $F_{\mu} (t;\theta) \ne F_{\nu} (t;\theta)$ for $(\lambda \otimes \sigma)$-almost all $(t,\theta) \in \R \times \unitsph$ and the conclusion of Corollary \ref{COR:asymp_normal_Wone} applies.
\end{example}

\section{Smooth Wasserstein Distance with Compactly Supported Kernels}
\label{sec: smooth Wp}
This section studies structural, statistical, and computational aspects of smooth Wasserstein distances with compactly supported smoothing kernels.

\subsection{Background} To set up the smooth Wasserstein distance, we first define a smoothing kernel as follows. Let $\chi\in C^{\infty}(\R^d)$ be any non-negative function with $\int_{\R^d}\chi(x)dx=1$ and $\int_{\R^d}\|x\|^p{\chi(x)}dx<\infty$, for all $1\leq p<\infty$. Then, for any $\sigma>0$, define $\chi_{\sigma}=\sigma^{-d}\chi(\cdot/\sigma)\in C^{\infty}(\R^d)$ and let $\eta_{\sigma}\in \cP(\R^d)$ be a probability measure whose (Lebesgue) density is $\chi_{\sigma}$. We call $\eta_\sigma$ a \textit{smoothing kernel} of parameter $\sigma$, and define the corresponding smooth Wasserstein distance as follows. 

\begin{definition}[Smooth Wasserstein distances] Let $1\leq p < \infty$ and $\eta_{\sigma}$ be a smoothing kernel. The associated smooth $p$-Wasserstein distance between $\mu,\nu\in\cP_p(\R^d)$ is 
\[
\mathsf{W}_p^{(\eta_{\sigma})}(\mu,\nu)=\mathsf{W}_p(\mu*\eta_{\sigma},\nu*\eta_{\sigma}).
\]
\end{definition}

\begin{example}[Standard mollifier]
A canonical example of a smooth compactly supported function is the standard mollifier
\begin{equation}
   \chi(x)=\begin{cases}\frac{1}{C_\chi}\exp\left(-\frac{1}{1-\|{x}\|^2}\right) &\text{if} \ \|{x}\|< 1\\
    0 &\text{otherwise}
    \end{cases}
    ,
    \label{eq:mollifier}
\end{equation}
where $C_\chi=\int_{\R^d}\chi d{\lambda}$, from which a compactly supported kernel is readily constructed. Our results, however are not specialized to the  mollifier kernel, and hold for any $\eta_\sigma$ as described above.
\end{example}

As reviewed next, Gaussian-smoothed Wasserstein distances, i.e., when $\eta_\sigma=\gamma_\sigma:=N(0,\sigma^2I_d)$, have been extensively studied for their structural and statistical~properties. 

\subsubsection{Literature review} Gaussian-smoothed Wasserstein distances were introduced in \cite{Goldfeld2020convergence} as a means to mitigate the curse of dimensionality in empirical estimation. Indeed, \cite{Goldfeld2020convergence} demonstrated that $\EE\big[\GWp(\hat{\mu}_n,\mu)\big] = O(n^{-1/2})$, for $p=1,2$, in arbitrary dimension provided that $\mu$ is sufficiently sub-Gaussian (cf. the recent preprint \cite{block2022smooth} for sharp bounds on the sub-Gaussian constant for which the rate is parametric when $p=2$). 
Structural properties of $\GWo$ were explored in~\cite{Goldfeld2020GOT}, showing that it metrizes the classic Wasserstein topology and establishing regularity in~$\sigma$. These structural and statistical results were later generalized to $\GWp$ for any $p> 1$ \cite{nietert21}, and asymptotics of the smooth distance as $\sigma\to\infty$ were explored \cite{chen2021asymptotics}. Relations between $\GWp$ and maximum mean discrepancies were studies in \cite{zhang2021convergence}, and nonparametric mixture model estimation under $\GWp$ was considered \cite{han2021nonparametric}, again demonstrating scalability of error bounds with dimension. The study of limit distributions for empirical $\GWp$ was initiated in \cite{Goldfeld2020limit_wass} for $p=1$ in the one-sample case, extended to the two-sample setting in \cite{sadhu21limit}, and generalized to arbitrary $p>1$ via a non-trivial application of the functional delta method in \cite{goldfeld22limit}. These works also considered bootstrap consistency and applications to minimum distance estimation and homogeneity testing. To date, a relatively complete limit distribution theory of $\GWp$ in arbitrary dimension is available, as opposed to the rather limited account of classic~$\Wp$.

\subsection{Structural properties} Henceforth, we consider smooth $\Wp$ with a compactly supported kernel. This is motivated by computational considerations, as the compact support enables leveraging algorithms such as \cite{vacher2021dimension} to compute OT between distributions with smooth densities. We note that while empirical $\mathsf{W}_p^{(\eta_{\sigma})}$ can also be evaluated by sampling the kernel and applying computational methods for classic $\Wp$, this approach fails to exploit the smoothness of our framework and hence our interest in the method of \cite{vacher2021dimension}. We first revisit the structural and statistical properties previously established for the Gaussian-smoothed case, using \cref{prop: master proposition} for the limit distribution theory. Afterwards, we discuss the computational aspect and, specifically, how to lift the algorithm from \cite{vacher2021dimension} to $\mathsf{W}_p^{(\eta_{\sigma})}$. 

\medskip

We adopt the shorthand $\mathsf W_p^{(\sigma)}:=\mathsf{W}_p^{(\eta_{\sigma})}$, and first explore its structural proprieties. We begin with a comparison between the smooth and unsmoothed distances.

\begin{proposition}[Stability of $\mathsf{W}_p^{(\sigma)}$]
\label{prop:smoothstability}
    For any $1\leq p<\infty$, $\sigma>0$, and $\mu,\nu\in\cP(\R^d)$, we have
    \[
        \mathsf{W}_p^{(\sigma)}(\mu,\nu)\leq\mathsf{W}_p(\mu,\nu)\leq \mathsf{W}_p^{(\sigma)}(\mu,\nu)+2\sigma (\E_{\eta_1}[\|X\|^p])^{1/p}.
    \]
    In particular, $\lim_{\sigma\downarrow 0} \mathsf W_p^{(\sigma)}(\mu,\nu)=\mathsf W_p(\mu,\nu)$.
\end{proposition}
The first bound is due to contractivity of $\mathsf W_p$ w.r.t. convolution. Constructing a coupling between $\rho\in\cP(\R^d)$ and $\rho*\eta_{\sigma}$ with total cost $\sigma (\E_{\eta_1}[\|X\|^p])^{1/p}$ proves the second. 
 Since $\eta_{\sigma}$ is compactly supported, $\supp(\eta_1)$ is contained in a ball of radius $r>0$, whereby $(\E_{\eta_1}[\|X\|^p])^{1/p}\leq r$ which is in contrast to the dimension dependent gap for Gaussian kernels; cf. Lemma~1 in \cite{Goldfeld2020GOT} and \cite{nietert21}. 

\medskip
As the smooth distance converges to the standard distance as $\sigma\to 0$, it is natural to expect that optimal couplings converge as well. This is stated in the next proposition.

\begin{proposition}[Stability of transport plans]
For $1\leq p<\infty$, $\mu,\nu\in \cP_p(\R^d)$, and $\sigma_k\downarrow 0$. Let $\pi_k$ be an optimal coupling for $\Wp^{({\sigma_k})}(\mu,\nu)$ for each $k\in\NN$. Then, there exists an optimal coupling $\pi$ for $\mathsf W_p(\mu,\nu)$ for which $\pi_k\stackrel{w}{\to} \pi$ along a subsequence.
\end{proposition}

The proof of this result follows that of Theorem~4 in \cite{Goldfeld2020GOT} and \cite{goldfeld22limit} with only minor changes and is hence omitted. Note that when the limiting $\pi$ is unique (e.g., when $p>1$ and $\mu$ has a density), then extraction of a subsequence is not needed. 

\medskip
With these stability results at hand, we next show that $\Wp^{(\sigma)}$ is indeed a metric on $\cP_p(\R^d)$ that induces the Wasserstein topology.
 
\begin{proposition}[Metric and topological structure]
    \label{prop:metricspace}
    For $1\leq p<\infty$ and $\sigma> 0$, $\mathsf W_p^{(\sigma)}$ is a metric on $\cP_p(\R^d)$ inducing the same topology as $\mathsf W_p$.
\end{proposition}   

The proof of \cref{prop:metricspace} is similar to that of Proposition 1 in \cite{nietert21} and Theorem 1 in \cite{Goldfeld2020GOT} for the Gaussian kernel, which uses the fact that the Gaussian characteristic function never vanishes. This property does not necessarily hold for compactly supported kernels, and hence we instead show their characteristic function nullifies at most on a null set, which suffices to adapt the argument. 

\subsection{Statistical analysis}
\label{sec:SmoothLimitDistributionTheory}

This section studies statistical aspects of smoothed Wasserstein distances with  compactly supported kernels. Specifically, we establish limit distributions and parametric convergence rates for empirical $\Wp^{(\sigma)}$ for compactly supported distributions and kernels. In what follows, we restrict our attention to distributions $\mu$ supported in a compact set $\mathcal{X}\subset \R^d$ (i.e., $\supp (\mu) \subset \calX$).  
We will identify a Borel probability measure on $\R^d$ whose support is contained in $\calX$  with a probability measure on $\calX$, and vice versa (i.e., we identify $\mu \in \calP(\calX)$ with its extension to $\R^d$ via $\mu(\cdot \cap \calX)$). 
Let $\mathcal{X}_{\sigma}:=\mathcal{X} + \overline{B(0,\sigma)}$, and assume for simplicity that the density of $\eta_{\sigma}$ is identically zero on $\R^d\setminus B(0,\sigma)$.
The set $\mathcal{X}_{\sigma}$ contains the support of any convolved measure $\mu*\eta_{\sigma}$ for $\mu \in\cP(\mathcal{X})$.

\subsubsection{Limit distributions for $\bm{p>1}$ under the alternative}

Building on the unified framework, the next theorem establishes asymptotic normality of empirical $\mathsf{W}_p^{(\sigma)}$ under the alternative. This contrasts the non-Gaussian limit distributions under the null and the $p=1$ case, which are treated in the sequel. Recall that $\calX \subset \R^d$ is compact. 
 
\begin{theorem}[Limit distributions for $\mathsf W_p^{(\sigma)}$ under the alternative] 
\label{thm:LimitDistributionsSmoothWasserstein}
Set $1<p<\infty$, $\sigma>0$, $S_p^{(\sigma)}:=\big[\mathsf W_p^{(\sigma)}\big]^p$, and let $\mu,\nu\in\cP(\mathcal{X})$ be such that  $\interior(\supp(\mu*\eta_{\sigma}))$ is connected. Let $\varphi$ be an OT potential from $\mu*\eta_{\sigma}$ to $\nu*\eta_{\sigma}$ for $\mathsf W_p$,
which is unique on $\interior(\supp(\mu*\eta_{\sigma}))$ up to additive constants.  
\begin{enumerate}
    \item[(i)] We have
    \[
      \sqrt n\left(
      S_p^{(\sigma)}(\hat\mu_n,\nu)-S_p^{(\sigma)}(\mu,\nu)\right){\stackrel{d}{\to}} N\left(0,{v}^2_{p}\right)
    \]
    where ${v}^2_{p}:=\Var_{\mu}(\varphi*\chi_{\sigma})$. The asymptotic variance ${v}^2_{p}$ coincides with the semiparametric effiency bound for estimating $S_p^{(\sigma)}(\cdot,\nu)$ at $\mu$. Also, provided that $v_p^2 > 0$, we have
    \[
\sup_{t \in \R} \left | \Prob^B \Big ( S_p^{(\sigma)}(\hat{\mu}_n^B,\nu)-S_p^{(\sigma)}(\hat\mu_n,\nu)
\big) \le t \Big ) - \Prob\big(N(0,v_{p}^2) \le t\big) \right| \stackrel{\Prob}{\to} 0. 
\]
    \item[(ii)] 
    If in addition $\nu*\eta_{\sigma}$ has connected support, then
    \[
        \sqrt n\left(
        S_p^{(\sigma)}(\hat\mu_n,\hat\nu_n)-S_p^{(\sigma)}(\mu,\nu)
       \right){\stackrel{d}{\to}}N\left(0,{v}_{p}^2+{w}_{p}^2\right),
    \]
    where ${v}^2_{p}$ is as in (i) and  ${w}_{p}^2:=\Var_{\nu}(\varphi^c*\chi_{\sigma})$. The asymptotic variance ${v}^2_{p}+{w}^2_{p}$ coincides with the semiparametric efficiency bound for estimating $S_p^{(\sigma)}$ at $(\mu,\nu)$. Also, provided that ${v}_{p}^2+{w}_{p}^2 > 0$, we have
    \[
\sup_{t \in \R} \left | \Prob^B \Big ( S_p^{(\sigma)}(\hat{\mu}_n^B,\hat\nu_n^B)-S_p^{(\sigma)}(\hat\mu_n,\hat\nu_n)
\big) \le t \Big ) - \Prob\big(N(0, v_{p}^2+ w_{p}^2) \le t\big) \right| \stackrel{\Prob}{\to} 0. 
\]
\end{enumerate}
\end{theorem}

The proof of \cref{thm:LimitDistributionsSmoothWasserstein} applies \cref{prop: master proposition} to the functional $\rho*\eta_\sigma\mapsto \mathsf{W}_p^p(\rho*\eta_{\sigma},\nu*\eta_{\sigma})$ for $\rho \in \calP(\calX)$ with $\supp(\rho) \subset \supp (\mu)$. To this end, we will show that the preceding functional is Lipschitz continuous w.r.t. $\|\cdot\|_{\infty,B}$ for the unit ball $B$ in  $L^2(\mathcal{X}_{\sigma})$, which follows by duality \eqref{eq: duality} and uniform bounds on the OT potentials (cf. Remark 1.13 in \cite{villani2003topics}). Here $L^2(\calX_\sigma)$ is the Hilbert space of square integrable functions on $\calX_\sigma$ w.r.t. the Lebesgue measure, equipped with the standard $L^2$-inner product. The differentiability result follows by adapting the Gaussian kernel case (cf. Lemma 3.3 of \cite{goldfeld22limit}). To prove weak convergence of the smoothed empirical process $\sqrt n(\hat \mu_n-\mu)*\eta_{\sigma}$ in $\ell^{\infty}(B)$, we employ the CLT in $L^2(\calX_\sigma)$ and use a linear isometry from $L^2(\mathcal{X}_{\sigma})$ into $\ell^{\infty}(B)$. Linearity of the derivative yields asymptotic efficiency and bootstrap consistency. 
The overall argument differs from the Gaussian kernel case in \cite{goldfeld22limit}, where they compare the Gaussian smoothed $\Wp$ with a certain dual Sobolev norm and show weak convergence of the smoothed empirical process by showing that the Sobolev unit ball convolved with the Gaussian density is Donsker. As in \cref{thm: limit distribution SWp}, while the results are stated for the $p$th power of the smooth distance, we can obtain limit distributions for $\mathsf W_p^{(\sigma)}$ itself via the delta method for $s\mapsto s^{1/p}$.

\begin{remark} In \cref{thm:LimitDistributionsSmoothWasserstein}, the condition that $\interior(\supp(\mu*\eta_{\sigma}))$ be connected is satisfied if $\mu$ has connected support by \cref{lem:connectedsupport}. The same lemma also shows that the boundary of $\interior(\supp(\mu*\eta_{\sigma}))$ is negligible for any choice of $\mu\in\cP(\mathcal{X})$.
\end{remark}

The ideas from the proof of \cref{thm:LimitDistributionsSmoothWasserstein} coupled with Hilbertian structure of $L^2(\mathcal{X}_{\sigma})$ yield rates of convergence in expectation for empirical $\mathsf{W}_p^{(\sigma)}$.  

\begin{proposition}[Parametric convergence rates]
\label{thm:ExpectedW2Rates}
For $1 <  p<\infty$, $\sigma>0$, and $\mu,\nu\in\cP(\mathcal{X})$ with $\mu\neq \nu$, we have 
\[
\begin{split}
&\E\left[\abs{\mathsf{W}_p^{(\sigma)}(\hat\mu_n,\nu)-\mathsf{W}_p^{(\sigma)}(\mu,\nu)}\right]\\
&\qquad\qquad\leq 2\| \chi_\sigma \|_{\infty} \sqrt{\lambda (B(0,\sigma)) \lambda(\mathcal{X}_{\sigma})} \diam(\mathcal{X}_{\sigma})^p \big[\mathsf{W}_p^{(\sigma)}(\mu,\nu)\big]^{1-p} n^{-1/2}.
\end{split}
\]
\end{proposition}

\subsubsection{Limit distributions for $\bm{p=2}$ under the null}
We now derive limit distributions for $\mathsf W_2^{(\sigma)}$ under the null. Our approach is based on the CLT in Hilbert spaces and is thus limited to the case $p=2$ (see the discussion after \cref{thm:LimitDistributionsSmoothWassersteinNULL} for details). To state this result, we require the following machinery. Let $C_{0}^\infty$ denote the space of infinitely differentiable, compactly supported real functions on $\R^d$.

\begin{definition}[Sobolev spaces and their duals]
    The Sobolev seminorm of a differentiable function $f:\R^d\to \R$ w.r.t. a reference measure  $\mu\in\cP(\R^d)$ is denoted by $\|f\|_{\Ht{\mu}}:=\|\nabla f\|_{L^2(\mu)}$. The homogeneous Sobolev space is defined as the completion of $C_0^{\infty}+\R$ with respect to $\|\cdot\|_{\Ht{\mu}}$. The dual Sobolev space $\Hmt{\mu}$ is the topological dual of $\Ht{\mu}$. 
\end{definition} 

\begin{definition}[$2$-Poincar{\'e} inequality]
    A probability measure $\mu\in \cP(\R^d)$ is said to satisfy the $2$-Poincar{\'e} inequality if there exists $C<\infty$, such that
    \[
        \|f-\mu(f)\|_{L^2(\mu)}\leq C\|\nabla f\|_{L^2(\mu;\R^d)},\quad f\in C_0^{\infty},
    \]
    where $L^2(\mu;\R^k)$ denotes the space of measurable maps $f:\R^d\to\R^k$ for which $\|f\|_{L^2(\mu;\R^d)}^2:=\int_{\R^d}\|f\|^2\dx{\mu}<\infty$.
\end{definition}

With these definitions in place, we state the limit distribution for $\mathsf{W}_2^{(\sigma)}$.

\begin{theorem}[Limit distributions for $\mathsf{W}_2^{(\sigma)}$ under the null] 
\label{thm:LimitDistributionsSmoothWassersteinNULL}
    Let $\mu\in\cP(\mathcal{X})$ be such that $\mu*\eta_{\sigma}$ satisfies the $2$-Poincar{\'e} inequality and set $\sigma>0$. The following hold.
\begin{enumerate}
    \item[(i)] We have 
    \[
        \sqrt n \mathsf W_2^{(\sigma)}(\hat\mu_n,\mu)\stackrel{d}{\to}\|{\mathbb{G}_{\mu}}\|_{\Hmt{\mu*\eta_{\sigma}}},
    \]
    where $\big(\mathbb{G}_{\mu}(f)\big)_{f\in\dot H^{1,2}(\mu*\eta_{\sigma})}$ is a centered Gaussian process with
    covariance function $\Cov\big(\mathbb{G}_\mu(f),\mathbb{G}_\mu(g)\big) = \Cov_{\mu}(f*\chi_{\sigma},g*\chi_{\sigma})$ whose paths lie in $\Hmt{\mu*\eta_{\sigma}}$ a.s. 
    \item[(ii)] Additionally, if $\mu=\nu$, then
    \[
        \sqrt n \mathsf W_2^{(\sigma)}(\hat\mu_n,\hat\nu_n)\stackrel{d}{\to}\|{\mathbb G_{\mu}-
    \mathbb G'_{\mu}}\|_{\Hmt{\mu*\eta_{\sigma}}},
    \]
    where $\mathbb G_{\mu}'$ is an independent copy of $\mathbb G_{\mu}$.
\end{enumerate}
\end{theorem}

\begin{remark}
   In \cref{thm:LimitDistributionsSmoothWassersteinNULL}, a sufficient condition for $\mu*\eta_{\sigma}$ to satisfy the $2$-Poincar{\'e} inequality is that both $\mu$ and $\eta_{\sigma}$ satisfy $2$-Poincar{\'e} inequalities (see Proposition 1.1 in \cite{WanWan16}). For instance, when $\eta_{1}$ has density given by \eqref{eq:mollifier}, $\eta_{\sigma}$ is a log-concave measure (cf. \cite{lovasz2007logcon,saumard2014log}) and hence satisfies the $2$-Poincar{\'e} inequality \cite{bobkov1999,milman2009role}.
\end{remark}

The derivation of \cref{thm:LimitDistributionsSmoothWassersteinNULL} follows an essentially similar approach to the Gaussian kernel case in \cite{goldfeld22limit}. 
However, in contrast to the proof of Proposition 3.1 in \cite{goldfeld22limit}, to show weak convergence of the smoothed empirical process in $\Hmt{\mu*\eta_{\sigma}}$, we exploit the fact that $\Hmt{\mu*\eta_{\sigma}}$ is a Hilbert space (since it is isometrically isomorphic to a closed subspace of $L^2(\mu*\eta_{\sigma};\R^d)$), and apply the CLT in the Hilbert space.
To this end, we first verify that the smoothed empirical process has paths in $\Hmt{\mu*\eta_{\sigma}}$. This step requires 
control of the inverse of the density of $\mu*\eta_{\sigma}$, which decays to zero near the boundary of its support; see the proof of \cref{lem:SmoothDualSobolev}. A direct extension of the proof of Proposition 3.1 in \cite{goldfeld22limit} to compactly supported kernels for general $1<p<\infty$ requires a much finer analysis of this density 
than provided in the proof of \cref{lem:SmoothDualSobolev}. Such an analysis appears highly nontrivial and  hence we focus here on the $p=2$ case.

Analogously to \cref{thm:ExpectedW2Rates}, parametric rates for empirical $\mathsf{W}_2^{(\sigma)}$ under the null follow from Hilbertianity of $\Hmt{\mu*\eta_{\sigma}}$ and the ideas from the proof of  \cref{thm:LimitDistributionsSmoothWassersteinNULL}.

\begin{proposition}[Parametric convergence rates]
\label{thm:ExpectedW2RatesNULL}
For $\sigma>0$ and $\mu\in\cP(\mathcal{X})$ for which $\mu*\eta_{\sigma}$ satisfies the $2$-Poincar{\'e} inequality with constant $C_{\mu,\sigma}$, we have
\[
    \E\left[\mathsf{W}_2^{(\sigma)}(\hat\mu_n,\mu)\right]\leq 2C_{\mu,\sigma}\sqrt{(1\vee \|{\chi_{\sigma}^2}\|_{\infty})\lambda(\mathcal{X}_{\sigma}) }n^{-1/2}. 
\]
\end{proposition}

\subsubsection{Limit distributions for $\bm{p=1}$}
We now treat the limit distributions for $\mathsf{W}_1^{(\sigma)}$ under both the null and the alternative. The Kantorovich-Rubinstein duality for $\mathsf{W}_1$ enables us to do so in the absence of the additional assumptions required when $p>1$.  
In what follows, let $\Lip$ denote the set of $1$-Lipschitz functions $f$ on $\R^d$ with $f(0)=0$ and $\calF_\sigma = \{ f * \chi_\sigma : f \in \Lip \}$. Observe that $\mathsf{W}_1^{(\sigma)}(\mu,\nu) = \sup_{f \in \calF_\sigma}(\mu - \nu)(f)$ by the Kantorovich-Rubinstein duality. 

\begin{theorem}[Limit distributions for $\mathsf W_1^{(\sigma)}$]
    \label{thm:RatesSmooth1WassersteinNULL}
    Let $\sigma>0$ and $\mu,\nu\in\cP(\mathcal{X})$. There exist independent, tight $\mu$- and $\nu$-Brownian bridge process $G_\mu$ and $G_\nu'$ in $\ell^\infty (\calF_\sigma)$, respectively, such that the following hold. 
    \begin{enumerate}
     \item[(i)] We have
  \[
        \sqrt n\left(\mathsf W_1^{(\sigma)}(\hat\mu_n,\nu)-\mathsf W_1^{(\sigma)}(\mu,\nu)\right){\stackrel{d}{\to}} \sup_{f \in M_\sigma} G_\mu(f),
   \]
   where $M_\sigma =\big \{ f \in \overline{\calF}^\mu_\sigma : \mu \big(f - \nu(f)\big) = \mathsf{W}_1^{(\sigma)}(\mu,\nu)\big\}$ and $\overline{\calF}^\mu_\sigma$ is the completion of $\calF_\sigma$ for the pseudometric $(f,g) \mapsto \sqrt{\Var_\mu (f-g)}$.
        
    \item[(ii)] We have
  \[
        \sqrt n\left(\mathsf W_1^{(\sigma)}(\hat\mu_n,\hat\nu_n)-\mathsf W_1^{(\sigma)}(\mu,\nu)\right){\stackrel{d}{\to}} \sup_{f \in M_\sigma'} [G_{\mu}(f)-G_{\nu}'(f)],
   \]
where $M_\sigma' = \big\{ f \in \overline{\calF}^{\mu,\nu}_\sigma : (\mu-\nu) (f) = \mathsf{W}_1^{(\sigma)}(\mu,\nu)\big\}$ and $\overline{\calF}^{\mu,\nu}_\sigma$ is the completion of $\calF_\sigma$ for the pseudometric $(f,g) \mapsto \sqrt{\Var_\mu (f-g)}+\sqrt{\Var_\nu (f-g)}$.
   \end{enumerate}
\end{theorem}

\cref{thm:RatesSmooth1WassersteinNULL}  follows by showing that the function class $\calF_\sigma$ is Donsker combined with the extended functional delta method for the supremum functional.
The proof of \cref{thm:RatesSmooth1WassersteinNULL} also establishes parametric rates for empirical $\mathsf W_1^{(\sigma)}$.

\begin{corollary}[Parametric convergence rate]
    \label{cor:RatesSmooth1WassersteinNull}
    In the setting of \cref{thm:RatesSmooth1WassersteinNULL}, we have
    $\E\big[\mathsf{W}_1^{(\sigma)}(\hat\mu_n,\mu)\big] =O(n^{-1/2})$.
\end{corollary}

\subsection{Computational aspects}
\label{sec:SmoothComputation}
 
In this section, we take a first step towards addressing computational aspects for smooth Wasserstein distances, leveraging the framework of \cite{vacher2021dimension} for computing the $\mathsf{W}_2$ distance between measures with smooth compactly supported densities that are bounded from above and below. We first outline the approach of \cite{vacher2021dimension}, then present a truncation argument that enables applying it for $\mathsf{W}_2^{(\sigma)}$, and lastly describe some practical limitations of the original method.    

\subsubsection{Summary of the method} 
\label{sec:summaryVacher}
Throughout, let $c(x,y)=\|x-y\|^2/2$ and $\mu_1,\mu_2\in\cP(\R^d)$ be such that $A_i:=\interior(\supp(\mu_i))$ is bounded and convex for $i=1,2$. Assume in addition that $\mu_i$ has (Lebesgue) density $f_i$ which is bounded away from zero and infinity on $A_i$ with Lipschitz derivatives up to order $m>d$ and that $A_i$ is smooth and uniformly convex (see e.g. \cite[p.339]{gilbarg2015elliptic}) for $i=1,2$.\footnote{Assumption 1 in \cite{vacher2021dimension} omits smoothness and uniform convexity of $A_i$. However, in the absence of these conditions it is unclear if their claimed boundary regularity of OT maps and potentials hold (cf. Theorem 3.3 in \cite{dephilippis2014monge}).} These technical conditions guarantee that the OT potentials $\varphi_1^{\star},\varphi^{\star}_2$ solving the alternate dual,
\begin{equation}
\label{eq:RKHSdual} \frac{1}{2}\mathsf{W}_2^2(\mu_1,\mu_2)=
\sup_{\substack{(\varphi_1,\varphi_2)\in L^1(\mu_1)\times L^2(\mu_2)\\\varphi_1\oplus\varphi_2\leq c}}\left[\int_{A_1}\varphi_1\dx{\mu_1}+\int_{A_2}\varphi_2\dx{\mu_2}\right],
\end{equation}
are elements of certain Sobolev spaces.   

The first insight of \cite{vacher2021dimension} is to exploit the reproducing kernel Hilbert spaces (RKHS) of such Sobolev spaces (cf. {\cite{wendland2004scattered},\cite[Proposition 4]{vacher2021dimension}}) so as to facilate computations via the kernel trick \cite{scholkopf1998nonlinear} and a representer theorem \cite[Lemma 13]{vacher2021dimension}. Next, \eqref{eq:RKHSdual} is replaced by an equivalent equality-constrained problem \cite[Equation 2]{vacher2021dimension}. This is motivated by the observation that equality constraints can be subsampled (i.e., enforced only on a set $Z_{\ell}$ of cardinality $\ell$) at the cost of incurring an error of $O(h_{\ell}^{m+1-d})$ provided $h_{\ell}\lesssim m^{-2}$, 
 where $h_{\ell}:=\sup_{x\in A_1,y\in A_2}\inf_{(\tilde x_{\ell},\tilde y_{\ell})\in Z_{\ell}}\|{(x,y)-(\tilde x_{\ell},\tilde y_{\ell})}\|$ \cite[Theorem 2.12]{narcowich2005sobolev}. 

Finally, the objective of the subsampled and equality-constrained version of \eqref{eq:RKHSdual} is regularized to guarantee the uniqueness of its solution. The dual of this problem is a tractable $\ell$ dimensional optimization problem whose solution can be used to compute the estimate $\widetilde{\mathsf W}_2^2(\mu_1,\mu_2)$ of ${\mathsf W}_2^2(\mu_1,\mu_2)$ which has error $O\big(\lambda_1+\lambda_2+{(\gamma+h _{\ell}^{m+1-d})^2}/{\lambda_2}\big)$ provided $h_{\ell}\lesssim m^{-2}$ and $h_{\ell}^{m+1-d}\lesssim\lambda_1$. Here $(\lambda_1,\lambda_2)$ are positive hyperparameters introduced in the regularization procedure, and $\gamma$ is related to the error of approximate integration (see Theorem 9 in \cite{vacher2021dimension}). 
A variant of Newton's method is shown to 
 solve the dual problem within a tolerance of $\epsilon$ with time complexity $O(C+E\ell+\ell^{3.5}\log(\ell/\epsilon) )$ and memory complexity $O(\ell^2)$, where $C$ and $E$ are costs of evaluating certain functions which may depend on $\ell$ \cite[Equation 7]{vacher2021dimension}.

    The followup work \cite{Muzellec2021estimation} provides a variant of this method tailored for approximate computation of OT maps for the cost $c(x,y)=\ip{x,y}$ given samples from $\mu$ and $\nu$. The improvement of this approach upon the method of \cite{vacher2021dimension} is twofold: it provides a heuristic by which to choose the hyperparameters 
    and it proposes an alternative regularization method so as to leverage optimization methods which scale better with $\ell$ than the Newton method.

\begin{remark}[Smoothness versus subsampling]
The value of $h_{\ell}^{m+1-d}$ plays an important role in the proposed method. While it may appear negligible provided $m$ is sufficiently large and $h_{\ell}<1$, the theoretical guarantees of \cite{vacher2021dimension} are contingent on assuming $h_{\ell}\lesssim m^{-2}$. Indeed, this condition enables controlling the amount by which solutions of the subsampled equality constraint violate the original constraint 
\cite[Theorem 2.12]{narcowich2005sobolev}. Hence, utilizing higher regularity of densities comes at the cost of increasing the number of subsamples which in turn implies a greater computational cost in terms of time and memory.
\end{remark}

\subsubsection{Truncation of measures and application of the algorithm}

We first observe that the above described method is not directly applicable to $\mathsf{W}_2^{(\sigma)}$, since the density of $\mu*\eta_{\sigma}$ decays to zero on $\supp(\mu*\eta_{\sigma})$, which violates the required lower-boundedness. To rectify this issue, we truncate the measures of interest and quantify the OT gap between the original measure and its truncated version. 

Let $\mu\in\cP(\R^d)$ be arbitrary and $A\subset\R^d$ be Borel measurable with $\mu(A)>0$.~The truncation of $\mu$ to $A$, denoted $\mu\vert_A$, is the conditional probability measure $\mu(\cdot\vert A)$. The next result quantifies the error (in $\mathsf{W}_p$) of approximating a measure by its~truncation.

\begin{proposition}[Truncation error]
    \label{prop:truncboundcompact}
     Let $\mu\in\cP_{p+1}(\R^d)$, and  $A\subset\R^d$ be a bounded Borel set with $\mu(A)>0$. The following hold.
    \begin{enumerate}
    \item[(i)] If $1<p<\infty$ and $\mu$ has a $(c_1,c_2)$-regular density $f_{\mu}$ \footnote{$f_{\mu}$ is a $(c_1,c_2)$-regular density if $\|\nabla \log f_{\mu}\|\leq c_1\|\cdot\|+c_2$ for $c_1>0$ and $c_2\geq 0$.}, then 
    \[
        \mathsf W_p^p(\mu\vert_A,\mu)\leq C\left(\E_{\mu}[\|{X}\|]+\E_{\mu}\big[\|{X}\|^{p+1}\big]\mspace{-2mu}+\mspace{-2mu}\big(1\mspace{-2mu}-\mspace{-2mu}2\mu(A)\big)\big(\E_{\mu\vert_A}[\|{Y}\|]\mspace{-2mu}+\mspace{-2mu}\E_{\mu\vert_A}\big[\|{Y}\|^{p+1}\big]\big) \right),
    \]
    where $C$ depends only on $p,d,c_1,c_2,\diam(A),$ and a lower bound on $f_{\mu}(0)$. 
    \item[(ii)] If $1\leq p<\infty$ and $\mu$ has compact support, then
    \[ 
        \mathsf W_p^p(\mu\vert_A,\mu)\leq  \left( \frac{1}{\mu(A)}-1 \right)\diam\big(\supp(\mu)\big)^{p}.   
    \]
    \end{enumerate}
\end{proposition}
These bounds exhibit the expected behaviour in the sense that the larger $\mu(A)$ is, the closer $\mu\vert_A$ is to $\mu$ under $\mathsf{W}_p$, and if $\mu(A)=1$ the error is zero.  Both results follow from estimates for OT potentials between probability measures adhering to the above assumptions; see \cref{subsec:proof-prop:truncboundcompact}. 

\medskip
\cref{prop:truncboundcompact} implies that, at the cost of introducing an approximation error, kernel smoothed measures can be truncated so as to match the assumptions of \cref{sec:summaryVacher}. The following result summarizes the overall accuracy and complexity guarantees obtained by combining the algorithm from \cite{vacher2021dimension} with \cref{prop:truncboundcompact}.

\begin{proposition}[Computational guarantees]
    \label{prop:ComputationalGuarantees}
    Let $\mu_1,\mu_2\in\mathcal P(\mathbb R^d)$ be compactly supported, and $B_1,B_2\subset \R^d$ be closed sets with $\mu_i*\eta_{\sigma}(B_i)>0$, such that $A_i:=\interior(B_i)$ is smooth and uniformly convex for $i=1,2$. Further, assume that the density of $\mu_i*\eta_{\sigma}$ is bounded away from $0$ on $B_i$ for $i=1,2$. Then, 
    $\widetilde{\mathsf{W}}_2\left(\mu_1*\eta_{\sigma}\vert_{A_1},\mu_2*\eta_{\sigma}\vert_{A_2}\right)$
 described in \cref{sec:summaryVacher} can be computed up to accuracy $\epsilon$ in $O(C+E\ell+\ell^{3.5}\log(\ell/\epsilon) )$ time using $O(\ell^2)$ memory and, for any choice of $m\in \NN$, this estimate approximates 
    $\mathsf{W}_2(\mu_1*\eta_{\sigma},\mu_2*\eta_{\sigma})$ within error 
    \[
    O\left((\mu_1(A_1)^{-1}-1)^{1/2}+(\mu_2(A_2)^{-1}-1)^{1/2}+\left(\lambda_1+\lambda_2+{(\gamma+h _{\ell}^{m+1-d})^2}/{\lambda_2}\right)^{1/2}\right)
    \]
    provided $h_{\ell}^{m+1-d}\lesssim\lambda_1,0<\lambda_2,$ and $h_{\ell}\lesssim m^{-2}$.
\end{proposition} 

 \cref{prop:ComputationalGuarantees} can be adapted to the case where $\mu_1*\eta_{\sigma}$, $\mu_2*\eta_{\sigma}$, or both have regular densities by replacing the truncation error for the compactly supported case with the one from  \cref{prop:truncboundcompact} (i). For instance, if the densities of $\mu_1$ and $\mu_2$ are regular, then so are those of the convolved measures (cf. Proposition 3 in \cite{polyanski2016wasserstein}).

\begin{remark}[Truncation to smooth uniformly convex sets]
    In \cref{prop:ComputationalGuarantees}, if $\interior\big(\supp(\mu_i*\eta_{\sigma})\big)$ is convex, then it can be approximated from the inside by a sequence of open smooth uniformly convex sets converging to $\interior\big(\supp(\mu_i)\big)$ in the Hausdorff distance \cite{Le2017eigenvalue}. In particular, if $\supp(\mu_i)$ is convex, then so is $\interior\big(\supp(\mu_i*\eta_{\sigma})\big)$. 
\end{remark}

\section{Entropic optimal transport}
\label{sec: EOT}

EOT is an efficiently-computable convexification of the original OT problem. The general machinery developed in Proposition \ref{prop: master proposition} enables deriving limit distribution results for empirical EOT, generalizing previously available statements to allows for dependent data. In addition, our theory provides new results on semiparametric efficiency of empirical EOT and consistency of the bootstrap estimate.

\subsection{Background} 
EOT regularizes the transportation cost by the Kullback-Leibler (KL) divergence \cite{schrodinger1931uber,leonard2014survey} as
\begin{equation}
    \EOT(\mu,\nu):=\inf_{\pi\in\Pi(\mu,\nu)}\int_{\RR^d \times \RR^d} c(x,y) d\pi(x,y) + \epsilon\, \mathsf{D}_{\mathsf{KL}}(\pi\|\mu\otimes\nu),\label{EQ:EOT}
\end{equation}
where $\epsilon>0$ and $\mathsf{D}_{\mathsf{KL}}(\mu\|\nu):=\int\log (d \mu/d \nu)d\mu$ if $\mu\ll\nu$ and $+\infty$ otherwise. We restrict our attention to the quadratic cost $c(x,y)=\|x-y\|^2/2$ and assume that $\epsilon = 1$; the corresponding EOT problem is abbreviated $\sS(\mu,\nu)=\sS_{\|\cdot\|^2/2}^1(\mu,\nu)$. The assumption that $\epsilon=1$ comes without loss of generality by a rescaling argument, since ${\sS_{\|\cdot\|^2/2}^{\epsilon}(\mu,\nu) = \epsilon \sS(\mu,\nu)(\mu_\epsilon,\nu_\epsilon)}$, where ${\mu_\epsilon=f_{\epsilon\,\sharp}\mu}$ for ${f_\epsilon(x)=\epsilon^{-1/2} x}$. 

To apply Proposition \ref{prop: master proposition} to empirical EOT, we rely on the duality theory for EOT, which states that $\sS(\mu,\nu)$ admits the dual formulation
\[
\sS(\mu,\nu) = \sup_{(\varphi,\psi) \in L^1(\mu) \times L^1(\nu)} \int_{\RR^d} \varphi d\mu + \int_{\RR^d} \psi d\nu - \int_{\RR^d\times\RR^d} e^{\varphi \oplus \psi - c} d\mu\otimes\nu + 1,
\]
where $(\varphi \oplus \psi)(x,y) = \varphi (x) + \psi (y)$. 
Assuming $\mu,\nu \in \calP_2 (\R^d)$, the supremum is attained by a pair $(\varphi,\psi)\in L^1(\mu)\times L^1 (\nu)$ satisfying 
\begin{equation}
\begin{split}
\int_{\RR^d} e^{\varphi(x)+\psi(y')-c(x,y')} d\nu(y') = 1 \quad \text{$\mu$-a.e. $x \in \R^d$,} \\
\int_{\RR^d} e^{\varphi(x')+\psi(y)-c(x',y)} d\mu(x') = 1 \quad \text{$\nu$-a.e. $y \in \R^d$.}
\end{split}
\label{eq: optimality}
\end{equation}
We refer to such $(\varphi,\psi)$ as \textit{optimal EOT potentials} (from $\mu$ to $\nu$ for $\varphi$ and vice versa for $\psi$). Optimal EOT potentials are unique $(\mu \otimes \nu)$-almost everywhere up to additive constants. Conversely, any $(\varphi,\psi)\in L^1(\mu)\times L^1 (\nu)$ that admit \eqref{eq: optimality} are optimal EOT potentials. See Section 1 in \cite{nutz2021entropic} and the references therein for details of the duality results for EOT. 

\subsubsection{Literature review}\label{SUBSUBSEC:EOT_lit} The entropic penalty transforms the OT linear optimization problem into a strongly convex one, allowing efficient computation via the Sinkhorn algorithm \cite{cuturi2013sinkhorn,altschuler2017near}. While EOT forfeits the metric and topological structure of $\Wp$,\footnote{Indeed, e.g., $\EOT(\mu,\mu)\neq 0$; while this can be fixed via centering EOT to obtain the so-called Sinkhorn divergence, it is still not a metric since it lacks the triangle inequality \cite{bigot2019central}.} it attains fast empirical convergence in certain cases. Specifically, empirical EOT converges as $n^{-1/2}$ for smooth costs and compactly supported distributions \cite{genevay2019sample}, or for the squared cost with sub-Gaussian distributions \cite{mena2019statistical}. 

Limit distributions~for EOT (and the Sinkhorn divergence) for $c(x,y)=\|x-y\|^p$ in the discrete support case were provided in \cite{bigot2019central,klatt2020empirical}. A CLT for EOT between sub-Gaussian distribution was first derived in \cite{mena2019statistical}, showing asymptotic normality of $\sqrt{n}\big(\mathsf{S}(\hat{\mu}_n,\nu)-\E\big[\mathsf{S}(\hat{\mu}_n,\nu_n)\big]\big)$ and its two-sample analogue. The main limitation of this result is that the centering term is the expected empirical EOT, which is undesirable because it does not enable performing inference for $\sS(\mu,\nu)$. This limitation was addressed in the recent preprint \cite{delbarrio22EOT} following arguments similar to \cite{delbarrio2019central} for $\mathsf{W}_2^2$. Namely, \cite{delbarrio22EOT} combine the aforementioned CLT from \cite{mena2019statistical} with a bias bound of the form $\E\big[\mathsf{S}(\hat{\mu}_n,\nu)\big] - \mathsf{S}(\mu,\nu) =o(n^{-1/2})$ (leveraging regularity of optimal EOT potentials between sub-Gaussian measures) to obtain the desired result. We complement this result by demonstrating asymptotic efficiency of empirical EOT and consistency of the bootstrap estimate, as well as allowing for dependent data in the CLT.

\subsection{Statistical analysis} We next state the CLT, asymptotic efficiency, and bootstrap consistency for empirical EOT. 
 
\begin{theorem}[CLT, efficiency, and bootstrap consistency for EOT]\label{THM:EOT_CLT}
Suppose that $\mu,\nu\in \cP(\RR^d)$ are sub-Gaussian. Let $(\varphi,\psi)$ be optimal EOT potentials for $(\mu,\nu)$. Then, the following hold. 
\begin{enumerate}
    \item[(i)] We have $\sqrt{n}\big(\sS(\hat{\mu}_n,\nu)-\sS(\mu,\nu)\big)  \stackrel{d}{\to} N\big(0,\mathfrak{v}_1^2 \big)$ with $\mathfrak{v}_1^2 = \Var_{\mu}(\varphi)$. The asymptotic variance $\mathfrak{v}_1^2$ coincides with the semiparametric efficiency bound for estimating $\sS(\cdot,\nu)$ at $\mu$. Finally, provided that  $\mathfrak{v}_1^2 > 0$, we have
    \[
\sup_{t \in \R} \left | \Prob^B \Big ( \sqrt{n}\big(\sS(\hat{\mu}_n^B,\nu)-\sS(\hat{\mu}_n,\nu)\big) \le t \Big ) - \Prob\big(N(0,\mathfrak{v}_1^2)) \le t\big) \right| \stackrel{\Prob}{\to} 0. 
\]
    \item[(ii)] We have $\sqrt{n}\big(\sS(\hat{\mu}_n,\hat{\nu}_n)-\sS(\mu,\nu)\big)  \stackrel{d}{\to} N\big(0,\mathfrak{v}_1^2 + \mathfrak{v}_2^2\big)$ where $\mathfrak{v}_1^2$ is as in (i) and $\mathfrak{v}_2^2 =\Var_{\nu}(\psi)$. The asymptotic variance $\mathfrak{v}_{1}^2 + \mathfrak{v}_2^2$ coincides with the semiparametric efficiency bound for  estimating $\sS(\cdot,\cdot)$ at $(\mu,\nu)$. Finally, provided that $\mathfrak{v}_{1}^2 + \mathfrak{v}_2^2  > 0$, we have
\[
\sup_{t \in \R} \left| \Prob^B \Big ( \sqrt{n}\big(\sS(\hat{\mu}_n^B,\hat{\nu}_n^B)-\sS(\hat{\mu}_n,\hat{\nu}_n)\big) \le t \Big ) 
-\Prob\big(N(0,\mathfrak{v}_1^2+\mathfrak{v}_2^2) \le t\big) \right| \stackrel{\Prob}{\to} 0. 
\]
\end{enumerate}
\end{theorem}

\begin{remark}[Comparison with \cite{delbarrio22EOT}] As mentioned in \cref{SUBSUBSEC:EOT_lit}, a CLT for one- and two-sample EOT was derived in Theorem 3.6 of \cite{delbarrio22EOT}. We rederive this result via a markedly different proof technique, relying on the unified framework from \cref{prop: master proposition}, which automatically also implies bootstrap consistency and asymptotic efficiency via \cref{cor: master} and \cref{prop: master proposition 2}, respectively. In addition, as \cref{prop: master proposition} does not assume i.i.d. data, the above result readily extends to dependent data, which falls outside the framework of  \cite{delbarrio22EOT}. For instance, suppose that $\{ X_{t} \}_{t \in \mathbb{Z}}$ is a stationary $\beta$-mixing process with compactly supported marginal distribution $\mu$. Then, by Theorem 1 in \cite{doukhan1995invariance}, 
\[
\sqrt{n}(\hat{\mu}_n - \mu) \stackrel{d}{\to} G \quad \text{in} \ \ell^\infty (\calF_\sigma),
\]
where $\calF_\sigma$ is the function class given in (\ref{eq: EOT function class}) ahead with $s=\max \{ \lfloor d/2 \rfloor + 1,2 \}$ and sufficiently large $\sigma > 0$, while $G$ is a tight centered Gaussian process in $\ell^\infty (\calF_\sigma)$ with covariance function $\Cov\big(G(f),G(g)\big) = \sum_{t \in \mathbb{Z}} \Cov\big(f(X_0),g(X_t)\big)$. Conclude from the proof of Theorem \ref{THM:EOT_CLT} that in this case
\[
\sqrt{n}\big(\sS(\hat{\mu}_n,\nu)-\sS(\mu,\nu)\big)  \stackrel{d}{\to} G(\varphi) \sim N\Big (0, \sum\nolimits_{t \in \mathbb{Z}} \Cov\big(\varphi (X_0),\varphi(X_t)\big) \Big). 
\]
\end{remark}

\section{Concluding Remarks}\label{sec: summary}
This work developed a unified framework for proving limit distribution results for empirical regularized OT distances, semiparametric efficiency of the plug-in empirical estimator, and consistency of the bootstrap.  As applications, we focused on three prominent OT regularization methods---smoothing, slicing, and entropic penalty---providing a comprehensive statistical treatment thereof. 
We closed existing gaps in the literature (e.g., a limit distribution theory for sliced $\Wp$) and provided several new results concerning empirical convergence rates, asymptotic efficiency, and bootstrap consistency. In particular, for the smooth Wasserstein distance, we explored compactly supported smoothing kernels, which were shown to inherit the structural and statistical properties of the well-studied Gaussian-smoothed framework.
The analysis of compactly supported kernels is motivated by computational considerations, as we demonstrated how to lift the efficient algorithm from \cite{vacher2021dimension} for computing $\mathsf{W}_2^2$ between smooth densities to the considered smooth OT distance.

Our framework is flexible and can treat a broad class of functionals, potentially well beyond the three specific examples considered herein. For instance, straightforward adaptations of our arguments for sliced $\Wp$ would yield limit distributions, efficiency, and consistency of the bootstrap of the projection-robust Wasserstein distance from \cite{lin2021projection}, when the projected subspace is of dimension $k\leq 3$ (indeed, the class of projected OT potentials is still Donsker in that case). Going forward, we also plan to explore applicability of the unified framework to empirical OT maps or certain functionals thereof (e.g., inner product with a smooth test function).

\section*{Acknowledgement}
The authors would like to thank Sloan Nietert for fruitful discussions on smooth Wasserstein distances with compactly supported kernels. They also would like to thank Tudor Manole for bringing their attention to Theorem 4 in \cite{manole2022minimax}.

\appendix

\section{Proofs for Sections \ref{sec: background} and \ref{sec: unified}}
\label{sec: proof master proposition}

\subsection{Proof of Lemma \ref{lem: H differentiability}}
The lemma follows from Proposition 2.1 in \cite{fang2019}. We include its proof for completeness.
Pick any $h \in \mathfrak{D}_0, t_n \to 0, h_n \to h$ in $\mathfrak{D}$ such that $\theta + t_n h_n \in \Theta$. For any subsequence $n'$ of $n$, there exists a further subsequence $n''$ such that (i) $t_{n''} > 0$ for all $n''$ or (ii) $t_{n''} < 0$ for all $n''$.  If (i) holds, then  $(\phi(\theta+t_{n''}h_{n''})-\phi(\theta))/t_{n''} \to \phi'(h)$, while if (ii) holds, then $(\phi(\theta+t_{n''}h_{n''})-\phi(\theta))/t_{n''} = - (\phi(\theta + (-t_{n''}) (-h)) - \phi(\theta))/(-t_{n''}) \to -\phi'(-h) = \phi'(h)$. Since the limit does not depend on the choice of subsequence, we have the result. \qed

\subsection{Proof of Proposition \ref{prop: master proposition}}
Let $\calP_{0,\epsilon} = \{ \nu \in \calP_0 : \| \nu - \mu \|_{\infty,\calF} < \epsilon \}$. Since $\calP_{0,\epsilon}$ is convex, $\bigcup_{t > 0} t (\calP_{0,\epsilon} - \mu) = \bigcup_{t > 0} t (\calP_0 - \mu)$, and $\mu_n \in \calP_{0,\epsilon}$ with inner probability approaching one, it suffices to prove the proposition with $\calP_0$ replaced by $\calP_{0,\epsilon}$, i.e., we may assume without loss of generality that $\delta: \calP_0 \to \mathfrak{E}$ is globally Lipschitz w.r.t. $\| \cdot \|_{\infty,\calF}$,
\[
\| \delta(\nu) - \delta(\nu') \|_{\mathfrak{E}} \le C \| \nu - \nu' \|_{\infty,\calF}, \quad \forall \nu,\nu' \in \calP_0
\]
for some constant $C< \infty$. 
We apply the extended functional delta method, Lemma \ref{lem: functional delta method}, by identifying $\calP_{0}$ as a subset of $\ell^\infty (\calF)$. Consider the map $\tau: \calP_{0} \ni \nu \mapsto (f \mapsto \int fd\nu)_{f \in \calF} \in \ell^\infty (\calF)$. Define the map $\bar{\delta}: \tau \calP_{0} \subset \ell^\infty (\calF) \to \mathfrak{E}$ by $\bar{\delta} (\tau \nu) = \delta (\nu)$ for $\nu \in \calP_0$. The map $\bar{\delta}$ is well-defined since, whenever $\tau \nu = \tau \nu'$, we have $\| \delta(\nu) - \delta(\nu') \|_{\mathfrak{E}} \le C \| \tau \nu - \tau \nu' \|_{\infty,\calF} = 0$, i.e., $\delta (\nu) = \delta(\nu')$. In what follows, we identify $\delta$ with $\bar{\delta}$. 

Under such identification, we apply Lemma \ref{lem: functional delta method} with  $\mathfrak{D} = \ell^\infty (\calF), \Theta = \calP_0, \theta = \mu$, and $\phi (\nu) = \delta (\nu)$ for $\nu \in \calP_0$. 
We will verify that $\nu \mapsto \delta (\nu)$ is Hadamard directionally differentiable at $\nu = \mu$ with derivative coinciding with $\delta_{\mu}'$ on $\calP_0-\mu$. This (essentially) follows from the fact that $\delta$ is Lipschitz continuous together with the fact that it is G\^{a}teaux directionally differentiable at $\nu = \mu$ with derivative $\delta_{\mu}'$; cf. \cite{shapiro1990}, p.~483--484. However, in our application, the G\^{a}teaux derivative is a priori defined only on $\calP_0-\mu$, and to extend $\delta_{\mu}'$ to the tangent cone $\calT_{\calP_0}(\mu)$, we require $\mathfrak{E}$ to be complete. 
For completeness, we provide its full proof. 

Extend $\delta_{\mu}'$ to $\bigcup_{t > 0}t(\calP_0 - \mu)$ as 
\[
\delta_{\mu}'(t(\nu-\mu)) =  t \delta_{\mu}'(\nu-\mu), \ t>0, \nu \in \calP_0.
\]
This extension is well-defined. Indeed, suppose that $t (\nu -\mu) = t'(\nu'-\mu)$ (as elements of $\ell^\infty (\calF)$) for some $t' > 0$ and $\nu' \in \calP_0$; then
\[
\begin{split}
t' \delta_{\mu}'(\nu'-\mu) &= t'\lim_{s \downarrow 0} \frac{\delta (\mu +st' (\nu'-\mu)) - \delta(\mu)}{st'} \\
&=\lim_{s \downarrow 0} \frac{\delta (\mu +st' (\nu'-\mu)) - \delta(\mu)}{s} \\
&=\lim_{s \downarrow 0} \frac{\delta (\mu +st (\nu-\mu)) - \delta(\mu)}{s}\\
&= t\lim_{s \downarrow 0} \frac{\delta (\mu +st (\nu-\mu)) - \delta(\mu)}{st}= t\delta_\mu'(\nu-\mu). 
\end{split}
\]
Pick any $m_n \in \bigcup_{t > 0} t(\calP_0 - \mu)$ with $m_n \to m$ in $\ell^\infty (\calF)$ and $t_n \downarrow 0$. Also, pick any $\epsilon' > 0$ and $m' \in \bigcup_{t > 0} t(\calP_0 - \mu)$ such that $\| m - m' \|_{\infty,\calF} < \epsilon'$. 
For sufficiently large $n$, $t_n m_n, t_nm' \in \calP_0 - \mu$ (as $\calP_0$ is convex), so that
\[
\begin{split}
\| \delta (\mu +t_n m_n) - \delta(\mu) - t_n\delta_{\mu}'(m')  \|_{\mathfrak{E}} &\le \| \delta (\mu +t_n m_n) - \delta(\mu+t_nm') \|_{\mathfrak{E}} \\
&\quad + \underbrace{\| \delta (\mu +t_n m') - \delta(\mu) - t_n\delta_{\mu}'(m') \|_{\mathfrak{E}}}_{=o(t_n)} \\
&\le C t_n \| m_n - m' \|_{\infty,\calF} + o(t_n) \\
&\le C \epsilon' t_n + o(t_n),
\end{split}
\]
which implies that
\[
\limsup_{n,n' \to \infty}\left \| \frac{\delta (\mu +t_n m_n) - \delta(\mu)}{t_n}  - \frac{\delta (\mu +t_{n'} m_{n'}) - \delta(\mu)}{t_{n'}} \right \|_{\mathfrak{E}} \le 2C \epsilon'.
\]
Since $\epsilon' > 0$ is arbitrary, $t_{n}^{-1}[\delta (\mu +t_n m_n) - \delta(\mu)]$ is Cauchy, and since $\mathfrak{E}$ is a Banach space, the limit 
\[
\lim_{n \to \infty}\frac{\delta (\mu +t_n m_n) - \delta(\mu)}{t_n}
\]
exists in $\mathfrak{E}$. Thus, we have proved that the map $\nu \mapsto \delta (\nu)$ from $\calP_0$ into $\mathfrak{E}$ is Hadamard directionally differentiable at $\nu=\mu$. 
The conclusion of the proposition then follows from the functional delta method. We note that the second claim that $G_\mu \in \calT_{\calP_0}(\mu)$ with probability one follows from the portmanteau theorem.  \qed

\subsection{Proof of Corollary \ref{cor: master}}
By Lemma 3.1 in \cite{van2008reproducing}, $\supp(G_\mu)$ is the closure in $\ell^\infty (\calF)$ of the reproducing kernel Hilbert space for $G_\mu$ and thus a vector subspace of $\ell^\infty (\calF)$. 
The second  claim follows by Lemma \ref{lem: H differentiability}. \qed

\subsection{Proof of Proposition \ref{prop: master proposition 2}}
By direct calculation, for a given $h \in \dot \calP_{0,\mu}$, the submodel $\{ \mu_t : 0 \le t < \epsilon' \}$ with $\mu_t = (1+th) \mu$ for sufficiently small $\epsilon' > 0$ is differentiable in quadratic mean at $t=0$ with score function $h$, i.e.,
\[
\int \left [ \frac{d\mu_t^{1/2} -d\mu^{1/2}}{t} - \frac{1}{2} hd\mu^{1/2} \right]^2 \to 0, \quad t \downarrow 0.
\]
Thus, $\dot \calP_{0,\mu}$ is the tangent space at $\mu$ 
corresponding to all such submodels. 

By assumption, we see that, for the above submodel $\mu_t$,
\[
\frac{\delta(\mu_t) - \delta(\mu)}{t} \to \delta_\mu'(h\mu).
\]
Observe that the $L^2(\mu)$ closure of $\dot \calP_{0,\mu}$ is $L^2_0 (\mu) = \{ h \in L^2(\mu) : \mu(h) = 0 \}$. 
Since $\calF$ is $\mu$-pre-Gaussian, the variance function $f \mapsto \Var_\mu (f)$ is bounded (cf. Setion 1.5 in \cite{van1996weak}), so the Cauchy-Schwarz inequality implies that the mapping 
\[
h \mapsto \big (f \mapsto \mu (fh)\big)
\]
is continuous from $L_0^2(\mu)$ into $\ell^\infty(\calF)$.
Thus, the mapping $h \mapsto \delta_\mu'(h\mu)$ is a continuous linear functional on $L_0^2(\mu)$, and the functional $\delta: \calP_0 \to \R$ is differentiable at $\mu$ in the sense of \cite[p.~363]{vanderVaart1998asymptotic} relative to $\dot \calP_{0,\mu}$ with derivative $h \mapsto \delta_\mu'(h\mu)$.  By Theorem 25.20 and Lemma 25.19  in \cite{vanderVaart1998asymptotic}, the semiparametric efficiency bound for estimating $\delta$ at $\mu$ is given by
\begin{equation}
\sup_{h \in L^2_0(\mu), h \ne 0} \frac{(\delta_\mu'(h\mu))^2}{\| h \|_{L^2(\mu)}^2}. 
\label{eq: efficiency bound}
\end{equation}
We will show that (\ref{eq: efficiency bound}) coincides with $\Var (\delta_\mu'(G_\mu))$. 

Let $\calC_u (\calF)$ be the space of uniformly continuous functions on $\calF$ w.r.t. the pseudometric $d_\mu (f,g) = \sqrt{\Var_\mu (f-g)}$. Then $\calF$ is totally bounded w.r.t. $d_\mu$ and $G_\mu \in \calC_u (\calF)$ with probability one (cf. Example 1.5.10 in \cite{van1996weak}). 
The completion $\bar{\calF}$ of $\calF$ w.r.t. $d_\mu$ is compact and denote the (unique) extension of $G_\mu$ to $\bar{\calF}$ by the same symbol. Also the restriction of $\delta_\mu'$ to $\calC_u(\calF)$ uniquely extends to $\calC(\bar{\calF})$ (the space of continuous functions on $\bar{\calF}$), which we denote by the same symbol $\delta_\mu'$. Then, by the Riesz representation theorem, there exists a signed Borel measure $\mathfrak{m}$ on $\bar{\calF}$ such that 
\[
\delta_\mu'(z) = \int_{\bar{\calF}} z(f) d\mathfrak{m}(f), \quad z \in \calC (\bar{\calF}). 
\]
Since $\bar{\calF}$ is compact, for any $k=1,2,\dots$, one can find disjoint sets $B_{k,j},j=1,\dots,N_k$ in $\bar{\calF}$ with radius at most $1/k$ such that $\bigcup_{j=1}^{N_k} B_{k,j} = \bar{\calF}$. Approximate $\mathfrak{m}$ by $\sum_{j=1}^{N_k} \mathfrak{m}(B_{k,j}) \delta_{f_{k,j}}$ where each $f_{k,j}$ is an arbitrary point in $B_{k,j}$. Then,
\[
\left | \delta_\mu'(z) - \sum_{j=1}^{N_k} \mathfrak{m}(B_{k,j}) z(f_{k,j}) \right | \le \sup_{d_\mu(f,g) \le 1/k}| z(f) - z(g)| \times | \mathfrak{m} | (\bar{\calF}),
\]
where $| \mathfrak{m} |$ is the total variation of $\mathfrak{m}$. For any $h \in L_0^2 (\mu)$, 
\[
| \mu(fg) - \mu(gh)| \le \| h \|_{L^2(\mu)} d_{\mu}(f,g),
\]
so that, uniformly in $h \in L_0^2(\mu)$ such that $\| h \|_{L^2(\mu)} = 1$, 
\[
\sum_{j=1}^{N_k} \mathfrak{m}(B_{k,j}) \mu(f_{k,j}h) \to \delta_\mu'(h\mu). 
\]
The left-hand side can be written as
$
\big \langle h, \sum_{j=1}^{N_k} \mathfrak{m}(B_{k,j})\big(f_{k,j}-\mu(f_{k,j})\big) \big \rangle_{L^2(\mu)},
$
the supremum of which w.r.t. $h \in L_0^2(\mu)$ such that $\| h \|_{L^2(\mu)} = 1$ is 
\[
\Big\| \sum_{j=1}^{N_k} \mathfrak{m}(B_{k,j})\big(f_{k,j}-\mu(f_{k,j})\big) \Big\|_{L^2(\mu)}.
\]
Thus,
\[
\begin{split}
\sup_{h \in L^2_0(\mu), h \ne 0} \frac{(\delta_\mu'(h\mu))^2}{\| h \|_{L^2(\mu)}^2} &=\sup_{h \in L^2_0(\mu), \| h \|_{L^2(\mu)} = 1} (\delta_\mu'(h\mu))^2 = \Big ( \sup_{h \in L^2_0(\mu), \| h \|_{L^2(\mu)} = 1} \delta_\mu'(h\mu) \Big )^2 \\
&=
\lim_{k \to \infty}\Big \| \sum_{j=1}^{N_k} \mathfrak{m}(B_{k,j})\big(f_{k,j}-\mu(f_{k,j})\big) \Big \|_{L^2(\mu)}^2. 
\end{split}
\]
On the other hand,
\[
\delta_{\mu}'(G_\mu) = \lim_{k \to \infty}\sum_{j=1}^{N_k} \mathfrak{m}(B_{k,j}) G_\mu(f_{k,j}) \\
=\lim_{k \to \infty} G_\mu \Big ( \sum_{j=1}^{N_k} \mathfrak{m}(B_{k,j}) f_{k,j}\Big)
\]
a.s., where the second equality follows from the fact that $G_\mu$ is a $\mu$-Brownian bridge. 
This implies that 
\[
\begin{split}
\Var (\delta_{\mu}'(G_\mu)) &= \lim_{k \to \infty} \Var \Bigg ( G_\mu \Big ( \sum_{j=1}^{N_k} \mathfrak{m}(B_{k,j}) f_{k,j}\Big)\Bigg ) \\
&= \lim_{k \to \infty}\Big \| \sum_{j=1}^{N_k} \mathfrak{m}(B_{k,j})\big(f_{k,j}-\mu(f_{k,j})\big) \Big \|_{L^2(\mu)}^2.
\end{split}
\]
Conclude that (\ref{eq: efficiency bound}) coincides with $\Var \big(\delta_\mu'(G_\mu)\big)$. \qed


\begin{remark}[Relation with Theorem 3.1 in \cite{van1991efficiency}]
\label{rem: van der vaart}
Consider the assumption of Proposition \ref{prop: master proposition} with $\mathfrak{E}=\R$ and assume that the function class $\calF$ is $\mu$-Donsker and the derivative $\delta_\mu'$ is linear on any subspace of $\calT_{\calP_0}(\mu)$. 
If, in addition, $(1+th)\mu \in \calP_0$ for sufficiently small $t > 0$ for any $h \in \dot \calP_{0,\mu}$, then the fact that $\Var (\delta_\mu'(G_\mu))$ (with $G_\mu$ being a $\mu$-Brownian bridge) agrees with the semiparametric efficiency bound for estimating $\delta$ at $\mu$ can be deduced from Theorem 3.1 in \cite{van1991efficiency}. Indeed, recall the map $\tau: \calP_{0} \ni \nu \mapsto  (f \mapsto \nu(f)) \in \ell^\infty (\calF)$ and $\bar{\delta}: \tau \calP_{0,\epsilon} \to \R$ defined by $\bar{\delta} (\tau \nu) = \delta (\nu)$. Then, the map $\tau$ is differentiable at $\mu$ relative to $\dot \calP_{0,\mu}$ with derivative $\dot{\tau}_\mu:L^2 (\mu) \to \ell^\infty (\calF)$ given by $\dot{\tau}_\mu (h) = (f \mapsto \mu(fh))$, and the empirical distribution $\hat{\mu}_n = n^{-1}\sum_{i=1}^n \delta_{X_i}$ is asymptotically efficient for estimating $\tau$ at $\mu$ tangentially to $\dot \calP_{0,\mu}$; see Section 3.11 in \cite{van1996weak} for the result and its precise meaning. Theorem 3.1 in \cite{van1991efficiency} shows that, if $\bar\delta$ is Hadamard differentiable at $\tau \mu$ tangentially to $\overline{\dot \tau_\mu (\dot \calP_{0,\mu})}^{\ell^\infty (\calF)}$, then $\delta (\hat{\mu}_n) = (\bar\delta \circ \tau)(\hat{\mu}_n)$ is asymptotically efficient, implying that $\Var (\bar\delta_{\tau \mu}' (G_\mu))$ agrees with the semiparametric efficiency bound. Identifying $\delta$ with $\bar\delta$ and $\mu$ with $\tau \mu$ as in the proof of Proposition \ref{prop: master proposition}, we have $\Var (\bar\delta_{\tau \mu}' (G_\mu)) = \Var (\delta_{\mu}' (G_\mu))$. In our notation, $\dot \tau_\mu (\dot \calP_{0,\mu})=\{ h\mu : h \in \dot \calP_{0,\mu} \}$, and from the assumption that $(1+th)\mu \in \calP_0$ for sufficiently small $t > 0$ for any $h \in \dot \calP_{0,\mu}$, we see that $\dot \tau_\mu (\dot \calP_{0,\mu}) \subset \calT_{\calP_0}(\mu)$, and as $\calT_{\calP_0}(\mu)$ is closed in $\ell^\infty (\calF)$, it contains  $\overline{\dot \tau_\mu (\dot \calP_{0,\mu})}^{\ell^\infty (\calF)}$. Finally, as the derivative $\bar \delta'_\mu$ is linear on $\overline{\dot \tau_\mu (\dot \calP_{0,\mu})}^{\ell^\infty (\calF)}$ by assumption, $\bar\delta$ is Hadamard differentiable at $\tau \mu$ tangentially to $\overline{\dot \tau_\mu (\dot \calP_{0,\mu})}^{\ell^\infty (\calF)}$ by Lemma \ref{lem: H differentiability}.
\end{remark}

\subsection{Proof of \cref{cor: efficiency two sample}}
Fix $h_1 \oplus h_2 \in \dot \calP_{0,\mu} \oplus \dot \calP_{0,\nu}$. Observe that the submodel $\{ \mu_t \otimes \nu_t : 0 \le t < \epsilon'\}$ with $\mu_t = (1+th_1)\mu$ and $\nu_t = (1+th_2)\nu$ for sufficiently small $\epsilon' > 0$ is differentiable in quadratic mean at $t = 0$ with score function $h_1 \oplus h_2$ (note: $(1+th_1(x)) (1+th_2(y)) = 1+t(h_1(x)+h_2(y)) + o(t)$).  

Further, $h_1 \oplus h_2 \mapsto \delta_\mu'(h_1\mu) + \delta_{\nu}'(h_2\nu)$ is a continuous linear functional on $(\dot \calP_{0,\mu} \oplus \dot \calP_{0,\nu}, \| \cdot \|_{L^2(\mu \otimes \nu)})$, and thus the semiparametric efficiency bound is given by
\[
\sup_{h_1 \oplus h_2 \in \dot \calP_{0,\mu} \oplus \dot \calP_{0,\nu}, h_1 \oplus h_2 \ne 0} \frac{(\delta_\mu'(h_1\mu) + \delta_{\nu}'(h_2\nu))^2}{\| h_1\oplus h_2 \|_{L^2(\mu \otimes \nu)}^2}. 
\]
The rest of the proof is analogous to \cref{prop: master proposition 2}, upon observing that $\calF$ is totally bounded w.r.t. $d_{\mu,\nu} (f,g) = \sqrt{\Var_{\mu}(f-g)} + \sqrt{\Var_\nu(f-g)}$ (if a set $T$ is totally bounded for two pseudometrics $\mathsf{d}_1$ and $\mathsf{d}_2$, then $T$ is totally bounded for $\mathsf{d} = \mathsf{d}_1+\mathsf{d}_2$), and $G_\mu$ and $G_\nu$ have $d_{\mu,\nu}$-uniformly continuous paths, combined with the fact that, for $(f,g) \in L^2(\mu) \times L^2(\nu)$ and $(h_1,h_2) \in \dot \calP_{0,\mu} \times \dot \calP_{0,\nu}$, we have
\[
\mu (h_1 f) + \nu(h_2g) = (\mu \otimes \nu)\big((h_1 \oplus h_2) (f \oplus g) \big),
\]
and the supremum of the right-hand side w.r.t. $(h_1,h_2) \in \dot \calP_{0,\mu} \times \dot \calP_{0,\nu}$ such that $\| h_1\oplus h_2 \|_{L^2(\mu \otimes \nu)}^2 = 1$ is $\sqrt{\Var_{\mu \otimes \nu} (f \oplus g)} = \sqrt{\Var_\mu (f) + \Var_\nu(g)}$.
\qed

\section{Proofs for Section \ref{sec: slicing}}
\label{sec: proof}

\subsection{Proofs of Theorems \ref{thm: limit distribution SWp} and \ref{thm: limit distribution MSWp}}

\subsubsection{Preliminary results}

The section presents preliminary results needed for the derivation of Theorem \ref{thm: limit distribution SWp}, followed by their proofs. In what follows, we fix $1 < p < \infty$.

\begin{lemma}[OT potentials]
\label{lem: OT potential}
The following hold. 
\begin{enumerate}
    \item[(i)] Suppose $\mu, \nu \in \calP(\R^d)$ are supported in $B(0,M) = \{ x : \|x\| < M \}$ for some $M > 0$ (i.e, $\supp (\mu), \supp (\nu) \subset B(0,M)$). Then, there exists an OT potential $\varphi$ from $\mu$ to $\nu$ for $\Wp$ such that $|\varphi(x)-\varphi(x')| \le C_{p,M}\| x-x' \|$ for all $x,x' \in B(0,M)$, where $C_{p,M} < \infty$ is a constant that depends only on $p,M$.
    \item[(ii)] If $\mu_0,\mu_1,\nu \in \calP(\R^d)$ are supported in $B(0,M)$ for some $M > 0$, then 
    \[
\big | \Wp^p(\mu_1,\nu) - \Wp^p(\mu_0,\nu) \big| \le  C_{p,M} \mathsf{W}_1(\mu_0,\mu_1). 
\]
\end{enumerate}
\end{lemma}

\begin{proof}
\textbf{Part (i).}  Since $\supp(\mu), \supp (\nu) \subset B(0,M)$, there exists an OT potential $\varphi$ such that 
\[
\varphi (x) = \inf_{y \in B(0,M)} [ \|x-y\|^p - \varphi^c(y)], \ x \in B(0,M). 
\]
Indeed, regard $\mu$ and $\nu$ as probability measures on $B(0,M)$, and pick any OT potential $\varphi_M$ from $\mu$ to $\nu$ as probability measures on $B(0,M)$. This $\varphi_{M}$ is defined only on $B(0,M)$ so extend it to $\R^d$ by setting $\varphi (x) = \varphi_M(x)$ if $x \in B(0,M)$ and $\varphi (x) = -\infty$ if $x \notin B(0,M)$. Observe that
\[
\begin{split}
\varphi^c(y) &= \inf_{x \in \R^d}[\|x-y\|^p - \varphi(x)] \\
&= \inf_{x \in B(0,M)}[\|x-y\|^p - \varphi(x)] \\
&= \inf_{x \in B(0,M)}[\|x-y\|^p - \varphi_M(x)] = \varphi_M^c (y). 
\end{split}
\]
Thus, whenever $x \in B(0,M)$, 
\[
\begin{split}
\varphi(x) &= \varphi_{M}(x) = \varphi_{M}^{cc}(x) \\
&= \inf_{y \in B(0,M)}[\|x-y\|^p - \varphi_M^c(y)] \\
&=\inf_{y \in B(0,M)}[\|x-y\|^p - \varphi^c(y)].
\end{split}
\]

Now, for $x, x' \in B(0,M)$,
\[
|\varphi (x) - \varphi(x')| \le \sup_{y \in B(0,M)} \big | \|x-y\|^p - \|x'-y\|^p\big | \le p2^{p-1}M^{p-1}\| x - x' \|. 
\]
\medskip
\textbf{Part (ii).} Let $\varphi_i$ denote an OT potential from $\mu_i$ to $\nu$. Observe that
\[
\begin{split}
 \Wp^p(\mu_1,\nu) &\ge \int \varphi_0 d\mu_1 + \int \varphi_0^c d\nu \\
 &=\int \varphi_0 d\mu_0 + \int \varphi_0^c d\nu + \int \varphi_0 d(\mu_1-\mu_0) \\
 &=\Wp^p(\mu_0,\nu) + \int \varphi_0 d(\mu_1-\mu_0). 
 \end{split}
\]
Likewise, we have
\[
\begin{split}
 \Wp^p(\mu_1,\nu)  &= \int \varphi_1 d\mu_1 + \int \varphi_1^c d\nu \\
 &=\int \varphi_1 d\mu_0 + \int \varphi_1^c d\nu + \int \varphi_1 d(\mu_1-\mu_0) \\
 &\le \Wp^p(\mu_0,\nu) + \int \varphi_1 d(\mu_1-\mu_0). 
 \end{split}
\]
The conclusion follows from the fact that $\varphi_i, i=1,2$ are $C_{p,M}$-Lipschitz on $B(0,M)$ and any Lipschitz function on $B(0,M)$ can be extended to $\R^d$ without changing the Lipschitz constant (cf. \cite{dudley2002}, Theorem 6.1.1), followed by the Kantorovich-Rubinstein duality.  
\end{proof}

Part (ii) of the preceding lemma implies the following corollary. 

\begin{corollary}
\label{lem: Lipschitz}
If $\mu_0,\mu_1,\nu \in \calP(\R^d)$ are supported in $B(0,M)$ for some $M > 0$, then
\[
\left|\SWp^p(\mu_1,\nu) - \SWp^p(\mu_0,\nu) \right| \bigvee \left|\MSWp^p(\mu_1,\nu) - \MSWp^p(\mu_0,\nu) \right| 
\le C_{p,M} \MSWone (\mu_0,\mu_1).
\]
\end{corollary}

\begin{proof}
This follows from Part (ii) of the preceding lemma, the definitions of $\SWp$ and $\MSWp$, and the fact that $\ptheta_{\sharp} \mu_0, \ptheta_{\sharp} \mu_1, \ptheta_{\sharp} \nu$ are supported in $[-M,M]$. 
\end{proof}

\begin{lemma}[Regularity of sliced OT potential]
\label{lem: joint measurability}
Let $\mu,\nu \in \calP_p(\R^d)$ and assume that $\mu$ is absolutely continuous with convex support. Set $\calX = \inte(\supp(\mu))$.
For each $\theta \in \unitsph$, let $\varphi^\theta$ be an OT potential from $\ptheta_\sharp \mu$ to $\ptheta_\sharp \nu$ for $\Wp$. Pick and fix any $x_0 \in \calX$.
Then the following hold.
\begin{enumerate} 
    \item[(i)] The map $\calX \times \unitsph \ni (x,\theta) \mapsto \varphi^\theta (\theta^{\intercal}x)-\varphi^\theta(\theta^{\intercal}x_0)$ is measurable in $x$ and continuous in $\theta$, and thus jointly measurable in $(x,\theta)$.  
    \item[(ii)] The map $\calX \times \unitsph \ni (x,\theta) \mapsto \varphi^\theta (\theta^{\intercal}x)-\varphi^\theta(\theta^{\intercal}x_0)$ is unique, in the sense that if, for each $\theta \in \unitsph$, $\tilde{\varphi}^\theta$ is another OT potential from $\ptheta_\sharp \mu$ to $\ptheta_\sharp \nu$ for $\Wp$, then $\varphi^\theta (\theta^{\intercal}x)-\varphi^\theta(\theta^{\intercal}x_0) = \tilde{\varphi}^\theta (\theta^{\intercal}x)-\tilde{\varphi}^\theta(\theta^{\intercal}x_0)$ for all $(x,\theta) \in \calX \times \unitsph$. 
    \item[(iii)] If, in addition, $\supp(\mu)$ and $\supp(\nu)$ are contained in $B(0,M)$ for some $M > 0$, then  $|\varphi^\theta (\theta^{\intercal}x)-\varphi^\theta(\theta^{\intercal}x_0)| \le C_{p,M}$ for all $(x,\theta) \in \calX \times \unitsph$, where $C_{p,M}$ is given in Lemma \ref{lem: OT potential}. In particular, for any finite sighed Borel measure $\rho$ on $\R^d$ with total mass $0$, the map $\unitsph \ni \theta \mapsto \int \varphi^\theta (\theta^{\intercal}x) d\rho (x)$ is continuous. 
\end{enumerate}
    \end{lemma}

\begin{proof}
\textbf{Part (i).}
Measurability of the map $x \mapsto \varphi^\theta(\theta^{\intercal}x)$ follows by measurability of the map $t \mapsto \varphi^\theta (t)$ and continuity of $x \mapsto \theta^{\intercal}x$. We shall prove continuity of the map $\theta \mapsto \varphi^\theta(\theta^{\intercal}x)-\varphi^\theta(\theta^{\intercal}x_0)$.
 By convexity of the support and absolute continuity of $\mu$, we have $\mu (\calX) =1$, as the boundary of any convex set has Lebesgue measure zero. Let $\calX_{\theta} = \{ \theta^{\intercal}x : x \in \calX \}$. Then $\calX_\theta$ is an open interval and its closure coincides with $\supp(\ptheta_\sharp \mu)$; see Lemma \ref{lem: support} below. Pick any $x \in \calX$ and let $\theta_n  \to \theta$ in $\unitsph$.
Then, for all large $n$,  $\theta_n^{\intercal}x$ lies in a compact neighborhood of $\theta^{\intercal}x$ contained in $\calX_\theta$.
Since $\mathfrak{p}^{\theta_n}_\sharp \mu \stackrel{w}{\to} \ptheta_\sharp \mu$, by Theorem 3.4 in \cite{delbarrio2021}, $\varphi^{\theta_n} - \varphi^{\theta_n} (\theta^{\intercal}x_0) \to \varphi^\theta - \varphi^\theta (\theta^{\intercal}x_0)$ uniformly on each compact subset of $\calX_\theta$. 
This implies that $\varphi^{\theta_n} (\theta_n^{\intercal}x) - \varphi^{\theta_n} (\theta_n^{\intercal}x_0) \to \varphi^\theta (\theta^{\intercal}x) - \varphi^\theta (\theta^{\intercal}x_0)$. The last claim follows from, e.g., Lemma
4.51 in \cite{aliprantis2006infinite}.

\medskip

\textbf{Part (ii).} This follows from Corollary 2.7 in \cite{delbarrio2021}. 

\medskip

\textbf{Part (iii).} The first claim follows from Lemma \ref{lem: OT potential} Part (i) (note: $\supp(\ptheta_{\sharp} \mu)$ and $\supp(\ptheta_{\sharp} \nu)$ are contained in $[-M,M]$) together with the uniqueness result from Part (ii) of this lemma. The second claim follows from the fact that $\int \varphi^{\theta}(\theta^{\intercal}x) d\rho(x) = \int [ \varphi^{\theta}(\theta^{\intercal}x) - \varphi^{\theta}(\theta^{\intercal}x_0)] d\rho(x)$ and the dominated convergence theorem. 
\end{proof} 

We shall prove the following technical result used in the preceding proof. 

\begin{lemma}[Support of projected distribution]
\label{lem: support}
Let $\mu \in \calP(\R^d)$ be absolutely continuous with  convex support and let $\mathbb{A} \in \R^{d \times m}$ be of rank $m$ (thus $d \le m$). Set $\calX = \inte(\supp(\mu))$. Define $\nu \in \calP(\R^m)$ by $\nu (B) = \mu (\{ x : \mathbb{A}x \in B \})$.  Then $\mathbb{A}\calX= \{ \mathbb{A}x : x \in \calX \} \subset \R^m$ is open and convex, and its closure coincides with $\supp(\nu)$. 
\end{lemma}

\begin{proof}
Since the mapping $x \mapsto \mathbb{A}x$ is open, the set $\mathbb{A}\calX$ is open. Convexity of $\mathbb{A}\calX$ is trivial. Pick any $x \in \calX$ and any open neighborhood $N_y$ of $y = \mathbb{A}x$. Then, since $\mathbb{A}^{-1}N_y$ is open and contains $x$, $\nu (N_y) = \mu(\mathbb{A}^{-1}N_y) > 0$, which implies that $\mathbb{A}\calX \subset \supp(\nu)$. Since clearly $\nu (\mathbb{A} \calX) = 1$, the closure of $\mathbb{A}\calX$ coincides with $\supp(\nu)$. 
\end{proof}

\begin{lemma}[G\^{a}teaux derivatives of $\SWp$ and $\MSWp$]
\label{lem: directional derivative SWp}
Suppose that $\mu_0,\mu_1,\nu \in \calP(\R^d)$ are compactly supported, that $\mu_0$ is absolutely continuous, and that the support of $\mu_0$ is convex and and its interior has $\mu_1$-probability one, i.e., $\mu_1 (\inte (\supp(\mu_0))) = 1$.  For every $\theta \in \unitsph$, let $\varphi_0^{\theta}$ be an OT potential from $\ptheta_\sharp \mu_0$ to $\ptheta_\sharp \nu$ for $\mathsf{W}_p$. Let $\rho = \mu_1-\mu_0$.  Then we have
\begin{align}
\frac{d}{dt^+} \SWp^p(\mu_0+t\rho,\nu) \Big|_{t=0} &=\int_{\unitsph} \int \varphi^{\theta}_0 (\theta^{\intercal}x)d\rho(x) d\sigma(\theta),  \label{eq: derivative SWp} \\
\frac{d}{dt^+} \MSWp^p(\mu_0+t\rho,\nu) \Big|_{t=0} &= \sup_{\theta \in \mathfrak{S}_{\mu,\nu}} \int \varphi^{\theta}_0 (\theta^{\intercal}x)d\rho(x). \label{eq: derivative MSWp}
\end{align}
The right-hand sides of (\ref{eq: derivative SWp}) and (\ref{eq: derivative MSWp}) are well-defined and finite. 
\end{lemma}

\begin{proof}
We first prove (\ref{eq: derivative SWp}). 
Let $\varphi_t^\theta$ denote an OT potential from $\ptheta_\sharp\mu_t$ to $\ptheta_\sharp\nu$ with $\mu_t = \mu_0+t\rho$. We may assume without loss of generality that $\varphi_t^\theta (\theta^{\intercal}x_0) = 0$, where $x_0$ is an arbitrary (fixed) interior point of $\supp(\mu_0)$. Then,
\[
\begin{split}
\SWp^p(\mu_t,\nu) &\ge \int_{\unitsph} \int \varphi_0^{\theta} (\theta^{\intercal}x) \, d\mu_t(x) \, d\sigma(\theta) +  \int_{\unitsph} \int [\varphi_0^{\theta}]^c (\theta^{\intercal}y) \,d\nu (y)\,d\sigma(\theta) \\
&=\SWp^p(\mu_0,\nu) + t \int_{\theta \in \unitsph} \int \varphi_0^{\theta} (\theta^{\intercal}x) \, d\rho(x) \, d\sigma(\theta), \quad \text{and} \\
\SWp^p(\mu_t,\nu) &\le \SWp^p(\mu_0,\nu) + t \int_{\theta \in \unitsph} \int \varphi_t^{\theta} (\theta^{\intercal}x) \, d\rho(x) \, d\sigma(\theta).
\end{split}
\]
The first inequality implies that
\[
\liminf_{t \downarrow 0} \frac{\SWp^p(\mu_t,\nu) - \SWp^p(\mu_0,\nu)}{t} \ge \int_{\unitsph} \int \varphi^{\theta}_0 (\theta^{\intercal}x)d\rho(x) d\sigma(\theta). 
\]

To prove the upper bound, we note that $\varphi_t^\theta (\theta^{\intercal} \cdot)$ is uniformly bounded on $\supp(\mu_0)$ by Lemma \ref{lem: joint measurability} Part (iii),
and by Theorem 3.4 in \cite{delbarrio2021}, $\varphi_t^\theta (\theta^{\intercal}x) \to \varphi_0^\theta (\theta^{\intercal}x)$ as $t \downarrow 0$ for each $x \in \inte(\supp(\mu_0))$ and $\theta \in \unitsph$. Thus, by the dominated convergence theorem, we have
\begin{align*}
\limsup_{t \downarrow 0} \frac{\SWp^p(\mu_t,\nu) - \SWp^p(\mu_0,\nu)}{t} \le \int_{\unitsph} \int \varphi^{\theta}_0 (\theta^{\intercal}x)d\rho(x) d\sigma(\theta).
\end{align*}

Next we prove  (\ref{eq: derivative MSWp}). To save space, we write $\int \varphi_t^\theta (\theta^{\intercal}x)d\mu_t (x)= \int \varphi_t^\theta d\mu_t$ etc. 
Observe that for $t \in (0,1)$,
\begin{align*}
    &\frac{1}{t} \left [ \MSWp^p(\mu_t, \nu) - \MSWp^p(\mu_0, \nu) \right ] \\
    &= \frac{1}{t} \left [ \sup_{\theta \in \unitsph} \left( \int \varphi_t^\theta \,d\mu_t + \int [\varphi_t^{\theta}]^c d\nu \right) -  \sup_{\theta \in \unitsph} \left ( \int \varphi_0^\theta \, d\mu_0 + \int [\varphi_0^\theta]^c \, d\nu \right ) \right ]\\
    &\leq \sup_{\theta \in \unitsph} \left \{ \int \varphi_t^\theta \, d\rho + \frac{1}{t} \left ( \int \varphi_0^\theta \, d\mu_0 + \int [\varphi_0^\theta]^c \, d\nu \right ) \right \} \\
    &\qquad- \frac{1}{t}  \sup_{\theta \in \unitsph} \left ( \int \varphi_0^\theta \, d\mu_0 + \int [\varphi_0^\theta]^c \, d\nu \right ).
\end{align*}
Let $S_\epsilon = \{ \theta \in \unitsph : \mathsf{W}_p^p(\ptheta_\sharp \mu_0, \ptheta_\sharp \nu) \geq \MSWp^p(\mu_0, \nu) - \epsilon \}$. Then, for $\theta \notin S_\epsilon$,
\begin{align*}
 \int \varphi_t^\theta \, d\rho +  &  \frac{1}{t} \left ( \int \varphi_0^\theta \, d\mu_0 + \int [\varphi_0^\theta]^c \, d\nu \right )  \\
&- \frac{1}{t}  \sup_{\vartheta \in \unitsph} \left ( \int \varphi_0^\vartheta \, d\mu_0 + \int [\varphi_0^\vartheta]^c \, d\nu \right ) \leq \int \varphi_t^\theta \, d\rho - \frac{\epsilon}{t},
\end{align*}
which tends to $-\infty$ uniformly over $\theta$ as $t \downarrow 0$.
Hence, for every $\epsilon > 0$, we have
\[
\limsup_{t \downarrow 0} \frac{1}{t} \left [ \MSWp^p(\mu_t, \nu) - \MSWp^p(\mu_0, \nu) \right ] \leq \limsup_{t \downarrow 0} \sup_{\theta \in S_\epsilon} \int \varphi_t^\theta \, d\rho.
\]
Taking $\epsilon \downarrow 0$ above, since the right-hand side is decreasing in $\epsilon$, we have
\[
\limsup_{t \downarrow 0} \frac{1}{t} \left [ \MSWp^p(\mu_t, \nu) - \MSWp^p(\mu_0, \nu) \right ] \leq \lim_{\epsilon \downarrow 0} \limsup_{t \downarrow 0} \sup_{\theta \in S_\epsilon} \int \varphi_t^\theta \, d\rho.
\]

We will show that $\int \varphi_t^\theta \, d\rho \to \int \varphi_0^\theta \, d\rho$ uniformly over $\theta \in \unitsph$ as $t \downarrow 0$. To this end, by compactness of $\unitsph$ and continuity of $\int \varphi_0^\theta \, d\rho$ in $\theta$, it suffices to show that for any sequences $t_n \downarrow 0$ and $\theta_n \to \theta$, $\int \varphi_{t_n}^{\theta_n} \, d\rho \to \int \varphi_0^\theta \, d\rho$.\footnote{This can be verified as follows. Suppose on the contrary that $c:=\liminf_{t \downarrow 0} \sup_{\theta} |\int (\varphi_t^\theta -\varphi_0^\theta) d\rho| > 0$. Then there exist sequences $t_n \downarrow 0$ and $\theta_n \in \unitsph$ such that $|\int (\varphi_{t_n}^{\theta_n} -\varphi_0^{\theta_n}) d\rho| \to c$. Since $\unitsph$ is compact, there exists a subsequence of $\theta_n$ converging to some $\theta \in \unitsph$. Thus, along with the subsequence, we have $|\int \varphi_{t_n}^{\theta_n}d\rho - \int \varphi_0^\theta d\rho| \to c$, a contradiction.} Since, for any bounded continuous function $g$ on $\R$, 
\[
\begin{split}
\int g d(\mathfrak{p}^{\theta_n}_{\sharp}\mu_{t_n})&= \int g(\theta_n^{\intercal}x)d\mu_{t_n}(x) \\
&= \int g(\theta_n^{\intercal}x)d\mu_{0}(x) + t_n \int g(\theta_n^{\intercal}x)d\rho (x)\\
&=\int g(\theta^{\intercal}x)d\mu_{0} (x)+ o(1) \quad (\text{as $g$ is bounded and continuous}) \\
&=\int gd(\ptheta_{\sharp}\mu_{0}) + o(1), 
\end{split}
\]
we have $\mathfrak{p}^{\theta_n}_{\sharp}\mu_{t_n} \stackrel{w}{\to} \mathfrak{p}^{\theta}_{\sharp}\mu_{0}$. Thus, by Theorem 3.4 in \cite{delbarrio2021}, $\varphi_{t_n}^{\theta_n}  \to \varphi_0^\theta$ uniformly on each compact subset of $\inte (\supp(\mathfrak{p}^{\theta}_{\sharp}\mu_{0}))$, which implies that $\varphi_{t_n}^{\theta_n}(\theta_n^{\intercal}x) \to \varphi_0^\theta (\theta^{\intercal}x)$ for each $x \in \inte (\supp(\mu_0))$. Since $\varphi_t^{\vartheta}$ is bounded on $\inte(\supp(\mu_0))$ uniformly in  $(t,\vartheta)$, we conclude from the dominated convergence theorem that $\int \varphi_{t_n}^{\theta_n} \, d\rho \to \int \varphi_0^\theta \, d\rho$. This yields that
\[
\limsup_{t \downarrow 0} \frac{1}{t} \left [ \MSWp^p(\mu_t, \nu) - \MSWp^p(\mu_0, \nu) \right ] \leq \lim_{\epsilon \downarrow 0} \sup_{\theta \in S_\epsilon} \int \varphi_0^\theta \, d\rho.
\]

To prove the reverse inequality, we note that
\begin{align*}
    \frac{1}{t} \left [ \MSWp^p(\mu_t, \nu) - \MSWp^p(\mu_0, \nu) \right ] &\geq \frac{1}{t} \left [ \sup_{\theta \in \unitsph} \left( \int \varphi_0^\theta \,d\mu_t + \int [\varphi_0^{\theta}]^c d\nu \right) \right.\\
    &\qquad\qquad - \left. \sup_{\theta \in \unitsph} \left ( \int \varphi_0^\theta \, d\mu_0 + \int [\varphi_0^\theta]^c \, d\nu \right ) \right ].
\end{align*}
For $\theta \in S_\epsilon$,
\[ 
\left( \int \varphi_0^\theta \,d\mu_t + \int  [\varphi_0^{\theta}]^c d\nu \right) -  \sup_{\theta \in \unitsph} \left ( \int \varphi_0^\theta \, d\mu_0 + \int [\varphi_0^\theta]^c \, d\nu \right ) \geq t \int \varphi_0^\theta \, d\rho - \epsilon.
\]
Choosing $\epsilon = \epsilon_t = t^2$, we have
\[
\limsup_{t \downarrow 0} \frac{1}{t} \left [ \MSWp(\mu_t, \nu) - \MSWp(\mu_0, \nu) \right ] \geq \lim_{\epsilon \downarrow 0} \sup_{\theta \in S_\epsilon} \int \varphi_0^\theta \, d\rho.
\]
Since $\unitsph$ is compact and $\int \varphi_0^\theta \, d\rho$ is continuous in $\theta$, we further have 
\[
\lim_{\epsilon \downarrow 0} \sup_{\theta \in S_\epsilon} \int \varphi_0^\theta \, d\rho = \sup_{\theta \in \mathfrak{S}_{\mu,\nu}} \int \varphi_0^\theta \, d\rho.
\]
This completes the proof.
\end{proof}

Corollary \ref{lem: Lipschitz} shows that the mapping $\mu \mapsto \SWp(\mu,\nu)$ and $\mu \mapsto \MSWp(\mu,\nu)$ are Lipschitz continuous w.r.t. $\MSWone$. By the Kantorovich-Rubinstein duality, $\MSWone$ can be represented as $\MSWone (\mu,\nu) = \| \mu - \nu \|_{\infty, \calF}$ with $\cF = \{\varphi \circ \ptheta : \theta \in \unitsph,  \varphi \in \mathsf{Lip}_{1,0}(\R)\}$. It is not difficult to show that $\calF$ is $\mu$-Donsker if $\mu$ is compactly supported. 
For a later purpose, we allow $\mu$ to be unbounded in the following lemma. 

\begin{lemma}[Bracketing entropy bound for projected Lipschitz functions]
\label{lem: proj_lipschitz_entropy_bound}
Let $\cF = \{\varphi \circ \ptheta : \theta \in \unitsph,  \varphi \in \mathsf{Lip}_{1,0}(\R)\}$.
For any $\mu \in \calP_{2+\epsilon}(\R^d)$ for some $\epsilon > 0$, we have
\be
\label{eq:proj_lipschitz_entropy_bound}
\log N_{[\,]}(t, \cF, L^2(\mu)) \lesssim C^{1/\epsilon}t^{-(1+2/\epsilon)} + d \log (2C/t),  \quad t \in (0,1)
\ee
up to a constant that depends only on $\epsilon$,
where $C = \int \| x \|^{2+\epsilon} d\mu(x)$. In particular, $\calF$ is $\mu$-Donsker provided that $\mu \in \calP_{4+\epsilon}(\R^d)$ for some $\epsilon > 0$. 
\end{lemma}

The proof of this lemma relies on the following two auxiliary lemmas.

\begin{lemma}[Uniform entropy for $1$-Lipschitz functions]
\label{lem: lipschitz entropy}
For any $K > 0$,
\begin{equation}
\label{eq:BL1 entropy}
\log N(t, \Lip ([-K,K]), \|\cdot\|_\infty) \leq \big \lceil K/t \rceil 2\log 2, \  t \in (0,1),
\end{equation}
where $\Lip ([-K,K])$ denotes the class of $1$-Lipschitz functions $f$ on $[-K,K]$ with $f(0)=0$. 
\end{lemma}

\begin{proof}[Proof of Lemma~\ref{lem: lipschitz entropy}]
Let $N = \lceil K/t \rceil$. Partition $[0,K]$ into $N$ subintervals with length at most $t$. Enumerate the endpoints of these subintervals and label them as 
$0 = x_0 < \dots < x_N = K$. For each vector $j = (j_{-N},\dots,j_{-1},j_{1},\dots,j_{N}) \in \{-1,1\}^{2N}$ indexed from $-N$ to $N$ excluding $0$, define the function
\be\label{eq: lipschitz net defn}f_j(x) = \begin{cases} 0 &\text{ if } x = 0\\
\sum_{l=1}^k j_l (x_k - x_{k-1}) &\text{ if } x = x_k, 1 \leq k \leq N\\
\sum_{l=1}^{k} j_{-l} (x_k - x_{k-1}) &\text{ if } x = -x_k, 1 \leq k \leq N\\
\text{piecewise linear} &\text{ otherwise.}
\end{cases}
.
\ee
These form a $t$-net of $\Lip ([-K,K])$ w.r.t. $\|\cdot \|_{\infty}$.
\end{proof}

The following result is standard and thus we omit the proof. 

\begin{lemma}[Covering number for $\unitsph$]
\label{lem: unit sphere entropy}
For any $t \in (0,1)$, $N(t, \unitsph, \|\cdot\|) \leq (5/t)^{d}$. 
\end{lemma}

\begin{proof}[Proof of Lemma \ref{lem: proj_lipschitz_entropy_bound}]
Pick any $t \in (0,1)$ and choose $M > 0$ such that 
\[
\sqrt{\EE\big[\|X\|^2\ind_{\{\|X\| > M\}}\big]} \leq t/4
\]
for $X \sim \mu$. 
Consider a minimal $t/8$-net $f_1,\dots,f_m$ of $\Lip\big([-M,M]\big)$ w.r.t. $\|\cdot\|_\infty$, and a minimal $t/(8M)$-net $\theta_1, \dots, \theta_n$ of $\unitsph$ w.r.t. $\| \cdot \|$. Consider the following function brackets for $i = 1, \dots, m$ and $j = 1, \dots, n$:
\[
\Big[(f_i \circ \mathfrak{p}^{\theta_j} - t/4)\ind_{\{\|x\|\leq M\}} - \|x\|\ind_{\{\|x\|>M\}},\ (f_i \circ \mathfrak{p}^{\theta_j} +  t/4)\ind_{\{\|x\|\leq M\}} + \|x\|\ind_{\{\|x\|>M\}}\Big].
\]
It is not difficult to verify that these brackets cover $\cF$. The $L^2(\mu)$ width of the brackets is bounded above by $t$. By Lemmas \ref{lem: lipschitz entropy} and \ref{lem: unit sphere entropy}, 
\[
\log mn = \log m + \log n \leq \big ( 8  M/t  + 1 \big ) 2\log 2 + \log d + (d - 1) \log 3 + (d - 1) \log \left ( \frac{8M}{t} \right ).
\]
We are left to determine $M$. Observe that
\[
\EE[\|X\|^2\ind_{\{ \|X\|>M \}}] \leq \frac{\E[\| X \|^{2+\epsilon}]}{M^\epsilon} = \frac{C}{M^{\epsilon}}. 
\]
Hence $M = (4C^{1/2}/t)^{2/\epsilon}$ is a feasible choice. This leads to the first claim.

For the second claim, by Theorem 2.5.6 in \cite{van1996weak}, it is enough to show that
\[
\int_0^\infty \sqrt{N_{[\,]}(t, \cF, L^2(\mu))}\,dt < \infty.
\]
Provided that $\mu$ has finite $(4+\epsilon)$th moment for some $\epsilon >0$, $\sqrt{N_{[\,]}(t, \cF, L^2(\mu))}$ is integrable near $0$. Since $\cF$ has  envelope function $F(x) = \|x\|$ which is also in $L^2(\mu)$ under the given moment condition, $\calF$ is $\mu$-Donsker. 
\end{proof}

\subsubsection{Proof of Theorem \ref{thm: limit distribution SWp}}
\textbf{Part (i).} 
We shall apply Proposition \ref{prop: master proposition} with $\calP_0 = \{ \rho \in \calP(\R^d) : \rho (\calX) = 1 \}$, $\calX = \inte(\supp(\mu))$, $\cF = \{\varphi \circ \ptheta : \theta \in \unitsph,  \varphi \in \mathsf{Lip}_{1,0}(\R)\}$, $F(x) = \|x\|$, and $\delta (\rho) = \SWp^p(\rho,\nu)$. It is not difficult to verify that $\calP_0$ is convex and contains $\hat{\mu}_n$ with probability one. Further, by Lemma \ref{lem: proj_lipschitz_entropy_bound}, $\cF$ is $\mu$-Donsker, $\sqrt{n}(\hat{\mu}_n - \mu) \stackrel{d}{\to} G_\mu$ in $\ell^{\infty}(\calF)$. Combined with Corollary \ref{lem: Lipschitz} (recall that $\MSWone (\mu_0,\mu_1) = \| \mu_1 - \mu_0 \|_{\infty,\calF}$ by the Kantorovich-Rubinstein duality) and Lemma \ref{lem: directional derivative SWp}, we have verified all conditions in Proposition \ref{prop: master proposition} with $\delta'_\mu (\rho-\mu) = \int_{\unitsph}(\rho-\mu)\big(\varphi^\theta \circ \ptheta \big)d\sigma(\theta)$. 

Let $M > 0$ such that $\calX \subset B(0,M)$. Each $\varphi^\theta |_{[-M,M]}$ is $\mathsf{c}$-Lipshitz for some constant $\mathsf{c} < \infty$ independent of $\theta$ by Lemma \ref{lem: OT potential}. Extend $\varphi^\theta |_{[-M,M]}$ to $\R$ without changing the Lipschitz constant and redefine $\varphi^\theta$ by this extension.
In view of Lemma \ref{lem: joint measurability} Part (iii) and noting that $(\rho-\mu)\big(\varphi^\theta \circ \ptheta - a \big) = (\rho-\mu)\big(\varphi^\theta \circ \ptheta \big)$, $\delta_\mu'$ extends to $\calT_{\calP_0}(\mu)$ as 
\begin{equation}
\delta_\mu'(m) = \int_{\unitsph} \mathsf{c}\cdot  m\big(\bar{\varphi}^\theta \circ \ptheta \big)d\sigma(\theta)
\label{eq: H derivative SWp}
\end{equation}
with $\bar{\varphi}^\theta = (\varphi^\theta-\varphi^\theta(0))/\mathsf{c}$. 
Conclude that 
\[
\sqrt{n}(\SWp^p(\hat{\mu}_n,\nu) - \SWp^p(\mu,\nu)) \stackrel{d}{\to} \delta'_\mu (G_{\mu}) = \int_{\unitsph} \mathsf{c}\cdot G_\mu\big (\bar{\varphi}^\theta \circ \ptheta \big)  d\sigma(\theta). 
\]
Set $\mathbb{G}_\mu (\theta) = \mathsf{c}\cdot G_\mu\big(\bar{\varphi}^\theta \circ \ptheta \big)$. The process $(\mathbb{G}_\mu(\theta))_{\theta \in \unitsph}$ is centered Gaussian with covariance function $\Cov (\mathbb{G}_\mu(\theta), \mathbb{G}_\mu(\vartheta)) = \Cov_{\mu}\big ( \varphi^\theta \circ \ptheta, \varphi^\vartheta \circ \mathfrak{p}^{\vartheta} \big)$. Further, since $\theta \mapsto \bar{\varphi}^\theta \circ \ptheta$ is continuous relative to the pseudometric $d_{\mu}(f,g) = \sqrt{\Var_{\mu}(f-g)}$ (cf. Lemma \ref{lem: joint measurability} (i)) and $G_\mu$ has uniformly continuous paths relative to $d_\mu$, we see that $\mathbb{G}_\mu$ has continuous paths. By Riemann approximation, we see that $\int_{\unitsph} \mathbb{G}_\mu d\sigma$ is centered Gaussian with variance $\int_{\unitsph}\int_{\unitsph} \Cov(\mathbb{G}_\mu(\theta),\mathbb{G}_{\mu}(\vartheta)) d\sigma (\theta) d\sigma(\vartheta) = v_{p}^2$.

Second, we show that $v_{p}^2$ coincides with the semiparametric efficiency bound by invoking \cref{prop: master proposition 2}. We have shown that the function class $\calF$ is $\mu$-Donsker, which implies $\mu$-pre-Gaussianity. Also, for any bounded $\mu$-mean zero function $h$,  $\supp \big((1+th) \mu \big) = \supp (\mu)$ for sufficiently small $t > 0$, so that $(1+th)\mu \in \calP_0$.
It remains to show that the derivative $\delta_\mu'$ in \eqref{eq: derivative SWp} suitably extends to a continuous linear functional on $\ell^\infty (\calF)$. Let $\calC_u(\calF)$ denote the space of uniformly continuous functions on $\calF$ relative to $d_\mu$. As $\calF$ is $\mu$-pre-Gaussian, $\calF$ is totally bounded for $d_\mu$. Since $\theta \mapsto \varphi^\theta \circ \ptheta$ is continuous relative to $d_\mu$, the expression \eqref{eq: H derivative SWp} makes sense for any $m \in \calC_u (\calF)$, and the map $\delta_\mu'$ is continuous and linear on $(\calC_u(\calF),\| \cdot \|_{\infty,\calF})$. 
Extend $\delta_{\mu}'$ from $\calC_u(\calF)$ to a continuous linear functional on $\ell^\infty(\calF)$ by the Hahn-Banach theorem. For this extension, since $h \mu \in \calC_u (\calF)$ for any bounded $\mu$-mean zero function $h$, we have $t^{-1}(\delta_\mu ((1+th)\mu) - \delta(\mu)) \to \delta_\mu'(h\mu)$ as $t \downarrow 0$. Applying \cref{prop: master proposition 2}, we see that the semiparametric efficiency bound is given by $\Var (\delta_\mu'(G_\mu))$, which coincides with $v_{p}^2$ from the computation above. 

Finally, we show the bootstrap consistency. Since the derivative $\delta_\mu'$ in \eqref{eq: derivative SWp} is linear, $\rho \mapsto \delta(\rho)$ is Hadamard differentiable (w.r.t. $\| \cdot \|_{\infty,\calF}$) at $\rho = \mu$ tangentially to $\supp(G_\mu)$ by Corollary \ref{cor: master}. Thus the bootstrap consistency follows from Theorem 23.9 in \cite{vanderVaart1998asymptotic}.

\medskip

\textbf{Part (ii).} To account for the two-sample case, we consider the following setting. Set $\calP_0 = \{ \rho_1 \otimes \rho_2 \in \calP(\R^d \times \R^d) : (\rho_1 \otimes \rho_2) (\cX \times \cY) = 1 \}$ with $\cX = \inte (\supp(\mu))$ and $\cY = \inte(\supp(\nu))$, $\calF = \{ f_1 \oplus f_2 : f_i \in \calF_0, i=1,2 \}$ with $\calF_0 = \{ \varphi  \circ \ptheta : \theta \in \unitsph, \varphi\in \Lip (\R) \}$, and $\delta (\rho_1 \otimes \rho_2) = \SWp^p (\rho_1,\rho_2)$. Observe that, for $f_1,f_2 \in \calF_0$, $\rho_1 \otimes \rho_2, \rho_1' \otimes \rho_2' \in \calP_0$, and $t \in [0,1]$, we have
\[
\big (t (\rho_1 \otimes \rho_2) +(1-t) (\rho_1' \otimes \rho_2') \big) (f_1 \oplus f_2) = \Big ( \big(t \rho_1 + (1-t)\rho_1'\big) \otimes \big(t \rho_2 +(1-t)\rho_2'\big) \Big)(f_1 \oplus f_2).
\]
Thus
\[
t (\rho_1 \otimes \rho_2) +(1-t) (\rho_1' \otimes \rho_2') = \big(t \rho_1 + (1-t)\rho_1'\big) \otimes \big(t \rho_2 +(1-t)\rho_2'\big)
\]
as elements of $\ell^\infty (\calF_0)$, 
so that $\calP_0$ is covex as a subset of $\ell^\infty (\calF)$.
It is not difficult to see from Corollary \ref{lem: Lipschitz} that
\[
\begin{split}
|\delta (\mu_0\otimes\nu_0) - \delta (\mu_1\otimes\nu_1)| &\lesssim \MSWone(\mu_0,\mu_1) + \MSWone(\nu_0,\nu_1) \\
&\le 2 \| (\mu_0 \otimes \nu_0) - (\mu_1 \otimes \nu_1) \|_{\infty,\calF}.
\end{split}
\]
Also, for $\mu_t = (1-t)\mu + t\rho_1$ and $\nu_t = (1-t)\nu+t\rho_2$ with $\rho_1 \otimes \rho_2 \in \calP_0$, arguing as in the proof of Lemma \ref{lem: directional derivative SWp} and using the fact that $[\varphi^\theta]^{cc} = \varphi^\theta$, we have 
\[
\begin{split}
&\lim_{t \downarrow 0}\frac{\delta (\mu_t \otimes \nu_t) - \delta(\mu \otimes \nu)}{t} \\
&=\int_{\unitsph} \left [(\rho_1-\mu) (\varphi^\theta \circ \ptheta) + (\rho_2-\nu) ([\varphi^\theta]^c \circ \ptheta) \right] d\sigma(\theta) \\
&= \int_{\unitsph} \int (\varphi^\theta \circ \ptheta) \oplus ([\varphi^\theta]^c \circ \ptheta) d\big(\rho_1 \otimes \rho_2- \mu \otimes \nu\big) d\sigma(\theta). 
\end{split}
\]
Here recall that $\mu_t \otimes \nu_t$ can be identified with $(1-t) (\mu \otimes \nu) +t(\rho_1 \otimes \rho_2)$ as elements of $\ell^\infty (\calF)$.
Let $\psi^\theta = [\varphi^\theta]^c$. Again, choose versions of $\varphi^\theta$ and $\psi^\theta$ such that, for some constant $\mathsf{c}$ independent of $\theta$, $\bar{\varphi}^\theta = (\varphi^\theta-\varphi^\theta (0))/\mathsf{c}, \bar{\psi}^\theta = (\psi^\theta - \psi^\theta(0))/\mathsf{c} \in \Lip (\R)$ for every $\theta \in \unitsph$. Then, the derivative $\delta_{\mu \otimes \nu}'$ extends to $\calT_{\calP_0}(\mu \otimes \nu)$ as 
\[
\delta_{\mu \otimes \nu}'(m) = \int_{\unitsph} \mathsf{c} \cdot m \big ((\bar{\varphi}^\theta \circ \ptheta) \oplus (\bar{\psi}^\theta \circ \ptheta)\big ) d\sigma(\theta).
\]

We shall verify that the process $\sqrt{n}(\hat{\mu}_n \otimes \hat{\nu}_n - \mu \otimes \nu)$ converges weakly to a tight Gaussian process in $\ell^\infty (\calF)$.  
Since $\calF_0$ is Donsker w.r.t. $\mu$ and $\nu$, and $X_i$' and $Y_i$'s are independent, by Example 1.4.6 in \cite{van1996weak} (combined with Lemma 3.2.4 in \cite{dudley2014uniform} concerning measurable covers), we see that 
\begin{equation}
\big(\sqrt{n}(\hat{\mu}_n-\mu),\sqrt{n}(\hat{\nu}_n-\nu) \big) \stackrel{d}{\to} (G_\mu,G_\nu') \quad \text{in} \ \ \ell^\infty (\calF_0) \times \ell^\infty(\calF_0),
\label{eq: expression}
\end{equation}
where $G_\mu$ and $G_\nu'$ are independent and the weak limits of $\sqrt{n}(\hat{\mu}_n-\mu)$ and $\sqrt{n}(\hat{\nu}_n-\nu)$ in $\ell^{\infty}(\calF_0)$, respectively. Since the map $\ell^\infty(\calF_0) \times \ell^\infty (\calF_0) \ni (z_1,z_2) \mapsto (z_1(f_1)+z_2(f_2))_{f_1 \oplus f_2 \in \calF} \in \ell^\infty(\calF)$ is continuous, 
we have
\[
\sqrt{n}(\hat{\mu}_n \otimes \hat{\nu}_n - \mu \otimes \nu) \stackrel{d}{\to} G_{\mu \otimes \nu} \quad \text{in} \ \ell^\infty (\calF)
\]
where $G_{\mu \otimes \nu} (f_1 \oplus f_2) = G_{\mu}(f_1) + G_{\nu}'(f_2)$ for $f = f_1 \oplus f_2 \in \calF$. Applying Proposition \ref{prop: master proposition} (see also \cref{rem: convexity}), we conclude that
\[
\begin{split}
\sqrt{n}(\delta(\hat{\mu}_n\otimes\hat{\nu}_n) - \delta (\mu \otimes \nu)) &\stackrel{d}{\to} \int_{\unitsph} \left [ \mathsf{c} \cdot G_{\mu}(\bar{\varphi}^\theta \circ \ptheta) + \mathsf{c} \cdot G_{\nu}' (\bar{\psi}^\theta \circ \ptheta) \right ]d\sigma(\theta) \\
&=\int_{\unitsph} \mathbb{G}_\mu  d\sigma + \int_{\unitsph} \mathbb{G}_{\nu}' d\sigma. 
\end{split}
\]
where $\mathbb{G}_\mu$ is given before and $\mathbb{G}_{\nu}'$ is defined as $\mathbb{G}_\nu'(\theta) = \mathsf{c} \cdot G_\nu'(\bar{\psi}^\theta \circ \ptheta)$. 
Arguing as in Part (i), the right-hand side is Gaussian with mean zero and variance $v_{p}^2 + w_{p}^2$.

Second, the fact that $v_{p}^2 + w_{p}^2$ coincides with the semiparametric efficiency bound follows from \cref{cor: efficiency two sample}. The argument is similar to the one-sample case and omitted for brevity.

Finally, we show the bootstrap consistency. The argument is analogous to the one-sample case above. To apply Theorem 23.9 in \cite{vanderVaart1998asymptotic},  we need to verify that, as maps into $\ell^\infty (\calF)$, the sequence $\sqrt{n}(\hat{\mu}_n^B \otimes \hat{\nu}_n^B - \hat{\mu}_n \otimes \hat{\nu}_n)$ is asymptotically measurable and converges conditionally in distribution to $(G_\mu (f_1)+G_\nu'(f_2))_{f_1 \oplus f_2 \in \calF}$ given $(X_1,Y_1),(X_2,Y_2),\dots$.
Since $\calF_0$ is Donsker w.r.t. $\mu$ and $\nu$, each of $\sqrt{n}(\hat{\mu}_n^B-\hat{\mu}_n)$ and $\sqrt{n}(\hat{\nu}_n^B-\hat{\nu}_n)$ is asymptotically measurable and converges conditionally in distribution (cf. Chapter 3.6 in \cite{van1996weak}). By Lemma 1.4.4 and Example 1.4.6 in \cite{van1996weak},  as maps into $\ell^\infty (\calF_0) \times \ell^\infty(\calF_0)$, the sequence $\big (\sqrt{n}(\hat{\mu}_n^B-\hat{\mu}_n),\sqrt{n}(\hat{\nu}_n^B-\hat{\nu}_n)\big)$ is asymptotically measurable and converges conditionally in distribution to $(G_\mu,G_\nu')$.  Since the map $\ell^\infty(\calF_0) \times \ell^\infty (\calF_0) \ni (z_1,z_2) \mapsto (z_1(f_1)+z_2(f_2))_{f_1 \oplus f_2 \in \calF} \in \ell^\infty(\calF)$ is continuous, we see that, as maps into $\ell^\infty (\calF)$,  $\sqrt{n}(\hat{\mu}_n^B \otimes \hat{\nu}_n^B - \hat{\mu}_n \otimes \hat{\nu}_n)$ is asymptotically measurable and converges conditionally in distribution to $(G_\mu (f_1)+G_\nu'(f_2))_{f_1 \oplus f_2 \in \calF}$, as desired. The rest follows from Corollary \ref{cor: master} and Theorem 23.9 in \cite{vanderVaart1998asymptotic}.
\qed

\subsubsection{Proof of Theorem \ref{thm: limit distribution MSWp}}
 Given the G\^{a}teaux derivative result for $\MSWp$ (cf. Lemma \ref{lem: directional derivative SWp}), the proof is analogous to \cref{thm: limit distribution SWp}. We omit the details for brevity.
 \qed

\subsection{Proof of Theorem~\ref{thm: sliced W1 limit}}

\subsubsection{Preliminaries}
\begin{lemma}[CLT in $L^1$ space]
\label{prop: L1 CLT}
Let $(S,\calS,\mu)$ be a $\sigma$-finite measure space with $\calS$ being countably generated, and let $Z_1,Z_2,\dots$ be i.i.d. $L^1(\mu)$-valued random variables. If $\Prob (\| Z_1 \|_{L^1(\mu)} > t) = o(t^{-2})$ as $t \to \infty$ and $\int_{S} \sqrt{\E[Z_1^2(x)]}\ d\mu(x) < \infty$, then there exists a centered jointly measurable Gaussian process $\mathsf{G}$  with paths in $L^1(\mu)$ such that $n^{-1/2} \sum_{i=1}^n Z_i \dconv \mathsf{G}$ in $L^1(\mu)$.
\end{lemma}

\begin{proof}
See \cite{ledoux1991probability}, Theorem 10.10.
\end{proof}

\begin{lemma}[Directional derivative of $L^1$ norm]
\label{lem: directional_derivative_l1_norm}
Let $(S,\calS,\mu)$ be a measure space. Then, for any $f,g \in L^1 (\mu)$ and $t_n \downarrow 0$, we have 
\[
\lim_{n \to \infty} \frac{\|f+t_ng \|_{L^1(\mu)} - \| f \|_{L^1(\mu)}}{t_n} = \int [\sign (f)] g \, d\mu  + \int_{\{ f = 0 \}}|g| \, d\mu. 
\]
\end{lemma}

\begin{proof}[Proof of Lemma \ref{lem: directional_derivative_l1_norm}]
Let $S^{+} = \{ f > 0 \}$ and $S^{-} = \{ f < 0 \}$, and set $h = \frac{g}{|f|}\ind_{\{f \ne 0\}}$ and $d\gamma = |f| d\mu$. Observe that 
\[
\begin{split}
    t_n^{-1}&\big(\|f+t_ng \|_{L^1(\mu)} - \| f \|_{L^1(\mu)}\big)\\
    &=t_{n}^{-1}\int_{S^{+}} \big(|1+t_nh| - 1 \big) \, d\gamma + t_{n}^{-1}\int_{S_{-}} \big( |-1+t_n h| - 1 \big) \, d\gamma +  \int_{\{ f = 0 \}} |g| \, d\mu. 
\end{split}
\]
We next show that the first and second terms converge to $\int_{\{ f> 0 \}} gd\mu$ and $\int_{\{ f < 0 \}} (-g)d\mu$, respectively. The arguments are analogous, so we  only prove the former.

By normalization, we may assume without loss of generality that $\gamma (S^{+}) = 1$. Consider the set $A_{n} = \{ t_n |h| < 1 \}$. Since $h \in L^1(\gamma)$, by Markov's inequality we have $\gamma (A_n^c \cap S^{+}) \to 0$. Consequently
\[
\left | \int_{S^{+}}  |1+t_nh| d\gamma - \int_{S^{+}} |1+t_n h\ind_{A_n}| d\gamma \right | \le t_n \int_{S^+} |h| \ind_{A_n^c} d\gamma = o(t_n). 
\]
Since $1+t_n h\ind_{A_n} \ge 0$, we have 
\[
\int_{S^{+}}\big( |1+t_n h\ind_{A_n}| - 1 \big) d \gamma = t_n\int_{S^+} h \ind_{A_n} d\gamma = t_n \int_{S^+} h d\gamma+ o(t_n). 
\]
As $\int_{S^+} h d\gamma = \int_{\{ f>0 \}} gd\mu$, we obtain the desired result.
\end{proof}

\begin{lemma}[Directional derivative of supremum functional]
\label{lem: directional_derivative_sup_norm}
Let $S$ be a nonempty set. Consider the functional $\phi: \ell^{\infty} (S) \to \R$ defined by $\phi (f) = \sup_{x \in S} f(x)$. Suppose that there exists a pseudometric $\mathsf{d}$ on $S$ such that $(S,\mathsf{d})$ is totally bounded, and let $C_u(S,\mathsf{d})$ denote the set of all uniformly $\mathsf{d}$-continuous functions on $S$. Then $\phi$ is Hadamard directionally differentiable at $f \in C_u (S,\mathsf{d})$ with derivative
\[
\phi_{f}'(g) = \sup_{x \in \overline{S}: f(x)  = \phi (f)} g(x), \quad g \in C_u (S,\mathsf{d}),
\]
where $\overline{S}$ is the completion of $S$ with respect to $\mathsf{d}$ [note that every $f \in C_u (S,\mathsf{d})$ has the unique continuous extension to $\overline{S}$, which we denote by the same symbol $f$].
\end{lemma}

\begin{proof}
See \cite{carcamo2020directional}, Corollary 2.5.
\end{proof}

\subsubsection{Proof of Theorem~\ref{thm: sliced W1 limit}}

\noindent\textbf{Part (i).} Recall the expression $\SWone (\hat{\mu}_n,\mu) =\| F_{\hat{\mu}_n} - F_{\mu} \|_{L^1(\lambda \otimes \sigma)}$. We divide the proof into two steps.

\underline{Step 1}. We will first show that there exists a centered, jointly measurable Gaussian process $\mathsf{G}_\mu  = \big(\mathsf{G}_{\mu} (t,\theta)\big)_{(t,\theta) \in \R \times \unitsph}$ with paths in $L^1(\lambda \otimes \sigma)$ and covariance function (\ref{eq: cov function}) such that 
\[
\sqrt{n} (F_{\hat{\mu}_n}(t;\theta)-F_{\mu}(t;\theta))_{(t,\theta) \in \R \times \unitsph} \dconv \mathsf{G}_{\mu} \quad \text{in $L^1(\lambda \otimes \sigma)$}. 
\]
We will apply the CLT in the $L^1$ space (Lemma \ref{prop: L1 CLT}) to prove this claim. Let $Z_i (t,\theta) = \ind (\theta^{\intercal} X_i \le t) - F_{\mu}(t;\theta)$. Then
$\sqrt{n}\big(F_{\hat{\mu}_n}(t;\theta)-F_{\mu}(t;\theta)\big) = n^{-1/2} \sum_{i=1}^n Z_i(t,\theta)$. Since $((t,\theta),\omega) \mapsto Z_i (t,\theta,\omega) = \ind (\theta^{\intercal} X_i(\omega) \le t) - F_{\mu}(t;\theta)$ is (jointly) measurable, $Z_i$ can be regarded an an $L^1(\lambda \otimes \sigma)$-valued random variable as long as $\| Z_i \|_{L^1 (\lambda \otimes \sigma)} < \infty$ a.s. \cite{Bycz1977}, which will be verified below.   

Observe that $\| Z_i \|_{L^1(\lambda \otimes \sigma)}$ can be evaluated as
\begin{align*}
&\int_{\unitsph} \int_\R \left|\ind (\theta^{\intercal} X_i \leq t) - \Prob(\theta^{\intercal} X_i \leq t)\right| \, dt  \,d\sigma(\theta)\\
&= \int_{\unitsph} \Bigg(\int_{-\infty}^0 \left|\ind (\theta^{\intercal} X_i \leq t) - \Prob(\theta^{\intercal} X_i \leq t)\right| \, dt \,  \\
&\qquad \qquad +  \int_0^\infty \left|\ind (\theta^{\intercal} X_i > t) - \Prob(\theta^{\intercal} X_i > t)\right| \, dt \, \Bigg) d\sigma(\theta)\\
&\leq \int_{\unitsph} \Bigg[(-\theta^{\intercal} X_i \vee 0) + \int_{-\infty}^0 \Prob(\theta^{\intercal} X_i \leq t) \, dt \,  \\
&\qquad \qquad + (\theta^{\intercal} X_i \vee 0) + \int_0^\infty \Prob(\theta^{\intercal} X_i > t) \, dt \, \Bigg] d\sigma(\theta)\\
&\leq \|X_i\| + \E[\|X_i\|].
\end{align*}
This implies that $Z_i \in L^1 (\lambda \otimes \sigma)$. Further, for $t>2\E[\|X_i\|]$, 
\begin{align*}
    \Prob(\|Z_i\|_{L^1(\lambda \otimes \sigma)} > t) &\leq \Prob\left(\|X_i\| + \E[\|X_i\|] > t\right) \\
    &\le \Prob\left(\|X_i\|  > t/2 \right) \\
    &\le \frac{2^{2+\epsilon}\E[\|X_i\|^{2+\epsilon}]}{t^{2+\epsilon}} = o(t^{-2})
\end{align*}
as $t \to \infty$. Finally, letting $K= \sup_{\theta \in \unitsph} \E[|\theta^{\intercal}X_i|^{2+\epsilon}] < \infty$, we have
\begin{equation}
\begin{split}
    &\int_{\unitsph}\int_{\R} \sqrt{\E[Z_i(t,\theta)^2]} \, dt \,  d\sigma(\theta) \\
    &= \int_{\unitsph} \int_{-\infty}^\infty \sqrt{F_{\mu}(t;\theta)(1-F_{\mu}(t;\theta))}\, dt \, \,d\sigma(\theta) \\
    &\le \int_{\unitsph}\left(\int_{-\infty}^0 \sqrt{F_{\mu}(t;\theta)}\, dt \,  + \int_0^\infty \sqrt{1 - F_{\mu}(t;\theta)}\, dt \, \right) d\sigma(\theta)\\
    &\leq \int_{\unitsph} \left(2 \int_0^\infty \sqrt{\Prob(|\theta^{\intercal} X_i| \geq t)}\, dt \, \right)d\sigma(\theta)\\
    &\leq 2 \left(1 + \int_1^\infty \sqrt{\frac{K}{t^{2+\epsilon}}}\, dt \, \right)\\
    &\leq 2 + 4\frac{\sqrt{K}}{\epsilon} < \infty.
\end{split}
\label{eq: SWone bound}
\end{equation}
Thus, by Lemma \ref{prop: L1 CLT}, we have $n^{-1/2}\sum_{i=1}^n Z_i \dconv \mathsf{G}_{\mu}$ in $L^1 (\lambda \otimes \sigma)$. 

\underline{Step 2}. We will apply the extended functional delta method to derive the limit distribution for $\sqrt{n} ( \SWone (\hat{\mu}_n,\nu) - \SWone (\mu,\nu) )$. Let $\phi: L^1(\lambda \otimes \sigma) \to \R$ be defined by
\[
\phi (f) = \| f \|_{L^1 (\lambda \otimes \sigma)}, \ f \in L^1 (\lambda \otimes \sigma). 
\]
By Lemma \ref{lem: directional_derivative_l1_norm} and the fact that $\phi$ is trivially Lipshitz in $\| \cdot \|_{L^1(\lambda \otimes \sigma)}$,  $\phi$ is Hadamard directionally differentialble at $f \in L^1 (\lambda \otimes \sigma)$  with derivative 
\[
\phi_{f}'(g) = \int [\sign (f)] g d(\lambda \otimes \sigma) + \int_{\{ f=0 \}}|g| d(\lambda \otimes \sigma). 
\]
Thus, by the extended functional delta method, we have 
\[
\begin{split}
\sqrt{n} ( \SWone (\hat{\mu}_n,\nu) - \SWone (\mu,\nu)) &= \sqrt{n} (\phi (F_{\mu}-F_{\nu} + F_{\hat{\mu}_n} - F_{\mu}) - \phi (F_{\mu} - F_{\nu})) \\
&\dconv \phi_{F_{\mu} - F_{\nu}}' (\mathsf{G}_{\mu}). 
\end{split}
\]
The right-hand side coincides with the limit distribution in (\ref{eq: limit distribution SW1}). 

\medskip

\noindent\textbf{Part (ii).} Observe that $\SWone (\hat{\mu}_n,\hat{\nu}_n) = \| F_{\hat{\mu}_n} - F_{\hat{\nu}_n} \|_{L^{1}(\lambda \otimes \sigma)}$. From the proof of Part (i), we have $\sqrt{n}(F_{\hat{\mu}_n} - F_{\hat{\nu}_n} - F_{\mu}+F_{\nu}) \stackrel{d}{\to} \mathsf{G}_{\mu} - \mathsf{G}_{\nu}'$ in $L^1(\lambda \otimes \sigma)$. The rest of the proof is analogous to Step 2 in the proof of Part (i). 

\medskip
\noindent\textbf{Part (iii).}
Observe that $\MSWone(\mu,\nu) = \sup_{f \in \cF} (\mu-\nu)(f)$ by the Kantorovich-Rubinstein duality. 
By Lemma \ref{lem: proj_lipschitz_entropy_bound}, the function class $\cF$ is $\mu$-Donsker under the given moment condition, $\sqrt{n}(\hat{\mu}_n-\mu) \stackrel{d}{\to} G_{\mu}$ in $\ell^\infty (\calF)$. 

Then, we apply the functional delta method to the supremum functional, Lemma~\ref{lem: directional_derivative_sup_norm}. Since, relative to the standard deviation metric, $f \mapsto \int f d(\mu-\nu) = \mu (f-\nu(f))$ is uniformly continuous, $\calF$ is totally bounded, $G_\mu$ has uniformly continuous paths (cf. Example 1.5.10 in \cite{van1996weak}), we obtain the desired result.

\medskip
\noindent\textbf{Part (iv).} The proof is analogous to Part (iii) and thus omitted. 
\qed

\subsection{Proof of Corollary \ref{cor: SWone}}
The limit distribution results directly follow from Theorem \ref{thm: sliced W1 limit}.
Under the setting of \cref{COR:asymp_normal_Wone}, inspection of the proof of Theorem \ref{thm: sliced W1 limit} shows that the map $f \mapsto \| f \|_{L^1(\lambda \otimes \lambda)}$ is Hadamard differentiable at $f = F_\mu - F_\nu$ with derivative 
\[
\phi_{F_\mu - F_\nu}'(g) = \int [\sign (F_\mu - F_\nu) ] g d(\lambda \otimes \sigma), \ g \in L^1(\lambda \otimes \sigma). 
\]
Combined with the equivalence between the CLT and the bootstrap consistency for i.i.d. random variables with values in a separable Banach space (cf. Remark 2.5 in \cite{gine1990bootstrapping}), the bootstrap consistency result for $\SWone$ (like that in Theorem \ref{thm: limit distribution SWp}) follows via Theorem 23.9 in \cite{vanderVaart1998asymptotic}.

Regarding semiparametric efficiency, for any bounded measurable function $h$ on $\R^d$ with $\mu$-mean zero, 
\[
\begin{split}
&\frac{1}{s} \big ( \| F_{(1+sh)\mu} - F_{\nu} \|_{L^1 (\lambda \otimes \sigma)} - \| F_\mu - F_\nu \|_{L^1(\lambda \otimes \sigma)} \big) \\
&\to \int_{\R \times \unitsph} [\sign (F_\mu - F_\nu)(t;\theta) ] \left \{ \int h (x)\ind(x^\intercal \theta \le t) d\mu(x) \right \} dt d\sigma(\theta) \quad (s \downarrow 0) \\
&=\int_{\R \times \unitsph} [\sign (F_\mu - F_\nu)(t;\theta) ] \left \{ \int h (x)\mathfrak{g}(x,t,\theta) d\mu(x) \right \} dt d\sigma(\theta) \\
&= \left \langle h, \int_{\R \times \unitsph} [\sign (F_\mu - F_\nu)(t;\theta) ] \mathfrak{g}(\cdot,t,\theta) dt d\sigma(\theta)\right \rangle_{L^2(\mu)},
\end{split}
\]
where $\mathfrak{g}(x,t,\theta) = \ind(x^\intercal \theta \le t) - \ind(t > 0)$ and the second equality follows as $h$ has $\mu$-mean zero. Hence, the semiparametric efficiency bound for estimating $\rho \mapsto \SWone(\rho,\nu)$ at $\rho = \mu$ is given by the $\mu$-variance of 
\[
\int_{\R \times \unitsph} [\sign (F_\mu - F_\nu)(t;\theta) ] \mathfrak{g}(\cdot,t,\theta) dt d\sigma(\theta),
\]
which agrees with $v_{1}^2$ above. Likewise, in the two-sample case, it is shown that the semiparametric efficiency bound agrees with $v_{1}^2+w_{1}^2$; see the proof of \cref{cor: efficiency two sample}. \qed

\section{Proofs for Section \ref{sec: smooth Wp}}

\subsection{Proof of Proposition {\ref{prop:smoothstability}}}
\label{subsec:proof-prop:smoothstability}
Lemma 5.2 of \cite{santambrogio2010} establishes
    that convolution acts as a contraction for $\mathsf W_p$, $\mathsf{W}_p^{(\sigma)}\leq\mathsf{W}_p$. For any $\rho\in \cP(\R^d)$, the coupling $\pi:=\chi_{\sigma}(x-y)\dx{\lambda(x)}\dx{\rho(y)}\in \Pi(\rho,\rho*\eta_{\sigma})$ satisfies
    \[
        \int_{\R^d\times \R^d}\|x-y\|^p\dx{\pi(x,y)}=\sigma^p\int_{\R^d}\int_{\R^d}\|z\|^p\chi_1(z)\dx{\lambda(z)}\dx{\mu(y)}=\sigma^p\E_{\eta_1}[\|X\|]^p,
    \]
     so $\mathsf W_p(\mu,\mu*\eta_{\sigma})\leq (\E[\|X_{\eta_1}\|^p])^{1/p}$ by Lemma 7.1.10 in \cite{ambrosio2005}. Conclude by applying the triangle inequality, $\mathsf{W}_p\leq \mathsf{W}_p^{(\sigma)}+2(\E_{\eta_1}[\|X\|^p])^{1/p}$.
\qed

\subsection{Proof of Proposition {\ref{prop:metricspace}}}
\label{subsec:proof-prop:metricspace}

\begin{lemma}
\label{lem:charfunction}
    For $\sigma>0$, the characteristic function $\phi_{\eta_{\sigma}}$ of $\eta_{\sigma}$ is zero on a set of negligible (Lebesgue) measure. 
\end{lemma}

\begin{proof}
The Fourier-Laplace transform of $\chi_{\sigma}$ is defined by 
\[
    \widehat{\chi}_{\sigma}:t\in \mathbb C^d\mapsto {\int_{\R^d}e^{-it^{\intercal}{x}}\chi_{\sigma}(x)d{x}}.
\]
 By Theorem 7.1.14 in \cite{hormander03}, $\widehat{\chi}_{\sigma}$ is complex-analytic and, since $\chi_{\sigma}$ is even, $\widehat{\chi}_{\sigma}\vert_{\R^d}=\phi_{\eta_{\sigma}}(-\cdot)$ is real analytic. The conclusion follows since the set of zeros of any real analytic function on $\R^d$ that is not identically has zero Lebesgue measure \cite[p.240]{federer14}.
\end{proof}

\begin{lemma}
\label{lem:weakconvergenceconvolution}
    For $p\geq 1$, $\sigma> 0$, $(\mu_k)_{k\in \NN}\subset \cP_p(\R^d)$ and $\mu\in\cP_p(\R^d)$, $\mu_{k}*\eta_{\sigma}\stackrel{w}{\to} \mu*\eta_{\sigma}$ as $k\to\infty$ if and only if $\mu_k\stackrel{w}{\to} \mu$ as $k\to\infty$.
\end{lemma}
\begin{proof}
First, if $\mu_k \stackrel{w}{\to} \mu$, $\phi_{\mu_k*\eta_{\sigma}}=\phi_{\mu_k}\phi_{\eta_{\sigma}}\to\phi_{\mu}\phi_{\eta_{\sigma}}=\phi_{\mu*\eta_{\sigma}}$ pointwise, 
so $\mu_k*\eta_{\sigma}\stackrel{w}{\to}\mu*\eta_{\sigma}$. 
To prove the opposite direction, if $\mu_k*\eta_{\sigma}{\stackrel{w}{\to}} \mu*\eta_{\sigma}$, arguing as above yields $\phi_{\mu_k}\to \phi_{\mu}$ pointwise almost everywhere, as $\phi_{\eta_{\sigma}}=0$ on a set of negligible (Lebesgue) measure by \cref{lem:charfunction}. It is not difficult to show that pointwise almost everywhere convergence of characteristic functions implies pointwise convergence thereof and hence that $\mu_k\stackrel{w}{\to}\mu$. 
\end{proof}

\begin{proof}[Proof of \cref{prop:metricspace}] The fact that $\mathsf W_p^{(\sigma)}$ is a metric follows essentially from the proof of Proposition 1 in \cite{nietert21}, with a slight modification using \cref{lem:charfunction} and the fact that characteristic functions are uniformly continuous. Similarly, the fact that $\mathsf W_p^{(\sigma)}$ and $\mathsf W_p$ induce the same topology on $\cP_p(\R^d)$ follows from the proof of Proposition 1 in \cite{nietert21} along with \cref{lem:weakconvergenceconvolution}.
\end{proof}

\subsection{Proof of Theorem {\ref{thm:LimitDistributionsSmoothWasserstein}}}
\label{subsec:proof-thm:LimitDistributionsSmoothWasserstein}
We only prove the one-sample result, as the two-sample result follows from essentially the same reasoning coupled with the fact that $\varphi^c$ is the unique OT potential from $\nu*\eta_{\sigma}$ to $\mu*\eta_{\sigma}$ for $\mathsf W_p$ up to additive constants under the assumption that $\nu$ has connected support; see also the proof of \cref{thm: limit distribution SWp} Part (ii). 
In what follows, we fix $\sigma > 0 $ and $1 < p < \infty$.
Recall that $B$ denotes the unit ball in $L^2(\mathcal{X}_{\sigma})$.
We first prove the weak convergence of the smoothed empirical process $\sqrt{n}(\hat\mu_n-\mu)*\eta_{\sigma}$ and the associated bootstrap process in $\ell^\infty (B)$.

\begin{lemma}
    \label{lem:LimitDistributionSmoothProcess}
    For any $\mu \in \calP(\calX)$, we have $\sqrt{n}(\hat\mu_n-\mu)*\eta_{\sigma}\stackrel{d}{\to}\mathbb{G}_{\mu}$ in $\ell^{\infty}(B)$, where $\mathbb{G}_{\mu} = \left(\mathbb G_{\mu}(f)\right)_{f\in B}$ is a tight centered Gaussian process in $\ell^{\infty}(B)$ with covariance function $\Cov (\mathbb G_{\mu}(f),\mathbb G_{\mu}(g)) = \Cov_{\mu}(f*\chi_{\sigma},g*\chi_{\sigma})$.
    Further, conditionally on $X_1,X_2,\dots$, $\sqrt{n}(\hat{\mu}_n^B - \hat{\mu}_n)*\eta_\sigma \stackrel{d}{\to} \mathbb{G}_\mu$ in $\ell^\infty (B)$ for almost every realization of $X_1,X_2,\dots$

\end{lemma}

\begin{proof}
We will show weak convergence of $\sqrt{n}(\hat{\mu}_n^B - \hat{\mu}_n)*\eta_\sigma$ in $(L^2(\calX_\sigma))'$, the topological dual of $L^2(\calX_\sigma)$. 
    Observe that, for every $f\in L^2(\mathcal{X}_{\sigma})$ and $x\in\supp(\mu)$, 
    \begin{equation}
    \label{eq:operatornorm}
    \abs{(f*\chi_{\sigma})(x)}\leq \|{f}\|_{L^2(\mathcal{X}_{\sigma})}\|\chi_{\sigma}(x-\cdot)\|_{L^2(\mathcal{X}_{\sigma})}\leq C \|f\|_{L^2(\mathcal{X}_{\sigma})},
    \end{equation}
    where $C = \| \chi_\sigma \|_{\infty} \sqrt{\lambda (B(0,\sigma))} < \infty$. 
    
  Set $Z_i:=\delta_{X_i}*\eta_{\sigma}$ for $i\in\NN$. The inequality (\ref{eq:operatornorm}) implies that $\| Z_i \|_{(L^2(\mathcal{X}_{\sigma}))'} \le C$. It is not difficult to see that the Borel $\sigma$-field on $(L^2(\mathcal{X}_{\sigma}))'$ agrees with the smallest $\sigma$-field for which the projections $z \mapsto z(f),  f \in L^2(\calX_\sigma)$ are measurable. Thus, $Z_i$ can be thought as i.i.d. random variables with values in $(L^2(\mathcal{X}_{\sigma}))'$ with (Bochner) expectation $\E[Z_i] = \mu*\eta_\sigma$ (cf. Lemma 5.1.1 in \cite{stroock2010probability} for Bochner integrals).
    Since $(L^2(\calX_\sigma))'$ is a separable Hilbert space, by Example 1.8.5 in \cite{van1996weak}, we have
$\sqrt{n}(\hat\mu_n-\mu)*\eta_{\sigma} = n^{-1/2}\sum_{i=1}^n (Z_i-\E[Z_i]) \stackrel{d}{\to} \mathbb{G}_{\mu}^{\circ}$, where $\mathbb{G}_{\mu}^{\circ}$ is a centered Gaussian random variable in $\left(L^2(\mathcal{X}_{\sigma})\right)'$.

In addition, let $Z_{i}^B = \delta_{X_i^B}*\eta_\sigma$ for $i=1,\dots,n$. 
Then, $Z_1^B,\dots,Z_n^B$ are an i.i.d. sample from the empirical distribution $n^{-1}\sum_{i=1}^n \delta_{Z_i}$, so that by Remark 2.5 in \cite{gine1990bootstrapping}, we have that, conditionally on $X_1,X_2,\dots$, $\sqrt{n}(\hat{\mu}_n^B - \mu)*\eta_\sigma = n^{-1/2}\sum_{i=1}^n (Z_i^B - \overline{Z}_n)  \stackrel{d}{\to} \mathbb{G}_\mu$ in $\left(L^2(\mathcal{X}_{\sigma})\right)'$ for almost every realization of $X_1,X_2,\dots$, where $\overline{Z}_n = n^{-1}\sum_{i=1}^n Z_i$.

The conclusion of the lemma follows by noting that the map $ (L^2(\mathcal{X}_{\sigma}))'\ni z \mapsto z|_{B} \in \ell^{\infty}(B)$ is an isometry.
\end{proof}

Second, we shall establish Lipschitz continuity of $S_p^{(\sigma)}(\cdot,\nu)$ w.r.t. $\| \cdot * \eta_\sigma \|_{\infty,B}$. 

\begin{lemma}
    \label{lem:SmoothLipschitzFunctional}
    For every $\mu_0,\mu_1 \in \calP(\calX)$, 
    \[
    |S_p^{(\sigma)}(\mu_1,\nu) - S_p^{(\sigma)}(\mu_0,\nu)| \le \sqrt{\lambda(\mathcal{X}_{\sigma})}\diam(\mathcal{X}_{\sigma})^p\| (\mu_1 - \mu_0)*\eta_\sigma \|_{\infty,B}.
    \]
\end{lemma}

\begin{proof}
Recall $S_p^{(\sigma)} (\mu_i,\nu) = [\Wp(\mu_i*\eta_\sigma,\nu*\eta_\sigma)]^p$. 
 Let $\varphi_i$ be an OT potential from $\mu_i*\eta_\sigma$ to $\nu*\eta_\sigma$ for $\mathsf W_p$ satisfying $0\leq \varphi_i\leq \diam(\mathcal{X}_{\sigma})^p$ on $\mathcal{X}_{\sigma}$ for $i=0,1$ as in Remark 1.13 of \cite{villani2003topics}. Then, by duality, 
    \begin{align*}
       S_p^{(\sigma)}(\mu_1,\nu) - S_p^{(\sigma)}(\mu_0,\nu)&\leq
       \int_{\mathcal{X}_{\sigma}}\varphi_1\dx{\big((\mu_1-\mu_0)*\eta_\sigma\big)}
       \\&
       \leq \|\varphi_1\|_{L^2(\mathcal{X}_{\sigma})} \| (\mu_1-\mu_0)*\eta_\sigma \|_{\infty,B}.
    \end{align*}
    Likewise,
    \[
    S_p^{(\sigma)}(\mu_1,\nu) - S_p^{(\sigma)}(\mu_0,\nu) \ge -\|\varphi_0\|_{L^2(\mathcal{X}_{\sigma})} \| (\mu_1-\mu_0)*\eta_\sigma \|_{\infty,B},
    \]
 so that 
 \[
|S_p^{(\sigma)}(\mu_1,\nu) - S_p^{(\sigma)}(\mu_0,\nu)| \le \big ( \|\varphi_0\|_{L^2(\mathcal{X}_{\sigma})} \vee \|\varphi_1\|_{L^2(\mathcal{X}_{\sigma})}\big )\| (\mu_1-\mu_0)*\eta_\sigma \|_{\infty,B}. 
 \]
 The conclusion of the lemma follows by noting that 
 \[
 \|\varphi_0\|_{L^2(\mathcal{X}_{\sigma})} \vee \|\varphi_1\|_{L^2(\mathcal{X}_{\sigma})}\leq \sqrt{\lambda(\mathcal{X}_{\sigma})}\diam(\mathcal{X}_{\sigma})^p
 \]
 by construction.    
\end{proof}

The following lemma concerns the G\^ateaux directional derivative of $S_p^{(\sigma)}(\cdot,\nu)$. 

\begin{lemma}
    \label{lem:SmoothFunctionalDerivative}
Let $\mu_0,\mu_1 \in \calP(\calX)$ be such that $\inte (\supp(\mu_0))$ is connected with negligible boundary and $\supp(\mu_1) \subset \supp(\mu_0)$. Then, 
\[
\lim_{t \downarrow 0} \frac{S_p^{(\sigma)}(\mu_0+t(\mu_1-\mu_0),\nu) - S_p^{(\sigma)}(\mu_0,\nu)}{t} = \big( (\mu_1-\mu_0)*\eta_\sigma \big) (\varphi_0),
\]
where $\varphi_0$ is an OT potential from $\mu*\eta_{\sigma}$ to $\nu*\eta_\sigma$ for $\mathsf W_p$
\end{lemma}

The proof of Lemma \ref{lem:SmoothFunctionalDerivative} relies on the following technical lemma concerning the support of $\mu*\eta_\sigma$. 

\begin{lemma}
    \label{lem:connectedsupport}
   For any $\mu \in \calP(\calX)$,  $\interior(\supp(\mu*\eta_{\sigma}))$ has negligible boundary. Further, $\interior(\supp(\mu*\eta_{\sigma}))$ is connected provided $\supp(\mu)$ is connected.
\end{lemma}
\begin{proof}[Proof of \cref{lem:connectedsupport}]
    By \cite[p.~106]{malliavin1995integration},  
    \[
    \supp(\mu*\eta_{\sigma})=\overline{\supp(\mu)+\overline{B(0,\sigma)}}=\cup_{x\in\supp(\mu)}\overline{B(x,\sigma)},
    \]
    where the second equality is due to the fact that a union of closed balls with centers in a closed set is closed. 
    
    Let $\Gamma=\{z\in\R^d:\text{dist}(z,\supp(\mu))=\sigma\}$ and $\Lambda=\cup_{x\in\supp(\mu)}B(x,\sigma)$. Observe that, since $\supp(\mu)$ is closed, $\text{dist}(z,\supp(\mu))=\|{z-x}\|$ for some $x\in\supp(\mu)$. It is clear that $\Gamma\subset \partial\Lambda$. On the other hand, if $z\not\in \Gamma\cup \Lambda$, then $\text{dist}(z,\supp(\mu))>\sigma$ in which case $z\in\R^d\setminus\cup_{x\in\supp(\mu)}\overline{B(x,\sigma)}$. Since $\cup_{x\in\supp(\mu)}\overline{B(x,\sigma)}$ is closed, we have $z\not\in\partial\Lambda$ and $\Gamma=\partial \Lambda$. The same argument implies that $\cup_{x\in\supp(\mu)}\overline{B(x,\sigma)}\subset \Gamma\cup \Lambda$ and the other inclusion is trivial, so ${\cup_{x\in\supp(\mu)}{B(x,\sigma)}}\subset \interior(\supp(\mu*\eta_{\sigma}))$ and the closure of both sets coincide with $\supp(\mu*\eta_{\sigma})$.
    
Example $1$ in  \cite{rudi2020finding} implies that  $\Lambda$ 
    has a (locally) Lipschitz boundary, which is negligible by Section 4.9 and Section 4.11 of \cite{adams2003sobolev}. As such, $\interior(\supp(\mu*\eta_{\sigma}))$  has negligible boundary.
    
    Assume that $\Lambda$ is not connected, then there exist nonempty sets $A,B\subset \supp(\mu)$ such that $A\cup B=\supp(\mu)$, and $\cup_{x\in A}B(x,\sigma)$ and $\cup_{y\in B}B(y,\sigma)$ are disjoint, which contradicts the assumption that $\supp(\mu)$ is connected. Hence $\interior(\supp(\mu*\eta_{\sigma}))$ is connected, as $\Gamma\subset \interior(\supp(\mu*\eta_{\sigma}))\subset \overline{\Gamma}=\supp(\mu*\eta_{\sigma})$.   
\end{proof}

\begin{proof}[Proof of \cref{lem:SmoothFunctionalDerivative}]
 Set $\mu_t=\mu_0+t(\mu_1-\mu_0)$, and let $\varphi^\circ_t$ be OT potentials from $\mu_t*\eta_\sigma$ to $\nu*\eta_\sigma$ for $\mathsf{W}_p$ satisfying $0\leq \varphi^\circ_t \leq \diam(\mathcal{X}_{\sigma})^p$ on $\mathcal{X}_{\sigma}$; cf. Remark 1.13 in \cite{villani2003topics}. For any fixed $x_0\in\interior(\supp(\mu_0*\eta_{\sigma}))$, $\varphi_t:=\varphi^\circ_t-\varphi^\circ_t(x_0)$ is again an OT potential from $\mu_t*\eta_\sigma$ to $\nu*\eta_\sigma$ for $\mathsf{W}_p$ satisfying $\varphi_t(x_0)=0$ and $\abs{\varphi_t}\leq \diam(\mathcal{X}_{\sigma})^p$ on $\mathcal{X}_{\sigma}$.

Observe that, by duality,
    \[
    \begin{split}
        S_p^{(\sigma)}(\mu_t,\nu) - S_p^{(\sigma)}(\mu_0,\nu) &= [\Wp(\mu_t*\eta_\sigma,\nu*\eta)]^p - [\Wp(\mu_t*\eta_\sigma,\nu*\eta_\sigma)]^p \\
        &\leq \int_{\mathcal{X}_{\sigma}}\varphi_t\dx{(\mu_t*\eta_\sigma)}-\int_{\mathcal{X}_{\sigma}}\varphi_t\dx{(\mu_0*\eta_{\sigma})}\\
        &\le t\int_{\calX_\sigma} \varphi_t d\big( (\mu_1-\mu_0)*\eta_\sigma \big),
        \end{split}
    \]
so that
    \[
        \limsup_{t\downarrow 0}\frac{S_p^{(\sigma)}(\mu_t,\nu) - S_p^{(\sigma)}(\mu_0,\nu)}{t}\leq \limsup_{t\downarrow 0}\int_{\calX_\sigma} \varphi_t d\big( (\mu_1-\mu_0)*\eta_\sigma \big).
    \]
    Since $\interior(\supp(\mu_0*\eta_{\sigma}))$ is connected with negligible boundary by \cref{lem:connectedsupport}, and $\mu_{t_n}*\eta_\sigma
{\stackrel{w}{\to}}\mu*\eta_{\sigma}$ for any sequence $t_n\downarrow 0$, Theorem 3.4 in \cite{delbarrio2021} implies that $\varphi_{t_n}-a_n\to\varphi$ pointwise on $\interior(\supp(\mu_0*\eta_{\sigma}))$ for some sequence of constants $a_n$. Since $\varphi_{t_n}(x_0)=\varphi(x_0)=0$, we must have that $a_n\to 0$, so in fact $\varphi_{t_n}\to\varphi$ pointwise on $\interior(\supp(\mu_0*\eta_{\sigma}))$.
Since $\inte (\supp(\mu_0*\eta_\sigma))$ has negligible boundary by \cref{lem:connectedsupport} and $\supp(\mu_1) \subset \supp(\mu_0)$, 
applying the dominated convergence theorem yields
    \[
       \int_{\mathcal{X}_{\sigma}}\varphi_{t_n}d\big( (\mu_1-\mu_0)*\eta_\sigma \big)\to\int_{\mathcal{X}_{\sigma}}\varphi_0 d\big( (\mu_1-\mu_0)*\eta_\sigma \big).
    \]
  Since $t_n\downarrow 0$ is arbitrary, we have
    \[
     \limsup_{t\downarrow 0}\frac{S_p^{(\sigma)}(\mu_t,\nu) - S_p^{(\sigma)}(\mu_0,\nu)}{t}\leq \int  \varphi_0 d\big( (\mu_1-\mu_0)*\eta_\sigma \big).
    \]

 For the reverse inequality, observe that 
    \[
    \begin{split}
           S_p^{(\sigma)}(\mu_t,\nu) - S_p^{(\sigma)}(\mu_0,\nu) &\geq \int_{\mathcal{X}_{\sigma}}\varphi\dx{(\mu_t*\eta_\sigma)}-\int_{\mathcal{X}_{\sigma}}\varphi_0\dx{(\mu_0*\eta_{\sigma})}\\
           &=t\int_{\mathcal{X}_{\sigma}}\varphi_0\dx{\left((\mu_1-\mu_0)*\eta_{\sigma}\right)}, 
           \end{split}
    \]
 so that
    \[
        \liminf_{t\downarrow 0}\frac{  S_p^{(\sigma)}(\mu_t,\nu) - S_p^{(\sigma)}(\mu_0,\nu)}{t}\geq \int_{\mathcal{X}_{\sigma}}\varphi_0\dx{\left((\mu_1-\mu_0)*\eta_{\sigma}\right)}.
    \]
    
    Conclude that $\lim_{t \downarrow 0} t^{-1}\big (S_p^{(\sigma)}(\mu_t,\nu) - S_p^{(\sigma)}(\mu_0,\nu) \big)=\big( (\mu_1-\mu_0)*\eta_\sigma \big) (\varphi_0)$. The derivative is uniquely defined, as $\varphi_0$ is unique up to additive constants on $\interior(\supp(\mu*\eta_{\sigma}))$ and $(\mu_1-\mu_0)*\eta_{\sigma}$ has total mass zero.
\end{proof}

\begin{proof}[Proof of \cref{thm:LimitDistributionsSmoothWasserstein}]
As noted before, we only prove the one-sample result. We shall apply
     \cref{prop: master proposition} by identifying $S_p^{(\sigma)}(\rho,\nu) = \Wp^p(\rho*\eta_\sigma,\nu*\sigma)$ as a functional defined on $\calP_0*\eta_\sigma = \{ \rho*\eta_\sigma : \rho \in \calP_0 \}$ with $\calP_0 = \{ \rho \in \calP(\calX) : \supp (\rho) \subset \supp (\mu) \}$, and taking $\calF = B$, $F \equiv C$ with constant $C$ given in \eqref{eq:operatornorm}, and $\mu_n = \hat{\mu}_n*\eta_\sigma$. The set $\calP_0*\eta_\sigma$ is convex and contains $\hat{\mu}_n * \eta_\sigma$ with probability one. Combining Lemmas \ref{lem:LimitDistributionSmoothProcess}--\ref{lem:SmoothFunctionalDerivative}, we have verified all conditions in \cref{prop: master proposition}. Pick a version of  $\varphi$ that is bounded on $\calX_\sigma$; cf. Remark 1.13 in \cite{villani2008optimal}.
     Choose a constant $\mathsf{c}>0$ such that $\bar{\varphi}=\varphi/\mathsf{c} \in B$. Observe that the derivative formula in Lemma \ref{lem:SmoothFunctionalDerivative} agrees with $\mathsf{c} \cdot ((\mu_1-\mu_0)*\eta_\sigma)(\bar{\varphi})$. 
     Conclude that 
     \[
     \sqrt{n}(S_p^{(\sigma)}(\hat\mu_n,\nu)-S^{(\sigma)}(\mu,\nu)) \stackrel{d}{\to} \mathsf{c} \cdot \mathbb{G}_\mu(\bar{\varphi}) \sim N(0,\Var_\mu (\varphi*\chi_\sigma)).
     \]

Pertaining to semiparametric efficiency, think of $S_p^{(\sigma)}(\cdot,\nu)$ now as a functional defined on $\calP_{0}$. 
It is not difficult to see that, for any bounded measurable function $h$ on $\R^d$ with $\mu$-mean zero, $(1+ 
th)\mu\in\cP_{0}$ for sufficiently small $t > 0$.  \cref{lem:SmoothFunctionalDerivative} implies that $t^{-1}(S_p^{(\sigma)}((1+th)\mu,\nu) - S_p^{(\sigma)}(\mu,\nu)) \to \big((h\mu)*\eta_\sigma\big)(\varphi) = (h\mu)(\varphi * \chi_\sigma)$ as $t \downarrow 0$, which is the point evaluation at $\varphi*\chi_\sigma \in L^2(\mu)$.
Hence, by \cref{prop: master proposition}, the semiparametric efficiency bound for estimating $\mathsf S_p^{(\sigma)}(\cdot,\nu)$ at $\mu$ agrees with $\Var_\mu (\varphi*\chi_\sigma)$.

Finally, the bootstrap consistency follows by linearity of the derivative in Lemma \ref{lem:SmoothFunctionalDerivative} and the second claim in \cref{lem:LimitDistributionSmoothProcess}, combined with Corollary \ref{cor: master} and Theorem 23.9 in \cite{vanderVaart1998asymptotic}.
\end{proof}

\subsection{Proof of Proposition {\ref{thm:ExpectedW2Rates}}}
\label{subsec:proof-thm:ExpectedW2Rates}

Let $Z_i=\delta_{X_i}*\eta_{\sigma}$ be i.i.d. $(L^2(\mathcal{X}_{\sigma}))'$-valued random variables as in the proof of \cref{lem:LimitDistributionSmoothProcess}. By the proof of \cref{lem:SmoothLipschitzFunctional},
\begin{align*}
\begin{split}
    \abs{S_p^{(\sigma)}(\hat\mu_n,\nu)-S^{(\sigma)}(\mu,\nu)}&\leq \sqrt{\lambda(\mathcal{X}_{\sigma})}\diam(\mathcal{X}_{\sigma})^p\|{(\hat\mu_n-\mu)*\eta_{\sigma}}\|_{(L^2(\mathcal{X}_{\sigma}))'}, \\
    &=\sqrt{\lambda(\mathcal{X}_{\sigma})}\diam(\mathcal{X}_{\sigma})^p\left \|n^{-1}\sum_{i=1}^n (Z_i - \E[Z_i])\right\|_{(L^2(\mathcal{X}_{\sigma}))'}.
    \end{split}
\end{align*}
Since $(L^2(\mathcal{X}_{\sigma}))'$ is a Hilbert space, 
\[
    \E\left[\left \|n^{-1}\sum_{i=1}^n (Z_i - \E[Z_i])\right\|_{(L^2(\mathcal{X}_{\sigma}))'}^2\right]=n^{-1}\E[\| Z_1 - \E[Z_1] \|_{(L^2(\mathcal{X}_{\sigma}))'}^2] \le 4n^{-1}C^2,
\]
where $C = \| \chi_\sigma \|_{\infty} \sqrt{\lambda (B(0,\sigma))}$ is a constant given in \eqref{eq:operatornorm}. Applying Jensen's inequality in the previous display and combining both bounds yields 
\[
    \E\left[ \abs{S_p^{(\sigma)}(\hat\mu_n,\nu)-S^{(\sigma)}(\mu,\nu)}\right]\leq 2C\sqrt{\lambda(\mathcal{X}_{\sigma})}\diam(\mathcal{X}_{\sigma})^pn^{-1/2}.
\]
Applying the inequality $\abs{x-y}\leq y^{1-p}\abs{x^p-y^p}$ yields 
\[
\E\left[\abs{\mathsf{W}_p^{(\sigma)}(\hat\mu_n,\nu)-\mathsf{W}_p^{(\sigma)}(\mu,\nu)}\right]\leq 2C\sqrt{\lambda(\mathcal{X}_{\sigma})}\diam(\mathcal{X}_{\sigma})^p [\mathsf{W}_p^{(\sigma)}(\mu,\nu)]^{1-p} n^{-1/2}, 
\]
as desired. \qed

\subsection{Proof of Theorem {\ref{thm:LimitDistributionsSmoothWassersteinNULL}}}
\label{subsec:proof-thm:LimitDistributionsSmoothWassersteinNULL}

In what follows, we fix $\sigma > 0$. We will first show that the smoothed empirical process $\sqrt{n}(\hat{\mu}_n-\mu)*\eta_\sigma$ can be regarded as the scaled sum of i.i.d. random variables with values in $\Hmt{\mu*\eta_{\sigma}}$.

\begin{lemma}
    \label{lem:SmoothDualSobolev}
   If $\mu\in\cP(\mathcal{X})$ is such that $\mu*\eta_{\sigma}$ satisfies a $2$-Poincar{\'e} inequality, then  $(\delta_X-\mu)*\eta_{\sigma}\in \Hmt{\mu*\eta_{\sigma}}$ $\mu$-a.s.
\end{lemma}

\begin{proof}
    Let $Y\sim\mu$. Observe that the Lebesgue density of  $\mu*\eta_{\sigma}$ agrees with $\E[\chi_{\sigma}(Y-\cdot)]$. We shall verify that, for any $x \in \inte (\supp(\mu))$,  $\E[\chi_{\sigma}(Y-\cdot)]>0$ on $B(x,\sigma)$. Indeed, for any $y\in B(x,\sigma)$, there exists $0<\epsilon_y\leq \sigma$ with $B(x,\epsilon_y)\subset B(y,\sigma)$. Since $\mu(B(x,\epsilon_y))>0$ and $\chi_{\sigma}(\cdot - y)>0$ on $B(x,\epsilon_y)\subset B(y,\sigma)$, we have $\E[\chi_{\sigma}(Y-y)]\ge \E[\chi_\sigma (Y-y)\ind_{B(x,\epsilon_y)}(Y)]>0$. 
    
    Now, for any $f\in C_0^\infty+\R$ and $x\in\supp(\mu)$, we have
    \begin{align*} 
        \abs{(f*\chi_{\sigma})(x)}&\leq\int_{B(x,\sigma)}\abs{f(y)\chi_{\sigma}(x-y)}\dx{y}
        \\
        &=\int_{B(x,\sigma)}\abs{f(y)}\frac{\chi_{\sigma}(x-y)}{\E[\chi_{\sigma}(Y-y)]}\dx{(\mu*\eta_{\sigma})(y)}
        \\
        &\leq \|f\|_{L^2(B(x,\sigma),\mu*\eta_{\sigma})}\left\|\frac{\chi_{\sigma}(x-\cdot)}{\E[\chi_{\sigma}(Y-\cdot)]}\right\|_{L^2(B(x,\sigma),\mu*\eta_{\sigma})},
    \end{align*}
where $\| \cdot \|_{_{L^2(B(x,\sigma),\mu*\eta_{\sigma})}}$ is the $L^2(\mu*\eta_\sigma)$-norm restricted to $B(x,\sigma)$.

   Observe that, since $(\delta_x-\mu)*\eta_\sigma$ has total mass zero, for $\bar{f} = f-(\mu*\eta_\sigma)(f)$, 
    \[
        \abs{\big((\delta_x-\mu)*\eta_\sigma\big)f}=\abs{(\bar{f}*\chi_{\sigma})(x)}\leq\|\bar{f} \|_{L^2(B(x,\sigma),\mu*\eta_{\sigma})}\left\|\frac{\chi_{\sigma}(x-\cdot)}{\E[\chi_{\sigma}(Y-\cdot)]}\right\|_{L^2(B(x,\sigma),\mu*\eta_{\sigma})}. 
    \]
    Since $\mu*\eta_{\sigma}$ satisfies the $2$-Poincar{\'e} inequality and $B(x,\sigma)\subset \supp(\mu*\eta_{\sigma})$, 
    \[
    \|\bar{f}\|_{L^2(B(x,\sigma),\mu*\eta_{\sigma})} \le \| \bar{f} \|_{L^2(\mu*\eta_\sigma)} 
    \leq C_{\mu,\sigma}\|f\|_{\Ht{\mu*\eta_{\sigma}}}
    \]
    for some constant $C_{\mu,\sigma}>0$ depending only on $\mu$ and $\sigma$. 
    
    As for the other term, let $\Gamma=\{ y : \E[\chi_{\sigma}(Y-y)]>0\}$, and observe that, for $X\sim\mu$ independent of $Y$,
    \begin{align*}
        \int_{\Gamma}\frac{\E[\chi^2_{\sigma}(X-y)]}{\E[\chi_{\sigma}(X-y)]}\dx{y}
        &=\int_{\Gamma}\frac{\E[\chi^2_{\sigma}(X-y)\left(\mathbbm{1}_{\{\chi_{\sigma}(X-\cdot)< 1\}}(y)+\mathbbm{1}_{\{\chi_{\sigma}(X-\cdot)\geq1\}}(y)\right)]}{\E[\chi_{\sigma}(X-y)\left(\mathbbm{1}_{\{\chi_{\sigma}(Y-\cdot)< 1\}}(y)+\mathbbm{1}_{\{\chi_{\sigma}(Y-\cdot)\geq 1\}}(y)\right)]}\dx{y}
        \\
        &\leq \int_{\Gamma}\frac{\E[\chi_{\sigma}(X-y)\mathbbm{1}_{\{\chi_{\sigma}(X-\cdot)< 1\}}(y)]+\|\chi^2_{\sigma}\|_{\infty}\mu(\{\chi_{\sigma}(\cdot-y)\geq 1\})}{\E[\chi_{\sigma}(Y-y)\mathbbm{1}_{\{\chi_{\sigma}(Y-\cdot)< 1\}}(y)]+\mu(\{\chi_{\sigma}(\cdot-y)\geq 1\})}\dx{y}
        \\
        &\leq (1\vee \|{\chi_{\sigma}^2}\|_{\infty})\lambda(\Gamma).
    \end{align*}
    Applying the Fubini theorem and Jensen's inequality yields 
    \[
    \begin{split}
        \int_{\Gamma}\frac{\E[\chi^2_{\sigma}(X-y)]}{\E[\chi_{\sigma}(Y-y)]}\dx{y}&=\E\left[\left\|\frac{\chi_{\sigma}(X-\cdot)}{\E[\chi_{\sigma}(Y-\cdot)]}\right\|_{L^2(\Gamma,\mu*\eta_{\sigma})}^2\right]\\
        &\geq\left(\E\left[\left\|\frac{\chi_{\sigma}(X-\cdot)}{\E[\chi_{\sigma}(Y-\cdot)]}\right\|_{L^2(\Gamma,\mu*\eta_{\sigma})}\right]\right)^2 \\
        &\ge \left(\E\left[\left\|\frac{\chi_{\sigma}(X-\cdot)}{\E[\chi_{\sigma}(Y-\cdot)]}\right\|_{L^2(B(x,\sigma),\mu*\eta_{\sigma})}\right]\right)^2,
    \end{split}
    \]
    where the last inequality follows as $B(x,\sigma)\subset \Gamma$ for every $x\in\supp(\mu)$. Conclude that 
    \[
    \E\left[\left\|\frac{\chi_{\sigma}(X-\cdot)}{\E[\chi_{\sigma}(Y-\cdot)]}\right\|_{L^2(B(x,\sigma),\mu*\eta_{\sigma})}\right]\leq \sqrt{(1\vee \|{\chi_{\sigma}^2}\|_{\infty})\lambda(\mathcal{X}_{\sigma})}.
    \]

 We have shown that
    \[
    \| (\delta_{X} -\mu)*\eta_\sigma \|_{\dot{H}^{-1,2}(\mu*\eta_\sigma)} \le  C_{\mu,\sigma}\left\|\frac{\chi_{\sigma}(X-\cdot)}{\E[\chi_{\sigma}(Y-\cdot)]}\right\|_{L^2(B(x,\sigma),\mu*\eta_{\sigma})}, 
    \]
    and the right-hand side is $\mu$-a.s. finite as its expectation is finite, so $(\delta_X-\mu)*\eta_{\sigma}\in \Hmt{\mu*\eta_{\sigma}}$ $\mu$-a.s.
\end{proof}

\begin{lemma}
    \label{lem:DualSobolevEmpirical}
    For $\mu\in\cP(\mathcal{X})$ such that $\mu*\eta_{\sigma}$ satisfies a $2$-Poincar{\'e} inequality, $\sqrt n(\hat\mu_n-\mu)*\eta_{\sigma}\stackrel{d}{\to}\mathbb{G}_{\mu}$ in $\Hmt{\mu*\eta_{\sigma}}$, where $\mathbb{G}_{\mu} =\left(\mathbb{G}_{\mu}(f)\right)_{f\in\Ht{\mu*\eta_{\sigma}}}$ is a centered Gaussian process with paths in $\Hmt{\mu*\eta_{\sigma}}$.
\end{lemma}

\begin{proof}
By \cref{lem:SmoothDualSobolev}, $Z_i := (\delta_{X_i}-\mu)*\eta_{\sigma}$ are i.i.d. mean-zero $\Hmt{\mu*\eta_{\sigma}}$-valued random variables. By Lemma 5.1 in \cite{goldfeld22limit}, $\Hmt{\mu*\eta_{\sigma}}$ is isometrically isomorphic to a closed subspace of $L^2(\mu*\eta_{\sigma};\R^d)$ and hence a separable Hilbert space. Further, from the proof of \cref{lem:SmoothDualSobolev}, 
    \[
       \E\left[\|{Z_i}\|^2_{\Hmt{\mu*\eta_{\sigma}}}\right]\leq C_{\mu,\sigma}^2{(1\vee\|\chi_{\sigma}^2\|_{\infty})\lambda(\mathcal{X}_{\sigma})}<\infty.
    \]
  Conclude from Example 1.8.5 in \cite{van1996weak} that $\sqrt n(\hat\mu_n-\mu)*\eta_{\sigma} = n^{-1/2}\sum_{i=1}^n Z_i$ satisfies a CLT in $\Hmt{\mu*\eta_{\sigma}}$. 
\end{proof}

\begin{lemma}
    \label{lem:SmoothHDDNull}
    For $\mu\in\cP(\mathcal{X})$, the map 
    \[
        \Phi:(h_1,h_2)\in\Xi_{\mu}\times \Xi_{\mu}\subset \Hmt{\mu*\eta_{\sigma}}\times \Hmt{\mu*\eta_{\sigma}}\mapsto \mathsf W_2(\mu*\eta_{\sigma}+h_1,\mu*\eta_{\sigma}+h_2),
    \]
    is Hadamard directionally differentiable at $(0,0)$ with derivative $\Phi'_{(0,0)}(h_1,h_2)=\|h_1-h_2\|_{\Hmt{\mu*\eta_{\sigma}}}$, where 
    \[
        \Xi_{\mu}=\Hmt{\mu*\eta_{\sigma}}\cap\left\{h=(\rho-\mu)*\eta_{\sigma}:\rho\in\cP(\mathcal{X}),\supp(\rho)\subset\supp(\mu)\right\}.
    \]
\end{lemma}

\begin{remark}
Any finite Borel signed measure on $\R^d$ with total mass $0$ and finite $\| \cdot \|_{H^{-1,2}(\mu*\eta_\sigma)}$ norm corresponds to an element of $H^{-1,2}(\mu*\eta_\sigma)$, and this correspondence is one-to-one (as $C_0^\infty$ is right enough). Thus, the above map $\Phi$ is well-defined. 
\end{remark}

The proof of \cref{lem:SmoothHDDNull} follows from that of Proposition 3.3 in \cite{goldfeld22limit} with only minor modifications and thus is omitted for brevity.

\begin{proof}[Proof of \cref{thm:LimitDistributionsSmoothWassersteinNULL}]
    Let $(T_{n,1},T_{n,2}):=\left((\hat\mu_n-\mu)*\eta_{\sigma},(\hat\nu_n-\nu)*\eta_{\sigma}\right)$. Then, by independence of $T_{n,1}$ and $T_{n,2}$, and \cref{lem:DualSobolevEmpirical}, 
    \[
    \sqrt n(T_{n,1},T_{n,2})\stackrel{d}{\to} (\mathbb G_{\mu},\mathbb G_{\mu}') \ \text{ in } \ \Hmt{\mu*\eta_{\sigma}}\times\Hmt{\mu*\eta_{\sigma}}
    .\] 
    Observe that $\Xi_{\mu}\times \Xi_{\mu}$ in \cref{lem:SmoothHDDNull} is convex, so that the tangent cone $\calT_{\Xi_{\mu}\times \Xi_{\mu}}(0,0)$ agrees with the closure 
     of $\left\{ (h_1,h_2)/t:(h_1,h_2) \in \Xi_{\mu}\times \Xi_{\mu},t>0 \right\}$ in $\Hmt{\mu*\eta_{\sigma}}\times \Hmt{\mu*\eta_{\sigma}}$. Thus, $\sqrt n(T_{n,1},T_{n,2})\in \calT_{\Xi_{\mu}\times\Xi_{\mu}}(0,0)$ and $(\mathbb G_{\mu},\mathbb G_{\mu}')\in \calT_{\Xi_{\mu}\times \Xi_{\mu}}(0,0)$ by the
    portmanteau theorem. 

    Now, apply the extended functional delta method, \cref{lem: functional delta method}, to conclude that
    \[
    \begin{split}
       \sqrt n  \mathsf W_p^{(\sigma)}(\hat\mu_n,\hat{\nu}_n) &=  \sqrt n \left( \Phi\left( T_{n,1},T_{n,2} \right)-\Phi\left(0,0\right) \right) \\
       &\stackrel{d}{\to} \Phi'_{(0,0)}(\mathbb{G}_\mu,\mathbb{G}_\mu') = \|{\mathbb G_{\mu}-\mathbb G_{\mu}'}\|_{\Hmt{\mu*\eta_{\sigma}}}.
       \end{split}
\]
The one-sample result is obtained analogously.
\end{proof}

\subsection{Proof of Proposition {\ref{thm:ExpectedW2RatesNULL}}}
\label{subsect:proof-thm:ExpectedW2RatesNULL}
Let $Z_i=(\delta_{X_i}-\mu)*\eta_{\sigma}$ be i.i.d.  $\Hmt{\mu*\eta_{\sigma}}$-valued random variables appearing in the proof of \cref{lem:LimitDistributionSmoothProcess}. By Proposition 2.1 of \cite{goldfeld22limit},
\[
\mathsf{W}_2^{(\sigma)}(\hat\mu_n,\mu)\leq 2\|(\hat\mu_n-\mu)*\eta_{\sigma}\|_{\Hmt{\mu*\eta_{\sigma}}} =2 \left \| n^{-1}\sum_{i=1}^n Z_i \right \|_{\Hmt{\mu*\eta_{\sigma}}}.
\]
Since $\Hmt{\mu*\eta_{\sigma}}$ is a separable Hilbert space, the expectation on the right-hand side is bounded as 
\[
2 n^{-1/2} \sqrt{\E\left[\left\|Z_1\right\|^2_{\Hmt{\mu*\eta_{\sigma}}}\right]}. 
\]
From the proof of \cref{lem:SmoothDualSobolev},
\[
    \E\left[\|Z_1\|^2_{\Hmt{\mu*\eta_{\sigma}}}\right]\leq C_{\mu,\sigma}^2\E\left[\left\|\frac{\chi_{\sigma}(X-\cdot)}{\E[\chi_{\sigma}(Y-\cdot)]}\right\|^2_{L^2(\mu*\eta_{\sigma},\Gamma)}\right]\leq C_{\mu,\sigma}^2(1\vee \|\chi_{\sigma}^2\|_{\infty})\lambda(\mathcal{X}_{\sigma}).
\]
This completes the proof.
\qed

\subsection{Proofs of Theorem \ref{thm:RatesSmooth1WassersteinNULL} and Corollary \ref{cor:RatesSmooth1WassersteinNull}}
\label{subsec:proof-thm:RatesSmooth1WassersteinNULL}
Recall $\calF_\sigma = \{ f*\chi_\sigma : f \in \Lip \}$. 
\begin{lemma}
    \label{lem:LipDonsker}
    The function class $\calF_\sigma$ is $\mu$-Donsker for any $\mu\in\cP(\mathcal{X})$.  
\end{lemma}

\begin{proof}
     Pick any $f\in\Lip$. Since $\abs{f(y)}=\abs{f(y)-f(0)}\leq \|y\|$, $\|{y}\|\leq \|{x}\|+\sigma$ for $y\in B(x,\sigma)$, and  $\chi_{\sigma}$ is supported on $B(0,\sigma)$, we have
    \[
        \abs{(f*\chi_{\sigma})(x)}\leq \int_{B(x,\sigma)}\|y\|\chi_{\sigma}(x-y)\dx{y}\leq \int_{B(x,\sigma)}(\|x\|+\sigma)\chi_{\sigma}(x-y)\dx{y}=\|{x}\|+\sigma.
    \]

    Next, we study the derivatives. For any multi-index $k = (k_1,\dots,k_d) \in \NN_{0}^{d}$, let $\partial^k = \partial_1^{k_1} \cdots \partial_d^{k_d}$ denote the differential operator and set $\bar k:= \sum_{j=1}^{d}k_j$. 
    Observe that, for any $x,x'\in\R^d$, 
    \begin{align*}
        \abs{\partial^{k}(f*\chi_{\sigma})(x)-\partial^{k}(f*\chi_{\sigma})(x')}&\leq \int_{B(0,\sigma)}\abs{f(x-y)-f(x'-y)}\abs{\partial^{k}\chi_{\sigma}(y)}\dx{y}\\
        &\leq \|{\partial^{k}\chi_{\sigma}}\|_{\infty}\lambda(B(0,\sigma))\|x-x'\|.
    \end{align*}
    Since $\chi_\sigma$ is smooth and compactly supported, for any $s \in \NN$, there exists a constant $C_s$ independent of $f$ such that $\| \partial^k (f*\chi_\sigma) \|_{\infty} \le C_s$ for any multi-index with $1 \le \bar k \le s$.
    
    Pick sufficiently large $R > 0$ such that $\calX \subset B(0,R)$. Let $\calF_\sigma |_{B(0,R)} = \{ f|_{B(0,R)} : f \in \calF_\sigma \}$. Then, by Theorem 2.7.1 in \cite{van1996weak}, we have 
    \begin{equation}
    \label{eq: bracketing smooth W1}
    \log N(\calF_\sigma|_{B(0,R)}, \| \cdot \|_\infty, \epsilon) = O(\epsilon^{-d/s}), \ \epsilon \downarrow 0
    \end{equation}
    for any $s \in \mathbb{N}$. Choosing $s = \lfloor d/2 \rfloor +1 > d/2$, we see that the function class $\calF_\sigma$ is $\mu$-Donsker from Theorem 2.5.6 in \cite{van1996weak}.
\end{proof}

\begin{proof}[Proof of Theorem {\ref{thm:RatesSmooth1WassersteinNULL}}]
Recall that $\mathsf{W}_1^{(\sigma)}(\mu,\nu) = \sup_{f \in \calF_\sigma}(\mu-\nu)(f)$ by the Kantorovich-Rubinstein duality. Since $\calF_\sigma$ is Donsker w.r.t. $\mu$ and $\nu$, the conclusion of the theorem follows from a similar argument to the proof of \cref{thm: sliced W1 limit} Parts (iii) and (iv). We omit the details for brevity. 
\end{proof}

\begin{proof}[Proof of \cref{cor:RatesSmooth1WassersteinNull}]
Recall the expression $\mathsf{W}_1^{(\sigma)}(\hat{\mu}_n,\mu) = \sup_{f \in \calF_\sigma}(\hat{\mu}_n-\mu)(f)$. As $\mu$ is compactly supported, the desired result follows by the bracketing entropy bound (\ref{eq: bracketing smooth W1}) combined with Theorem 2.14.2 in \cite{van1996weak}.
\end{proof}

\subsection{Proof of Proposition \ref{prop:truncboundcompact}}
\label{subsec:proof-prop:truncboundcompact}

The following result is a minor adaptation of Lemma 5.3 from {\cite{goldfeld22limit}} using the fact that distributions with regular densities are equivalent to the Lebesgue measure \cite[Section 2]{polyanski2016wasserstein}.
\begin{lemma}
\label{lem:potentialbound}
    Let $1<p<\infty$. Assume that $\nu$ is $\beta$-sub-Weibull for $\beta\in(0,2]$, and that $\mu\in\cP_p(\R^d)$ has a $(c_1,c_2)$-regular density $f_{\mu}$. Let $\varphi$ be an optimal transport potential from $\mu$ to $\nu$ for $\mathsf W_p$. Then there exists a constant $C$ depending only on $p,d,\beta,c_1,c_2$, an upper bound on $\| \| Y \| \|_{\psi_{\beta}}$ for $Y \sim \nu$, and a lower bound on $f_{\mu}(0)$ such that
    \begin{equation}
        \abs{\varphi(x)-\varphi(0)}\leq C\big(1+\| x \|^{2p/\beta}\big)\| x \|,\quad \forall x\in\R^d.
        \label{eq:potentialbound}
    \end{equation}
\end{lemma}

\begin{proof}[Proof of \cref{prop:truncboundcompact}]
\textbf{Part (i).}
    We note that \cref{lem:potentialbound} is applicable in the transport from $\mu$ to $\mu\vert_A$, as $\mu$ is $(c_1,c_2)$-regular and $\mu\vert_A$ is compactly supported.

    Let $\varphi$ be an OT potential from $\mu$ to $\mu\vert_A$ for $\mathsf W_p$ with $\varphi(0)=0$ which satisfies \eqref{eq:potentialbound}. Then,
    \begin{align*}
        \mathsf W_p^p(\mu,\mu\vert_A)&=\int_{\R^d} \varphi\dx{\mu}+\int_{\R^d} \varphi^c\dx{\mu\vert_A}\\
        &= \int_{\R^d} \varphi\dx{\mu}-\int_{\R^d} \varphi\dx{\mu\vert_A}+\int_{\R^d} \varphi\dx{\mu\vert_A}+\int_{\R^d} \varphi^c\dx{\mu\vert_A}\\
        &\leq \int_{\R^d} \varphi\dx{\mu}-\int_{\R^d} \varphi\dx{\mu\vert_A},
    \end{align*}
    where the last inequality is because $\int_{\R^d}\varphi\dx{\mu\vert_A}+\int_{\R^d}\varphi^c\dx{\mu\vert_A}\leq \mathsf W_p^p(\mu\vert_A,\mu\vert_A)=0$.
    It remains to analyze the last term of the above display, 
    \begin{align*}
        &\int_{\R^d}\varphi\dx{\mu}-\int_{\R^d}\varphi\dx{\mu\vert_A}=\int_{A^c}\varphi\dx{\mu}+\left( 1-\frac{1}{\mu(A)} \right)\int_{A}\varphi\dx{\mu}\\
        &\quad \leq \int_{A^c}C(\|{x}\|+\|{x}\|^{p+1})\dx{\mu(x)}+\left(\frac{1}{\mu(A)}-1\right)\int_{A}C(\|{x}\|+\|{x}\|^{p+1})\dx{\mu(x)}
    \\
    &\quad= \int_{\R^d}C(\|{x}\|+\|{x}\|^{p+1})\dx{\mu(x)}+\left(\frac{1}{\mu(A)}-2\right)\int_{A}C(\|{x}\|+\|{x}\|^{p+1})\dx{\mu(x)}
    \\
    &\quad= C\left( \E_{\mu}[\|{X}\|+\E_{\mu}[\|{X}\|^{p+1}]+(1-2\mu(A))(\E_{\mu\vert_A}[\|{Y}\|]+\E_{\mu\vert_A}[\|{Y}\|^{p+1}]) \right).
\end{align*}

\medskip

\textbf{Part (ii).} Let $\varphi$ be an OT potential from $\mu\vert_A$ to $\mu$ for $\mathsf W_p$ with $0\leq \varphi\leq \diam(\supp(\mu))^p$; cf. Remark 1.13 of \cite{villani2003topics}. Then,
    \begin{align*}
    \mathsf W_p^p(\mu\vert_A,\mu)&=\int_{\mathcal{X}} \varphi\dx{\mu\vert_A}+\int_{\mathcal{X}} \varphi^c\dx{\mu}\\
        &= \int_{\mathcal{X}} \varphi\dx{\mu\vert_A}-\int_{\mathcal{X}} \varphi\dx{\mu}+\int_{\R^d} \varphi\dx{\mu}+\int_{\R^d} \varphi^c\dx{\mu}\\
        &\leq \int_{\R^d} \varphi\dx{\mu\vert_A}-\int_{\R^d} \varphi\dx{\mu},
    \end{align*}
    where the last inequality is because $\int_{\R^d}\varphi\dx{\mu}+\int_{\R^d}\varphi^c\dx{\mu}\leq \mathsf W_p^p(\mu,\mu)=0$.
    Since $\varphi$ is nonnegative, 
    \begin{align*}
        \int_{\mathcal{X}}\varphi\dx{\mu\vert_A}-\int_{\mathcal{X}}\varphi\dx{\mu}\leq \left(\frac{1}{\mu(A)}-1\right)\int_{\mathcal{X}} \varphi\dx{\mu}\leq\left( \frac{1}{\mu(A)}-1 \right)\diam(\supp(\mu))^p.
\end{align*}
This completes the proof.
\end{proof}

\subsection{Proof of Proposition \ref{prop:ComputationalGuarantees}}
    The result follows by noting that 
    \[
        \abs{\mathsf W_2^{(\sigma)}(\mu_1,\mu_1)-\widetilde{\mathsf{W}}_2(\mu_1*\eta_{\sigma}\vert_{A_1},\mu_2*\eta_{\sigma}\vert_{A_2})}\leq E_1+E_2, 
    \]
    where 
    \begin{align*}
        E_1&=\abs{\mathsf W_2^{(\sigma)}(\mu_1,\mu_2)-\mathsf{W}_2(\mu_1*\eta_{\sigma}\vert_{A_1},\mu_2*\eta_{\sigma}\vert_{A_2})},
        \\
        &\leq (\mu_1(A_1)^{-1}-1)^{1/2}\diam(\supp(\mu_1))+(\mu_2(A_2)^{-1}-1)^{1/2}\diam(\supp(\mu_2)),
    \end{align*}
    where the upper bound is due to  \cref{prop:truncboundcompact} (ii), and 
    \begin{align*}
        E_2&=\abs{\mathsf{W}_2(\mu_1*\eta_{\sigma}\vert_{A_1},\mu_2*\eta_{\sigma}\vert_{A_2})-\widetilde{\mathsf{W}}_2(\mu_1*\eta_{\sigma}\vert_{A_1},\mu_2*\eta_{\sigma}\vert_{A_2})},
        \\
        &\leq 
        \abs{\mathsf{W}_2^2(\mu_1*\eta_{\sigma}\vert_{A_1},\mu_2*\eta_{\sigma}\vert_{A_2})-\widetilde{\mathsf{W}}_2^2(\mu_1*\eta_{\sigma}\vert_{A_1},\mu_2*\eta_{\sigma}\vert_{A_2})}^{1/2}.
    \end{align*}
   Conclude by applying the error of approximating $\mathsf W_2^2$ by $\widetilde{\mathsf W}_2^2$ from \cref{sec:summaryVacher}.
\qed

\section{Proof for Section \ref{sec: EOT}}
\label{sec: EOT proof}

\subsection{Proof of Theorem \ref{THM:EOT_CLT}}

We will mostly focus on the one-sample case and prove the two-sample case as an extension of the proof for the one-sample case. 
We start from recalling certain regularity properties of optimal EOT potentials between sub-Gaussian distributions. Recall that, for a multi-index $k = (k_1,\dots,k_d) \in \NN_{0}^{d}$,  $\partial^k = \partial_1^{k_1} \cdots \partial_d^{k_d}$ denotes the differential operator and $\bar k:= \sum_{j=1}^{d}k_j$. We say that $\mu$ is $\sigma^2$-sub-Gaussian for $\sigma > 0$ if $\E_\mu[e^{\| X \|^2/(2d\sigma^2)}] \le 2$. Observe that if $\mu$ is $\sigma^2$-sub-Gaussian, then $\mu$ is $\bar{\sigma}^2$-sub-Gaussian for any $\bar{\sigma}  > \sigma$. Also $\mu$ is sub-Gaussian (i.e., $\| \| X \| \|_{\psi_2}  < \infty$ for $X \sim \mu$) if and only if $\mu$ is $\sigma^2$-sub-Gaussian for some $\sigma > 0$. The following directly follows from Proposition 1 in \cite{mena2019statistical} and its proof.

\begin{lemma}[Regularity of optimal EOT potentials \cite{mena2019statistical}]
\label{lem: EOT potential}
Let $\mu,\nu\in\cP(\RR^d)$ be $\sigma^2$-sub-Gaussian. Then, there exist a version of optimal EOT potentials $(\varphi,\psi)$ for $(\mu,\nu)$ such that the following hold: for any positive integer $s \ge 2$, there exists a constant $C_{s,d,\sigma} < \infty$ that depends only on $s,d,\sigma$ such that 
\begin{equation}
\begin{split}
&\varphi(x) \le \frac{1}{2}(\|x\|+\sqrt{2d}\sigma)^2, \quad \forall x \in \R^d, \\
&|\partial^k \varphi(x)|\leq C_{s,d,\sigma}(1+\|x\|^s), \quad  \forall x\in\RR^d, \ \forall k \in \NN_0^k \ \text{with} \ \bar k  \le s, \ \\
&\psi(y) \le \frac{1}{2}(\|y\|+\sqrt{2d}\sigma)^2, \quad \forall y \in \R^d, \\
&|\partial^k \psi(y)|\leq C_{s,d,\sigma} (1+\|y\|^s),  \quad \forall y\in\RR^d, \ \forall k \in \NN_0^k \ \text{with} \ \bar k \le s.
\end{split}
 \label{EQ:EOT_smooth_potentials}
\end{equation}
Further, the optimality condition (\ref{eq: optimality}) holds for every $(x,y) \in \R^d \times \R^d$. 
\end{lemma}

With the estimates (\ref{EQ:EOT_smooth_potentials}) for optimal EOT potentials in mind, let 
\begin{equation}
\calF_\sigma =\big \{ f \in C^s(\R^d) : |\partial^k f(x)|\leq C_{s,d,\sigma}(1+\|x\|^s),  \forall x\in\RR^d,  \forall k \in \NN_0^k \ \text{with} \ \bar k  \le s \big \} 
\label{eq: EOT function class}
\end{equation}
for $s=\max \{\lfloor d/2 \rfloor+1,2\}$. We may assume without loss of generality that $C_{s,d,\sigma}$ is increasing in $\sigma$, so that $\calF_{\sigma} \subset \calF_{\bar{\sigma}}$ for any $\bar{\sigma} > \sigma$. We first establish Lipschitz continuity of the EOT objective w.r.t. $\| \cdot \|_{\infty,\calF_{\sigma}}$ and characterize its (directional) derivative.

\begin{lemma}
\label{lem: EOT gateaux}
Let $\mu_0,\mu_1,\nu \in \calP(\R^d)$ be $\sigma^2$-sub-Gaussian. Then,
\begin{align}
|\sS(\mu_1,\nu) - \sS(\mu_0,\nu)| \le \| \mu_0-\mu_1 \|_{\infty,\calF_{\sigma}}. \label{eq: EOT lipschitz}
\end{align}
Further, we have
\begin{align}
\lim_{t \downarrow 0} \frac{\sS(\mu_0 + t(\mu_1-\mu_0),\nu) - \sS(\mu_0,\nu)}{t}  = \int \varphi_0 d(\mu_1-\mu_0) = (\mu_1-\mu_0) (\varphi_0), \label{eq: EOT derivative}
\end{align}
where $(\varphi_0,\psi_0)$ are optimal EOT potentials for $(\mu_0,\nu)$. 
\end{lemma}

\begin{proof}
We first prove (\ref{eq: EOT derivative}). Let $\mu_t = \mu_0 + t(\mu_1-\mu_0) = (1-t)\mu_0 + t\mu_1$ for $t \in [0,1]$ and let $(\varphi_t,\psi_t)$ be optimal EOT potentials for $(\mu_t,\nu)$. Since $\mu_t$ is $\sigma^2$-sub-Gaussian, we may choose $(\varphi_t,\psi_t)$ in such a way that they satisfy (\ref{eq: optimality}) for every $(x,y) \in \R^d \times \R^d$ and (\ref{EQ:EOT_smooth_potentials}) for $s=2$. Observe that 
\begin{equation}
\begin{split}
\sS(\mu_t,\nu) &= \int \varphi_t d\mu_t + \int \psi_t d\nu \ge \int \varphi_0 d\mu_t + \int \psi_0 d\nu  -\int e^{\varphi_0 \oplus \psi_0 - c} d\mu_t\otimes \nu + 1 \\
&= \int \varphi_0 d\mu_t + \int \psi_0 d\nu=\int \psi_0 d\mu_0 + \int \psi_0 d\nu + t  \int \varphi_0 d(\mu_1-\mu_0)  \\
&=\sS(\mu_0,\nu) + t \int \varphi_0 d(\mu_1-\mu_0),
\end{split}
\label{eq: one-sided-1}
\end{equation}
where the third equality uses the fact that $\int e^{\varphi_0(x)+ \psi_0(y) - c(x,y)}  d\nu(y)= 1$ for all $x\in\RR^d$. Thus,
\[
\liminf_{t \downarrow 0} \frac{\sS(\mu_t,\nu) - \sS(\mu_0,\nu)}{t} \ge \int \varphi_0 d(\mu_1-\mu_0).
\]
Likewise, we have
\begin{equation}
\begin{split}
\sS(\mu_t,\nu) &= \int \varphi_t d\mu_t + \int \psi_t d\nu = \int \varphi_t d\mu_0 + \int \psi_t d\nu + t \int \varphi_t d(\mu_1-\mu_0) \\
&\le \int \varphi_0 d\mu_0 + \int \psi_0 d\nu + \int e^{\varphi_t \oplus \psi_t - c}  d(\mu_0 \otimes \nu) - 1 + t \int \varphi_t d(\mu_1-\mu_0) \\
&=\int \psi_0 d\mu_0 + \int \psi_0 d\nu + t  \int \varphi_t d(\mu_1-\mu_0)\\
&=\sS(\mu_0,\nu) + t \int \varphi_t d(\mu_1-\mu_0),
\end{split}
\label{eq: one-sided-2}
\end{equation}
where the penultimate equality is because $\int e^{\varphi_t(x)+\psi_t(y) - c(x,y)}  d\nu(y) = 1$ for all $x\in\RR^d$. It suffices to show that for any sequence $t_n \downarrow 0$,
\begin{equation}
\lim_{n \to \infty} \int \varphi_{t_n} d(\mu_1-\mu_0) = \int \varphi_0 d(\mu_1-\mu_0).
\label{eq: upper limit}
\end{equation}
Pick any subsequence $n'$ of $n$. From (\ref{EQ:EOT_smooth_potentials}) and the Ascoli-Arzela theorem, there exists a further subsequence $n''$ such that $\varphi_{t_{n''}} \to \varphi$ and $\psi_{t_{n''}} \to \psi$ locally uniformly for some (continuous) functions $\varphi,\psi$. Again, from (\ref{EQ:EOT_smooth_potentials}) and the dominated convergence theorem, $(\varphi,\psi)$ satisfies (\ref{eq: optimality}) for every $(x,y) \in \R^d \times \R^d$, so that they are optimal EOT potentials for $(\mu_0,\nu)$. We shall now verify that $\varphi(x) = \varphi_0 (x)+ a$ for every $x \in \R^d$ for some constant $a \in \R$. To see this, by uniqueness of optimal EOT potentials, $\psi (y) =  \psi_0(y) - a$ for $\nu$-almost every $y \in \R^d$ for some constant $a \in \R$. We then have
\[
\varphi (x) = -\log \int_{\R^d} e^{\psi(y) - c(x,y)} d\nu(y) =-\log \int_{\R^d} e^{\psi_0(y)-c(x,y)} d\nu(y) + a =\varphi_0 (x) + a
\]
for every $x \in \R^d$. Conclude that, by (\ref{EQ:EOT_smooth_potentials}) and the dominated convergence theorem, 
\[
\lim_{n'' \to \infty} \int \varphi_{t_{n''}} d(\mu_1-\mu_0) = \int \varphi d(\mu_1-\mu_0) = \int \varphi_0 d(\mu_1-\mu_0).
\]
Since the limit does not depend on the choice of subsequence, we have proved (\ref{eq: upper limit}), completing the proof of (\ref{eq: EOT derivative}). 

Next, we prove (\ref{eq: EOT derivative}). From the inequalities (\ref{eq: one-sided-1}) and (\ref{eq: one-sided-2}), we have 
\[
| \sS(\mu_1,\nu) - \sS(\mu_0,\nu) | \le \int \varphi_0 d(\mu_1-\mu_0) \bigvee \int \varphi_1 d(\mu_1-\mu_0).
\]
In view of the estimates (\ref{EQ:EOT_smooth_potentials}), we see that the right-hand side can be bounded by $\| \mu_1 - \mu_0 \|_{\infty,\calF_\sigma}$. 
\end{proof}

Second, we will verify that the function class $\calF_\sigma$ is $\mu$-Donsker. 

\begin{lemma}
\label{lem: EOT Donsker}
If $\mu$ is sub-Gaussian, then $\calF_{\sigma}$ is $\mu$-Donsker. 
\end{lemma}

The proof relies on the the following lemma, which is a simple application of Theorem 1 in \cite{vanderVaart1996}.

\begin{lemma}[Lemma 8 in \cite{nietert21}]
\label{lem:Donsker}
Let $\cF \subset C^{s}(\R^{d})$ be a function class where $s$ is a positive integer with $s > d/2$, and let $\{ \cX_{j} \}_{j=1}^{\infty}$ be a cover of $\R^{d}$ consisting of nonempty bounded convex sets with bounded diameter. 
Set $M_j = \sup_{f \in \cF} \| f \|_{C^{s}(\cX_j)}$ with $\| f \|_{C^{s}(\cX_j)} = \max_{\bar{k} \le s} \sup_{x \in \mathrm{int}(\cX_j)} |\partial^{k}f(x)|$. 
If $\sum_{j=1}^{\infty} M_{j}\mu(\cX_j)^{1/2} < \infty$, then $\cF$ is $\mu$-Donsker.
\end{lemma}

\begin{proof}[Proof of Lemma \ref{lem: EOT Donsker}]
We construct a cover $\{ \cX_j \}_{j=1}^{\infty}$ and verify the conditions of Lemma~\ref{lem:Donsker} for the function class $\cF_{\sigma}$. Let $B_r = \{ x : \|x\| \le r \}$. For $r =2,3,\dots$, let $\{ x_{1}^{(r)},\dots,x_{N_{r}}^{(r)} \}$ be a minimal $1$-net of $B_r \setminus B_{r-1}$.
Set $x_{1}^{(1)} = 0$ with $N_1 = 1$. From a simple volumetric argument, we see that $N_r = O(r^{d-1})$.
Set $\calX_j = \{ x : \|x-x_{j}^{(r)} \| \le 1 \}$ for $j=\sum_{\ell=1}^{r-1}N_{\ell} +1,\dots,\sum_{\ell=1}^rN_\ell$. By construction, $\{ \calX_j \}_{j=1}^{\infty}$ forms a cover of $\R^{d}$ with diameter $2$. 
Further, by construction, for $M_j := \sup_{f \in \cF_{\sigma}} \| f \|_{C^{s}(\cX_j)}$ with $s = \max \{ \lfloor d/2 \rfloor + 1,2 \}$, we have $\max_{\sum_{\ell=1}^{r-1}N_{\ell} +1\le j \le \sum_{\ell=1}^r N_{\ell}} M_j\lesssim r^{s}$. The function class $\calF_{\sigma}$ is $\mu$-Donsker if $
\sum_{r=1}^\infty r^{s+d-1} \mu(B_{r-1}^c)^{1/2} < \infty$,
which holds as $\mu$ is sub-Gaussian. 
\end{proof}

We are ready to prove Theorem \ref{THM:EOT_CLT}.

\begin{proof}[Proof of Theorem \ref{THM:EOT_CLT}]
The proof for the bootstrap consistency for each case of (i) and (ii) is analogous to that of $\SWp$ in Theorem \ref{thm: limit distribution SWp}, so we only prove the CLT results.

\textbf{Part (i).}  
Suppose $\mu$ is $\sigma^2$-sub-Gaussian. 
Observe that, for any constant $\bar{\sigma} > \sigma$, $\hat{\mu}_n$ is $\bar{\sigma}^2$-sub-Gaussian with probability approaching one. Indeed, by the law of large numbers,
\[
\E_{\hat{\mu}_n}[e^{\|X\|^2/(2d\bar{\sigma}^2)}] = \frac{1}{n} \sum_{i=1}^n e^{\|X_i\|^2/(2d\bar{\sigma}^2)} \to \E_\mu[e^{\|X\|^2/(2d\bar{\sigma}^2)}] \le 2^{\bar{\sigma}^2/\sigma^2} < 2
\]
a.s., so $\E_{\hat{\mu}_n}[e^{\|X\|^2/(2d\bar{\sigma}^2)}] \le 2$ with probability approaching one.

With this in mind, for any fixed $\bar{\sigma} > \sigma$, we apply Proposition \ref{prop: master proposition} with $\calP_0 = \{ \rho \in \calP(\R^d) : \text{$\rho$ is $\bar{\sigma}^2$-sub-Gaussian} \}$, $\calF = \calF_{\bar{\sigma}}, F(x) = C_{s,d,\bar{\sigma}}(1+\|x\|^s)$ with $s = \max \{ \lfloor d/2 \rfloor +1,2 \}$, and $\delta(\rho) = \sS (\rho,\nu)$. 
We have already verified Conditions (a)--(c) in Proposition \ref{prop: master proposition} with $\delta_{\mu}'(m) = m(\varphi)$. Conclude that
\[
\sqrt{n}\big(\sS(\hat{\mu}_n,\nu) - \sS(\mu,\nu)\big) = \sqrt{n}\big(\delta (\hat{\mu}_n) - \delta(\mu)\big) \stackrel{d}{\to} \delta_{\mu}'(G_\mu) = G_\mu(\varphi) \sim N\big(0,\Var_{\mu}(\varphi)\big). 
\]

As the derivative $\delta_{\mu}'$ is the point evaluation at $\varphi$, 
the fact that $\Var_{\mu}(\varphi)$ coincides with the semiparametric efficiency bound follows directly from \cref{prop: master proposition 2} (note: since $\mu$ is $\sigma^2$-sub-Gaussian and $\bar{\sigma} > \sigma$, for any bounded $\mu$-mean zero function $h$, $(1+th)\mu$ is $\bar{\sigma}^2$-sub-Gaussian for sufficiently small $t > 0$).

\medskip

\textbf{Part (ii)}. Suppose $\mu,\nu$ are $\sigma^2$-sub-Gaussian. Set $s=\max \{ \lfloor d/2 \rfloor+1,2 \}$.
Let
\[
\calF_\sigma^{\oplus} = \{ \varphi \oplus \psi  : (\varphi,\psi) \ \text{satisfies (\ref{EQ:EOT_smooth_potentials})} \}. 
\]
We will show that, if $\mu_i,\nu_i, i=0,1$ are $\sigma^2$-sub-Gaussian, then 
\begin{equation}
|\sS(\mu_1,\nu_1) - \sS(\mu_0,\nu_0) | \le \| \mu_1 \otimes \nu_1 - \mu_0 \otimes \nu_0 \|_{\infty,\calF_\sigma^{\oplus}}
\label{eq: lipschitz EOT}
\end{equation}
 Let $(\varphi_{i,j},\psi_{i,j})$ be optimal EOT potentials for $(\mu_i,\nu_j)$ satisfying \eqref{EQ:EOT_smooth_potentials}. Then, from \eqref{eq: one-sided-1}, we have
\[
\begin{split}
\sS(\mu_1,\nu_1) - \sS(\mu_1,\nu_1) &= \sS(\mu_1,\nu_1) - \sS(\mu_0,\nu_1) + \sS(\mu_0,\nu_1) - \sS(\mu_0,\nu_0) \\
&\ge \int \varphi_{0,1} d(\mu_1-\mu_0) + \int \psi_{0,0} d(\nu_1-\nu_0) \\
&= \int (\varphi_{0,1} \oplus \psi_{0,0}) d\big (\mu_1 \otimes \nu_1 - \mu_0 \otimes \nu_0 \big). 
\end{split}
\]
Likewise, using \eqref{eq: one-sided-2}, we see that 
\[
\sS(\mu_1,\nu_1) - \sS(\mu_1,\nu_1) \le
\int (\varphi_{1,1} \oplus \psi_{0,1}) d\big (\mu_1 \otimes \nu_1 - \mu_0 \otimes \nu_0 \big). 
\]
Since all $(\varphi_{i,j},\psi_{i,j})$ satisfy \eqref{EQ:EOT_smooth_potentials}, we obtain \eqref{eq: lipschitz EOT}. 

Further, arguing as in the proof of Lemma \ref{lem: EOT gateaux}, we have
\[
\begin{split}
&\lim_{t \downarrow 0}\frac{\sS (\mu_0 + t(\mu_1-\mu_0),\nu_0+t(\nu_1-\nu_0)) - \sS(\mu_0,\nu_0)}{t}\\
&\quad= \int (\varphi_{0,0} \oplus \psi_{0,0}) d\big(\mu_1 \otimes \nu_1 - \mu_0 \otimes \nu_0 \big).
\end{split}
\]

We shall verify that 
\[
\sqrt{n}(\hat{\mu}_n \otimes \hat{\nu}_n - \mu \otimes \nu) \stackrel{d}{\to}  G_{\mu \otimes \nu} \quad \text{in} \ \ \ell^\infty (\calF_\sigma^{\oplus})
\]
for some tight Gaussian process $G_{\mu \otimes \nu}$ in $\ell^\infty (\calF_\sigma^{\oplus})$. But this follows from the same argument as in the proof of Theorem \ref{thm: limit distribution SWp} Part (ii); we omit the details for brevity. 

Pick and fix any $\bar{\sigma} > \sigma$. Then $\hat{\mu}_n$ and $\hat{\nu}_n$ are $\bar{\sigma}^2$-sub-Gaussian with probability approaching one. We shall now apply Proposition \ref{prop: master proposition} with $\calP_0 = \{ \rho_1 \otimes \rho_2 \in \calP(\R^{2d}) : \text{$\rho_1,\rho_2$ are $\bar{\sigma}^2$-sub-Gaussian} \}$, 
$\calF = \calF_{\bar{\sigma}}^{\oplus}, F(x,y) = C_{s,d,\bar{\sigma}}(2+\|x\|^s+\|y\|^s)$ with $s=\max \{ \lfloor d/2 \rfloor +1, 2 \}$, and $\delta(\rho_1 \otimes \rho_2) = \sS (\rho_1,\rho_2)$. As in the proof of \cref{thm: limit distribution SWp} Part (ii), we may identify $(1-t) (\mu_0 \otimes \nu_0) + t (\mu_1 \otimes \nu)$ and $\big((1-t)\mu_0 + t\mu_1\big) \otimes \big((1-t)\nu_0 + t\nu_1\big)$ as elements of $\ell^\infty(\calF)$, and the class $ \calP_0$ is convex as a subset of $\ell^\infty (\calF)$.  We have already verified Conditions (a)--(c) in Proposition \ref{prop: master proposition} with $\delta_{\mu \otimes \nu}'(m) = m(\varphi\oplus \psi)$ (see also \cref{rem: convexity}). Conclude that 
\[
\begin{split}
\sqrt{n}\big(\sS(\hat{\mu}_n,\hat{\nu}_n) - \sS(\mu,\nu)\big) &= \sqrt{n}\big(\delta (\hat{\mu}_n\otimes \hat{\nu}_n) - \delta(\mu \otimes \nu)\big) \\
&\stackrel{d}{\to} G_{\mu \otimes \nu}(\varphi\oplus \psi) \sim N(0,\Var_{\mu}(\varphi) + \Var_{\nu}(\psi)). 
\end{split}
\]
Finally, the fact that $\Var_{\mu}(\varphi)+\Var_\nu (\psi)$ coincides with the semiparametric efficiency bound follows directly from \cref{cor: efficiency two sample}.
\end{proof}

\bibliographystyle{amsalpha}
\bibliography{ref}
\end{document}